\documentclass[10pt,oneside,a4paper,reqno]{amsart} 
\usepackage{mathrsfs}

\usepackage{amssymb}
\usepackage[colorlinks=true, linkcolor=magenta, citecolor=cyan, urlcolor=blue]{hyperref}
\usepackage{bbm}

\numberwithin{equation}{section}\newtheorem{theorem}{Theorem}[section]
\newtheorem{lemma}[theorem]{Lemma}

\newtheorem{proposition}[theorem]{Proposition}\theoremstyle{remark}
\newtheorem{remark}{Remark}[section]
\theoremstyle{definition}
\newtheorem{definition}{Definition}[section]

\DeclareMathOperator*{\esssup}{ess\,sup}

\newcommand{\Rtre}{\mathbb{R}^{3}}
\newcommand{\RT}{\mathbb{R}^{3}}

\newcommand{\curl}{\text{curl}}

\begin{document}

\title{Stability properties of the regular set for the Navier--Stokes equation}

\author{Piero D'Ancona}
\address{Piero D'Ancona: 
SAPIENZA --- Universit\`a di Roma,
Dipartimento di Matematica, 
Piazzale A.~Moro 2, I-00185 Roma, Italy}
\email{dancona@mat.uniroma1.it}

\author{Renato Luc\`a}
\address{Renato Luc\`a: 
Departement Matematik und Informatik,
            Universit\"at Basel,
            Spiegelgasse 1,
            4051 Basel,
            Switzerland.}
\email{renato.luca@unibas.ch}


\maketitle

\begin{abstract}
  We investigate the size of the regular set for small perturbations of 
  some classes of strong large solutions to the Navier--Stokes equation.
  We consider perturbations of the data that are small in suitable
  weighted $L^{2}$ spaces but can be arbitrarily large in any translation 
  invariant Banach space.
  We give similar results in the small data setting.
\end{abstract}

\section{Introduction and main results}

We consider the 3D Navier--Stokes equation on $(t,x) \in (0,\infty) \times \Rtre$:
\begin{equation}\label{CauchyNS}
\left \{
\begin{array}{rcl}
\partial_{t}u + (u \cdot \nabla) u  - \Delta u & = & -\nabla P \\ 
\nabla \cdot u & = & 0 \\ 
u\big|_{t=0} & = & u_{0} \, , 
\end{array}\right.
\end{equation}
describing the free motion of a viscous incompressible fluid with velocity 
$u$ and pressure~$P$. 
For simplicity, we have settled the kinematic viscosity equal to one and 
we will use the same notation for the norm of scalar,
vector or tensor quantities, thus we write
\begin{equation*}
  \textstyle
  \| P \|_{L^{2}}^{2}:= \int P^{2}dx,
  \qquad
  \|u\|_{L^{2}}^{2}:=\sum_{j}\|u_{j}\|_{L^{2}}^{2},
  \qquad
  \|\nabla u\|_{L^{2}}^{2}:=
  \sum_{j,k}\|\partial_{k} u_{j}\|_{L^{2}}^{2} \, ,
\end{equation*}
and $L^{2}(\mathbb{R}^{3})$
instead of $[L^{2}(\mathbb{R}^{3})]^{3}$,
{\it etc}. 

Global
weak solutions to \eqref{CauchyNS} are known to exist for any divergence free initial velocity in $L^{2}$, since the pioneering work of 
Leray \cite{Ler}. 
The uniqueness and the
persistence of regularity of Leray's solutions
are long standing open problems.

On the other hand, several
global well-posedness results have been proved for  
small initial data and for large data with suitable symmetry. 
We will work in this simpler framework.

Earliest results in the  
small data setting go back to Fujita and Kato~\cite{FKato} for~$u_0 \in \dot{H}^{1/2}$ 
and to Kato \cite{Kato} for $u_0 \in L^{3}$.
Then several functional spaces have been considered like Morrey spaces
\cite{Giga2}, \cite{Tay}, \cite{Kato4}, \cite{Fed}, \cite{Kozono}, 
Besov spaces 
\cite{Cannon}, \cite{Planc}, and others.
This approach culminated in the $BMO^{-1}$ well-posedness result of Koch and Tataru~\cite{Tat}, that is the most general one in this direction.

In the large data framework, global solutions have
been constructed  
imposing some additional symmetry on the data, e.g.,
$u_0$ axisymmetric, helical or two dimensional.   
Axisymmetric data have been studied by 
Ukhovskii and Iudovich \cite{Yud} and Ladyzhenskaya \cite{Lady}
under a zero swirl assumption (see also \cite{Leon}, \cite{Chae};
we recall that the \emph{swirl} is the angular component of the velocity
field with respect to the axis of symmetry).
The interesting case of non zero swirl is still open. 
Additional results on axisymmetric solutions are   
\cite{GIM}, \cite{Gall1}. Helical data, which means invariant under the composition
of a rotation and a translation along a fixed direction, have been considered in \cite{Mah}. 
Further interesting large data results are
\cite{Foias}, \cite{Gall2}, \cite{Iftim}, \cite{ChemYvGall}, \cite{ChemGall}.

Once a large solution, with good properties, has been constructed (with or without symmetry) 
a natural and important question concerns
its stability for small perturbations of the data.
This \emph{perturbative} approach
was followed in a systematic way
by Ponce, Racke, Sideris and Titi \cite{Ponce},    
for small $H^{1}$ perturbations.
As for the small data problem, this was then extended to small perturbations in weaker norms,
for instance $L^{3}$ \cite{Kaw}, 
Besov spaces \cite{GIP1}
and $BMO^{-1}$ \cite{Auch}.

In both small data and perturbative results 
it is customary to consider functional spaces that are
scaling and translation invariant. More precisely, since 
the Navier--Stokes equation is invariant under the family
of symmetries
\begin{equation}\label{CTINV}
u \mapsto
\lambda u (\lambda^{2} t, \lambda (x-\bar{x})), \qquad
\lambda \in (0,\infty),\ \ \bar{x} \in \Rtre \, , 
\end{equation}
it is natural consider initial data in Banach spaces invariant under
\begin{equation}\label{eq:symdata}
u_{0}(x) 
\mapsto \lambda u_{0} (\lambda (x-\bar{x})), \qquad
\lambda \in (0,\infty),\ \ \bar{x} \in \Rtre \, .
\end{equation}
On the other hand, 
if we are only interested in certain local regularity properties,
it may suffices to require invariance with respect to the scaling but not necessarily under translations: that is to say, the norm of the data 
is invariant
under \eqref{eq:symdata} for all $\lambda>0$ and a
fixed $\bar{x}$.
In the classical work
\cite{CKN} the authors prove, among other,
the smoothness, in time-increasing neighborhoods
of a point~$\bar{x} \in \Rtre$, of 
weak solutions
with initial data $u_{0}$ such that 
\begin{equation}\nonumber
 \int_{\Rtre}  |u_{0}(x)|^{2} |x-\bar{x}|^{-1} dx  \ll 1 \, .
\end{equation}
Namely, the weighted $L^{2}$ norm has to be sufficiently small if we center an homogeneous weight
of degree $-1$ at the point $\bar{x}$;
see Theorem~\ref{CKNSmallData} below for the precise statement.
The aim of this paper is to
give some extensions and improvements 
of this result, in both the perturbative and small data frameworks.

First of all, we 
recall a classical notion of regularity for weak solutions to the Navier--Stokes equation.

\begin{definition}\label{Def:Regulaity}
  Let $u(t,x)$ be a weak solution of~(\ref{CauchyNS}). 
  A point~$(t,x) \in (0,\infty) \times \mathbb{R}^{3}$ is 
  \emph{regular}
  if $u$ is bounded on a neighborhood of $(t,x)$. 
  In particular, this implies that $u$ is smooth, in the space variables, in a neighborhood of
  $(t,x)$; see \cite{Ser}. 
  A subset of $(0,\infty) \times \mathbb{R}^{3}$ 
  is \emph{regular} if all its points are regular.
\end{definition}

\begin{definition}
We write
\begin{equation}\nonumber
  \Pi_{\alpha, \bar{x}}:=
  \left\{ (t,x)\in (0,\infty) \times \mathbb{R}^{3}
  \ \colon\ 
  t > \frac{|x- \bar{x}|^{2}}{\alpha} 
  \right\}
\end{equation}
for the region above the paraboloid of aperture $\alpha$ 
in the upper half space
$(0,\infty)\times \mathbb{R}^{3}$, 
with vertex at $(t,x) = (0, \bar{x})$.
When $\bar{x} = 0$ we also write $\Pi_\alpha$, instead of $\Pi_{\alpha, 0}$.
Note that 
these sets are increasing with $\alpha$.
\end{definition}

The following statement \cite[Theorem D]{CKN}
applies to   
\emph{suitable} weak solutions; we refer to Section~\ref{sec:prelim} for the definition of suitability.
In particular, the weak solutions given by the Leray 
approximation procedure are suitable \cite[Theorem 2.3]{S2}.
We use the notation $L^{p}_{t}$, $L^{p}_{x}$, $\dot{H}^{k}_{x}$ when the integration is over the time~$t \in (0,\infty)$
and space~$x \in \Rtre$ variables, respectively. We write
$\| f \|_{XY} := \big\| \| f \|_{Y} \big\|_{X}$ for nested norms and $XY$ for the associated normed spaces.

\begin{theorem}[Caffarelli--Kohn--Nirenberg]\label{CKNSmallData}
  There exists a constant $\varepsilon_{0} > 0$ such that
  the following holds.  
  The set
  \begin{equation}\nonumber
    \Pi_{\varepsilon_{0}-\varepsilon, \bar{x}}
    :=
    \left\{ (t,x) \ : \ t > 
    \frac{|x-\bar{x}|^{2}}{\varepsilon_{0} - \varepsilon} \right\}
  \end{equation} 
  is regular for any suitable weak solution
  $u \in L^{\infty}_{t}L^{2}_{x} \cap L^{2}_{t}\dot{H}^{1}_{x}$ to the Navier--Stokes equation
  with divergence free initial data
  $u_{0} \in L^{2}$
  such that
  \begin{equation}\label{LSA}
    \| |x-\bar{x}|^{-1/2} u_{0} \|_{L^{2}} ^{2} =: \varepsilon < 
    \varepsilon_{0} \, .
    \end{equation}
\end{theorem}

In other words, if the weighted $L^{2}$ norm of $u_0$ is small enough, once we centre the weight in $\bar{x}$, then the solution
is smooth in the region above a space-time paraboloid with vertex at $(t,x) = (0, \bar{x})$.
The existence of 
suitable weak solutions in $L^{\infty}_{t}L^{2}_{x} \cap L^{2}_{t}\dot{H}^{1}_{x}$ is ensured by 
the Leray theory, see~\cite[Theorem 2.3]{S2}, for any divergence free $u_{0} \in L^{2}$.

It is important to stress out that condition
(\ref{LSA}) allows $u_0$ to be large at $x$ sufficiently far from $\bar{x}$. Thus (\ref{LSA}) 
is not comparable with any 
translation invariant smallness 
assumption on $u_0$. This is quantified in the following remark.

\begin{remark}\label{Rem1}
  There are data $u_{0}$ arbitrarily large in 
  $\dot{B}^{-1}_{\infty,\infty}$, or any translation invariant Banach space,
   but such that the norms 
  $\| |x-\bar{x}|^{-1/2} u_{0} \|_{L^{2}}$ are arbitrarily small. 
  We recall that 
  $\dot{B}^{-1}_{\infty,\infty}$
  contains any Banach space invariant under~(\ref{eq:symdata}); see~\cite{CannoneHand}.  
  
  Indeed, assume for simplicity $\bar{x} =0$ and let
  $\phi\in C^{\infty}_{c}(\mathbb{R}^{3})$ be a divergence free vector field. Letting $\phi_{K}(x):=\phi(x-K \xi)$ for the
  translate of $\phi$ 
  by the vector $\xi K$, with $|\xi|=1$ and $K\gg1$,
  since\footnote{As customary, we write $A \lesssim B$ if $A \leq C B$ for a certain constant $C>0$ and 
  $A \simeq B$ if $A \lesssim B$ and $B \lesssim A$.}
  \begin{equation}\nonumber
    \||x|^{-1/2} \phi_{K}\|
    _{L^{2}}
    \simeq
    K^{-1/2} \, ,
  \end{equation}
 recalling the translation invariance of $\dot{B}^{-1}_{\infty,\infty}$, we get as $K \to \infty$:
  \begin{equation}\nonumber
    \| |x|^{-1/2} \phi_{K} \|_{L^{2}(\Rtre)} \to 0
    \ \ \ \text{and}\ \ \ 
    \|\phi_{K}\|_{\dot{B}^{-1}_{\infty,\infty}} = const \, .
  \end{equation}
\end{remark}

Our main result is
a perturbative version of Theorem \ref{CKNSmallData}.
We need some preliminary definitions.

\begin{definition}\label{Def:Admissibility}  
A couple $(r,q)$ is {\it admissible} if $2 \leq r < \infty$ and $2/r + 3/q =1$.
\end{definition}


\begin{definition}\label{Def:RefSol}
A weak solution
$w \in L^{\infty}_{t}L^{2}_{x} \cap L^{2}_{t}\dot{H}^{1}_{x}$ to the Navier--Stokes equation 
with divergence free initial 
data $w_{0} \in L^{2}$ 
is a \emph{reference solution of size $\mathcal{K}$} if 
\begin{equation}\label{Def:KRefProp}
\int_{0}^{\infty} \left( \int_{\Rtre} |w(t,x)|^{q}dx \right)^{\frac{r}{q}} dt =: 
\| w \|^{r}_{L^{r}_{t}L^{q}_{x}}
=:
\mathcal{K} < \infty \, ,
\end{equation}
for an admissible couple $(r,q)$. We furthermore assume $w_{0} \in L^{2}(|x-x'|^{-1}dx)$, for some $x' \in \Rtre$. 
\end{definition}

As well known, the boundedness assumption (\ref{Def:KRefProp}) ensures the smoothness and 
uniqueness of 
reference solutions; see Remark~\ref{Rem:Uniqueness} for more details.

We are now ready to state the main result of the paper.

\begin{theorem}\label{MainThmPerturbation}
Given $(r,q)$ admissible, there exists a constant $\delta_{0} > 0$ such that the following holds.
Let~$w$ be a reference solution 
of size~$\mathcal{K}$ to the Navier--Stokes equation with divergence 
free initial data $w_{0}$. 
Then the set
\begin{equation}\label{ImprFinalm}
    \Pi_{\delta_{0},\bar{x}}
    :=
    \left\{ (t,x) \ : \ t > 
    \frac{|x-\bar{x}|^{2}}{\delta_{0}} \right\}
  \end{equation}
is regular
for every suitable weak solution $u \in L^{\infty}_{t}L^{2}_{x} \cap L^{2}_{t}\dot{H}^{1}_{x}$ 
to the Navier--Stokes equation with divergence free data $u_{0} \in L^{2}$ such that
\begin{equation}\label{MTSA}
    \| |x-\bar{x}|^{-1/2} (u_{0} - w_{0})\|_{L^{2}} \le \delta_{0} e^{-\mathcal{K}/\delta_{0}} \, .
\end{equation}

\end{theorem}


Thus,
if we perturb, in the weighted $L^2$ norm, the data
of a reference solution, the corresponding weak solutions
are still regular in the region above a paraboloid
with vertex at $(t,x)=(0,\bar{x})$, where
$\bar{x}$ is the center of the weight. 
We insist that the smallness assumption (\ref{MTSA}) 
allows for large perturbations, far from the point~$\bar{x}$;
see Remark \ref{Rem1}.

In Section~\ref{Sec:PertSol}
we give a few applications
of Theorem \ref{MainThmPerturbation} to some classes
of reference solutions.
In Proposition \ref{prop:Axysimm},
we consider large axisymmetric data $w_{0}$    
with zero swirl.
In Proposition \ref{prop:CGPert}, 
the reference data $w_{0}$ are assumed to fit
a nonlinear smallness assumption (see (\ref{GallChemSmallnessCond}))  
that can be satisfied by arbitrarily large $w_0 \in \dot{B}^{-1}_{\infty,\infty}$, thus escaping the 
hypothesis of the known
small data results. 
These solutions have been studied in \cite{ChemGall}. 

Further applications requires a
slightly more general version of Theorem~\ref{MainThmPerturbation}, since we also want to consider 
reference solutions with infinite energy. This will be allowed by Theorem~\ref{MainThmPerturbationMoreGen},
that generalizes~Theorem~\ref{MainThmPerturbation}.
Thus, in Proposition~\ref{Prop:2DNew}, we can handle small 3D (three dimensional) perturbations of large 2D
initial data~$W_{0} = (W_{0,1}, W_{0,2}) \in L^{2} \cap L^{1}$, that we can consider 3D objects via the natural extension   
\begin{equation}\label{def:3dExtensVect}
\widetilde{W_{0}}(x_{1},x_{2},x_{3}) := (W_{0,1}(x_{1},x_{2}), W_{0,2}(x_{1},x_{2}), 0) \, . 
\end{equation}
In Proposition~\ref{Prop:Beltr}, we focus on Beltrami fields, namely initial data~$w_{0}$ that are eigenvectors 
of the curl operator $\nabla \times w_{0} = \lambda w_{0}$ on~$\Rtre$, with $\lambda \neq 0$. 
These vector fields give rise to explicit solutions to the Navier--Stokes equation with very rich topological structures. 

In the second part of the paper we work in the small data setting.
We
investigate how the size of the regular set 
depends on the size of the initial data. 
Notice that the regular set $\Pi_{\varepsilon_0 - \varepsilon, \bar{x}}$ considered in Theorem \ref{CKNSmallData}
converges to a maximal one: 
$$
\Pi_{\varepsilon_{0} - \varepsilon, \bar{x}} \to \Pi_{\varepsilon_0, \bar{x}} 
\quad
\mbox{as}
\quad
\varepsilon \rightarrow 0 \, .
$$ 
In other words, we are not able to prove smoothness in the region below the 
limit paraboloid
$\Pi_{\varepsilon_0, \bar{x}}$, even if the size~$\varepsilon$ of the initial data
is arbitrarily small.

In Theorem~\ref{Theorem:DL} we will prove that
if the size $\mathcal{K}$ of the reference solution
is smaller than a threshold value,
then the regular set is actually larger and  
invades the whole half space $\{t > 0\}$ in the limit 
$$
\| |x-\bar{x}|^{-1/2} (u_{0} - w_{0})\|_{L^{2}} \rightarrow 0 \, , 
$$
improving in this way Theorem
\ref{MainThmPerturbation} in the case of very small reference solutions.

Using~Theorem \ref{Theorem:DL}, we are also able
to cover the gap between the regularity~Theorem \ref{CKNSmallData}
and 
the Kato well-posedness theory, which works for small~$L^{3}$ initial data. 
To do so, we consider 
$u_0$ such that the critical\footnote{
Notice that these norms are invariant under the natural scaling~\eqref{eq:symdata}, with $\bar{x}$ fixed, that is why we refer to them as {\it critical}.
} weighted $L^{p}$ norms
\begin{equation}\nonumber
\| |x-\bar{x}|^{\alpha} u_{0} \|_{L^{p}}, 
\quad 
2<p<3,
\quad
\alpha = 1 - 3/p \, ,
\end{equation} 
are sufficiently small. Then we
show that the size of the regular set 
improves
as $p$ increases. We recover full regularity in the limit $p \rightarrow 3^{-}$ (namely $\alpha \rightarrow 0^{-}$), as expected
in light of the $L^{3}$ well-posedness.

Noting that $\alpha <0$
when $p < 3$, the same argument of Remark~\ref{Rem1} allows to construct
initial data that are arbitrarily large in any translation invariant Banach space
but arbitrarily small in~$L^{p}(|x-\bar{x}|^{\alpha p}dx)$, so that also the following theorem is not
implied by the various known small data results. 
Let
\begin{equation}\label{def:theta12}
\theta_{1}(p) := \left( \frac{p-2}{3-p}  \right)^{1-p/3},
\qquad
\theta_{2}(p) := \left( \frac{p-2}{3-p}  \right)^{1-p/2} \, .
\end{equation}
It is easy to check that
\begin{equation}\nonumber
\lim_{p \rightarrow 2^{+}}  \theta_{1}(p) = 0, 
\quad
\lim_{p \rightarrow 3^{-}}  \theta_{1}(p) = 1\, ,  
\end{equation}
while $\theta_2$ behaves in the opposite way
\begin{equation}\nonumber
\lim_{p \rightarrow 2^{+}}  \theta_{2}(p) = 1, 
\quad
\lim_{p \rightarrow 3^{-}}  \theta_{2}(p) = 0 \, .
\end{equation}

\begin{theorem}\label{thm:Gap} 
  There exists a constant $\delta_{1}>0$ such that the following holds. 
  Let $2 < p < 3$, $\alpha = 1- 3/p$
  and $M > 1$. The set $\Pi_{M\delta_{1}, \bar{x}}$ is regular for any suitable weak solution
  $u \in L^{\infty}_{t}L^{2}_{x} \cap L^{2}_{t}\dot{H}^{1}_{x}$ to the Navier--Stokes equation 
  with divergence free initial data $u_0\in L^{2} \cap L^{2}(|x-x'|^{-1}dx)$, for some $x' \in \Rtre$, such that  
   \begin{equation}\label{ILSAGAP}
   \theta_1  \|  |x-\bar{x}|^{\alpha} u_0 \|_{L^{p}}^{p/3} \le \delta_{1},
  \quad
    \theta_2 \|  |x-\bar{x}|^{\alpha} u_0 \|_{L^{p}}^{p/2}  \le \delta_{1} e^{-M^{2}/\delta_{1}} \, .
  \end{equation}
\end{theorem}

This can be interpreted in the following way. Since
$\theta_{2}(p)\to0$ as $p\to3^{-}$,
we can
choose $p=p_{M}$ as a function
of $M$ in such a way that
\begin{equation}\nonumber
  e^{M^{2}/\delta_{1}}\cdot\theta_{2}(p_{M})\to0
  \quad \mbox{as} \quad 
  M \to \infty \qquad (p_{M} \to 3^{-}) \, ,
\end{equation}
so that, since $\theta_1(p_{M}) \rightarrow 1$, we have proved: 
\begin{equation}\nonumber
   \| |x-\bar{x}|^{\alpha} u_{0}\|_{L^{p_{M}}} \le\delta_{1}/2,
  \quad\Rightarrow\quad
  \Pi_{M \delta_{1}, \bar{x}}
  \ \mbox{is a regular set for $u$} \, ,
\end{equation}
for all sufficiently large $M$.
Notice that $\Pi_{M \delta_{1}, \bar{x}}$ increases indefinitely as $M \to \infty$.
In other words, if
$M\to\infty$ and  
the~$L^{p_{M}}(|x-\bar{x}|^{\alpha p_{M}})$ norm of~$u_{0}$ is less than~$\delta_{1}/2$,
the regular set invades the whole half space $\{ t>0 \}$ in the limit. 

It is worth noting that all the results presented in the paper rely upon 
the algebraic structure of the nonlinearity $N(u) := (u \cdot \nabla) u$.   
Indeed, like in other perturbative results \cite{Ponce}, \cite{Kaw}, \cite{GIP1}, \cite{Auch}, we exploit the cancellation~$\int_{\Rtre} N(u) \cdot u = 0$. The novelty here is that also the 
behavior of $N(u)$ under the change of variables
$$
 (t, y) = (t, x - \xi t),
  \qquad
  u_{\xi}(t,y) = u(t,x),
  \qquad \xi \in \Rtre,
$$ 
namely
$$
N(u_{\xi}) = (u_{\xi} \cdot \nabla)  u_{\xi} -  (\xi \cdot \nabla) u_{\xi} \, ,
$$
plays a key role.

The rest of the paper is organized as follows.
In the next section we state some applications of this perturbative theory, namely 
Propositions~\ref{prop:Axysimm}\,-\,\ref{Prop:Beltr}.   
In Section~\ref{sec:prelim} we fix our setting,
recalling the definition of suitable solutions and
the fundamental Caffarelli--Kohn--Nirenberg regularity criterion from~\cite{CKN}. 
In Section \ref{sec:proof} we prove our main Theorem \ref{MainThmPerturbation}, in fact we prove the 
slightly more general Theorem~\ref{MainThmPerturbationMoreGen}. 
As a consequence, in Section~\ref{Sec:ApplicProof}, we 
deduce Propositions~\ref{prop:Axysimm}\,-\,\ref{Prop:Beltr}. 

Section~\ref{Sect:SmallData} is devoted to the small data theory. In Theorem~\ref{Theorem:DL} we improve the perturbative result for 
reference solutions of sufficiently small size~$\mathcal{K}$. This also allows to prove Theorem~\ref{thm:Gap}.


\section{Applications: perturbative solutions}\label{Sec:PertSol}
We consider solutions with bounded energy in Propositions~\ref{prop:Axysimm},~\ref{prop:CGPert} 
and unbounded energy in Propositions~\ref{Prop:2DNew},~\ref{Prop:Beltr}. 
\subsection{Axisymmetric solutions with zero swirl}\label{Subsection:Axysimmetric}

Let $\{ (\mathbf{r}  \cos  \Theta, \mathbf{r}  \sin  \Theta, x_{3}), \Theta \in \mathbb{T} := \mathbb{R} / 2\pi \mathbb{Z}, \mathbf{r} \in [0,\infty), 
x_{3} \in (-\infty,\infty) \}$  
cylindrical polar coordinates on $\mathbb{R}^{3}$.
We say that a vector field $f$ is axisymmetric (with respect to the $x_3$-axis)
if it is independent of $\Theta$, namely
\begin{equation}\label{Def:AxySimmVF}
f = f_{e_{\mathbf{r}}}(\mathbf{r},x_{3}) e_{\mathbf{r}} + f_{e_{\Theta}}(\mathbf{r},x_{3}) e_{\Theta} + f_{e_{x_{3}}}(\mathbf{r},x_{3})e_{x_{3}}.
\end{equation} 
Here
$(e_{\mathbf{r}},e_{\Theta},e_{x_{3}})$ is a positively
oriented orthonormal frame
with $e_{\mathbf{r}}$ in the radial direction and $e_{x_{3}}$ in the direction
of increasing $x_{3}$.
The \emph{swirl} of $f$ is the scalar function $f_{e_{\Theta}}$. 
By rotation invariance of the problem, the choice of the symmetry axis is clearly unimportant.

\begin{proposition}\label{prop:Axysimm}
There exists $\delta_{2} > 0$ such that the following holds. 
Let $w_0$ be an axisymmetric divergence free vector field
with zero swirl, which belongs to~$H^{2} \cap L^{2}(|x-x'|^{-1} dx)$, for some $x' \in \Rtre$.
Then the set~$\Pi_{\delta_{2},\bar{x}}$
is regular for any
suitable weak solution~$u \in L^{\infty}_{t}L^{2}_{x} \cap L^{2}_{t}\dot{H}^{1}_{x}$ to the Navier--Stokes equation 
with divergence free data~$u_{0} \in L^{2}$ such that
\begin{equation}\nonumber
    \| |x-\bar{x}|^{-1/2} (u_{0} - w_{0})\|_{L^{2}}\le \delta_{2} e^{-\delta_{2}^{-1} \left( 1 + \| w_{0} \|^{16/3}_{H^{2}}  \right) } \, .
\end{equation}
\end{proposition}

\subsection{A nonlinear smallness assumption in the Koch--Tataru space} 
\label{SubSecGallChem}

Following~\cite{ChemGall},
we consider initial data $w_{0}$ such that
\begin{equation}\label{GallChemSmallnessCond}
\| \mathbb{P} (  e^{t\Delta}w_{0} \cdot \nabla e^{t \Delta}w_{0} )  \|_{E} \leq \sigma \exp \big( \, -\sigma^{-1} \| w_{0}\|^{4}_{\dot{B}_{\infty,2}^{-1}} \, \big) \, ,
\end{equation}
for a sufficiently small constant $\sigma$.
Here $\mathbb{P}$ is the projection on the linear subspace of the 
divergence free vector fields,
while 
$$
\textstyle
\| f \|_{E} := \| f \|_{L^{1}_{t}\dot{B}_{\infty, 1}^{-1}(\Rtre)} + \sum_{j \in \mathbb{Z}} 2^{-j} \left( \int_{0}^{\infty}  \| \Delta_{j} f (t)\|^{2}_{L^{\infty}(\Rtre)} t dt \right)^{1/2}
$$
and $\Delta_{j}$ is the frequency projection onto the annulus $2^{j-1} \leq |\xi| \leq 2^{j}$ ($\xi$ is the Fourier 
conjugate variable of $x$).

The well-posedness for these initial data have been proved in~\cite{ChemGall}. The authors also 
provide concrete examples of divergence free vector fields that satisfies~(\ref{GallChemSmallnessCond}) and are arbitrarily 
large in~$\dot{B}^{-1}_{\infty, \infty}$. Indeed, they consider the family  
\begin{equation}\nonumber
w_{0,\varepsilon} := (\partial_{2} \phi_{\varepsilon}, - \partial_{1} \phi_{\varepsilon}, 0) \, ,
\end{equation}
where
\begin{equation}\nonumber
\phi_{\varepsilon} := \frac{(- \ln \varepsilon)^{1/5}}{\varepsilon^{1-\alpha}} \cos(\varepsilon^{-1} x_{3}) \phi(x_{1}, \varepsilon^{-\alpha}x_{2},x_{3}),
\quad \alpha, \varepsilon \in (0,1) \, ,
\end{equation}
and $\phi$ is a Schwartz function.
Letting $\varepsilon \to 0$ (see~\cite[Theorem 2]{ChemGall})
\begin{equation}\nonumber
\| \mathbb{P} (  e^{it\Delta}w_{0, \varepsilon} \cdot \nabla e^{it \Delta}w_{0, \varepsilon} )  \|_{E} \to 0, 
\qquad
\| w_{0,\varepsilon} \|_{\dot{B}^{-1}_{\infty, \infty}} \to \infty \, .
\end{equation}
 
\begin{proposition}\label{prop:CGPert}
Let $(r,q)$ be an admissible couple (see Definition~\ref{Def:Admissibility}) with $q \neq \infty$. 
There exists $\delta_{0} > 0$ such that the following holds.
Let $w_{0} \in H^{1/2} \cap L^{2}(|x-x'|^{-1} dx)$, for some $x' \in \Rtre$, with zero divergence
such that~(\ref{GallChemSmallnessCond}) holds with $\sigma >0$ sufficiently small. 
Then the set $\Pi_{\delta_{0}, \bar{x}}$ is regular for any suitable weak 
solution $u \in L^{\infty}_{t}L^{2}_{x} \cap L^{2}_{t}\dot{H}^{1}_{x}$ to the Navier--Stokes equation with
divergence free data
$u_{0} \in L^{2}$ such that
\begin{equation}\nonumber
    \| |x-\bar{x}|^{-1/2} (u_{0} - w_{0})\|_{L^{2}}\le \delta_{0} e^{-\delta_{0}^{-1} \| w \|^{r}_{L^{r}_{t}L^{q}_{x}}} \, ,
\end{equation}
where $w$ is the (unique) solution to the Navier--Stokes equation with data $w_{0}$.
\end{proposition}

\subsection{2D solutions}\label{Subsec:2dSol}

For any 2D vector field
$ F(x_{1},x_{2})=(F_1(x_{1},x_{2}),F_2(x_{1},x_{2})) $, 
we have defined a 3D extension by 
\begin{equation}\label{def:3dExtensVect}
\widetilde{F}(x_{1},x_{2},x_{3}) := (F_1(x_{1},x_{2}), F_2(x_{1},x_{2}),0),
\end{equation}
If $\Phi(x_{1}, x_{2})$ is a 2D scalar field, the 3D extension 
is $\widetilde{\Phi}(x_{1}, x_{2}, x_{3}) := \Phi(x_{1}, x_{2})$.

In the following proposition we analyze 3D weak solutions with initial data that are 
close (in the weighted $L^{2}$ norm)
to the 3D extension of a 2D vector field $W_{0}$.
Indeed, it is well known that the 2D Navier stokes equation is well-posed as long as 
$W_{0} \in L^{2}(\mathbb{R}^{2})$ is divergence free.
%
%
%
%
%
%
%
%
Given a 2D solution $W$ with initial data $W_{0}$ and pressure~$P_{W}$, 
it is immediate 
to check
that the 3D extension~$\widetilde{W}$ is a solution to the 
3D Navier--Stokes equation 
with initial data~$\widetilde{W_{0}}$ and pressure~$P_{\widetilde{W}} := \widetilde{P_{W}}$. 
Thus we may wonder whether 
small perturbations of~$\widetilde{W_{0}}$ still give rise to weak solutions with 
some additional regularity. 
This is positively addressed in the following proposition. 
\begin{proposition}\label{Prop:2DNew}
There exists a constant $\delta_{0} >0$ such that the following holds.
Let $W$ be 
the (unique) solution 
to the 2D Navier--Stokes equation
with initial data~$W_{0} \in L^{2}(\mathbb{R}^{2}) \cap L^{1}(\mathbb{R}^{2})$, where
$W_0$ is a 2D divergence 
free vector field.
For any 3D divergence free initial datum $u_{0} \in L^{2}_{loc}(\Rtre)$, 
such that
$u_{0} - \widetilde{W_{0}} \in L^{2}(\Rtre)$
and
   \begin{equation}\nonumber
      \| |x-\bar{x}|^{-1/2} (u_{0} - \widetilde{W_{0}})\|_{L^{2}(\Rtre)}\le \delta_{0} e^{- \delta_{0}^{-1} \| W \|^{2}_{L^{2}_{t}L^{\infty}(\mathbb{R}^{2})}} \, ,
    \end{equation} 
there exists a suitable weak solution~$u$
to the 3D Navier--Stokes equation 
for which the set $\Pi_{\delta_{0},\bar{x}}$ is regular.
\end{proposition}

\subsection{Beltrami fields}\label{Sec:Beltrami}

A Beltrami field $w_{0}$ is an eigenfunction of the $\curl$ operator, relative to a non zero eigenvalue
$\lambda$. 
These vector fields are automatically divergence free 
and analytic. Indeed, they also satisfies $\Delta w_{0} = -\lambda^{2} w_{0}$. 
Letting~$P= - \frac{1}{2} |w_{0}(x)|^{2}$, they are also stationary solutions to the Euler equation, 
namely~$(w_{0} \cdot \nabla ) w_{0} = - \nabla P$. 
Keeping this in mind, it is easy to check that the family of rescaled fields 
$$
w(t,x) = e^{-\lambda^{2} t} w_{0}(x) \, ,
$$ 
is a solution to the Navier--Stokes equation with pressure $P (t,x)= - \frac{1}{2} |w(x)|^{2}$.
Simple examples of Beltrami fields 
are the so called ABC flows, namely
$$
( A \sin x_{3} + C \cos x_{2} ,\, B \sin  x_{1} + A \cos  x_{3}  ,  \, C \sin x_{2} + B \cos x_{1}), \quad A,B,C \in \mathbb{R} \, .
$$
Clearly, these are bounded vector fields with infinite energy. In fact, any Beltrami field on $\Rtre$ has infinite~$L^{2}$-norm, 
since it is a nonzero eigenfunction of the laplacian. 
The interest of Beltrami fields is that, as a consequence of the 
Arnold structure theorem~\cite{Arnold}, they are natural candidates to posses rich topological structures.   
This means that their stream and vortex lines (that in this case coincides), can be knotted and linked in very complicated ways. 
Indeed, any (locally bounded) link can be realized by the vortex lines
of a suitable Beltrami field on $\mathbb{R}^{3}$, as proved 
in~\cite{AlbDan}. Also vortex tubes linked and knotted in arbitrarily complicated way can be realized using Beltrami fields; see~\cite{AlbDan2}. 
Here, one can moreover assume that these Beltrami fields are bounded (in fact, bounded by $C(1+|x|)^{-1}$ with all their derivatives).
Similar results hold on the three dimensional torus~\cite{AlbDanTorr}. In this setting, small perturbations of Beltrami fields are interesting since 
they may realize the well known physical phaenomenon of vortex reconnection~\cite{AlbDanRen}. Whether a similar result may hold on~$\mathbb{R}^{3}$ is 
still open.
In the following proposition we observe that small perturbations of bounded Beltrami fields on~$\mathbb{R}^{3}$ 
can be also analyzed in the framework of this paper.

\begin{proposition}\label{Prop:Beltr}
There exists $\delta_{3} > 0$ such that the following holds.
Let $w_{0} \in L^{\infty}(\Rtre)$ such that $\nabla \times w_{0} = \lambda w_{0}$, for 
some $\lambda \neq 0$. 
For any divergence free initial datum~$u_{0} \in L^{2}_{loc}$, 
such that~$u_{0} - w_{0} \in L^{2}$ and
\begin{equation}\nonumber
    \| |x-\bar{x}|^{-1/2} (u_{0} - w_{0})\|_{L^{2}}\le \delta_{3} e^{- \delta_{3}^{-1} \lambda^{-2} \| w_{0} \|^{2}_{L^{\infty}}} \, ,
\end{equation}
there exists a suitable 
weak solution~$u$ to the Navier--Stokes equation for which the set $\Pi_{\delta_{3},\bar{x}}$ is regular.
\end{proposition}

\section{Set up and preliminaries}\label{sec:prelim}

\subsubsection*{Suitable solutions}
Let $u_{0} \in L^{2}_{loc}(\Rtre)$ be a divergence free vector field.
Following \cite[Section 7]{CKN}, \cite[Chapter 30]{Lem} and \cite{Lin}, we say that $u$ is a 
suitable weak  solution to the Navier--Stokes
with initial data $u_{0}$ if:
  \begin{enumerate}
    \item there exists $P \in L^{3/2}_{loc}((0,\infty) \times \Rtre)$ such that $(u,P)$ satisfies the first two equations in (\ref{CauchyNS}) 
    in the sense of distributions;    
    \item $u(t) \to u_{0}$ weakly in~$L^{2}$, as $t \rightarrow 0^{+}$. 
    \item For any compact set $K \subset \Rtre$: 
    $$
     \esssup_{t > 0} \int_{K} |u(t,x)|^{2} \ dx < \infty, 
      \qquad     
      \int_{0}^{\infty} \int_{K} |\nabla u(t,x)|^{2} \  dx dt < \infty  \, ;
     $$
    \item The following local energy inequality is valid 
    \begin{eqnarray}\label{GenEnIneq}
        & &
    \int_{\Rtre} |u|^{2} \phi (t) + 2 \int_{0}^{t}\int_{\Rtre} | \nabla u |^{2} \phi
    \leq    \int_{\Rtre} |u_{0}|^{2} \phi (0)
    \\ \nonumber
         & + & 
       \int_{0}^{t} \int_{\Rtre} |u|^{2} 
    (\phi_{t} + \Delta \phi) 
    +\int_{0}^{t}\int_{\Rtre} (|u|^{2} + 2P)u \cdot \nabla \phi \, ,
    \end{eqnarray}        
    for all non negative $\phi \in C^{\infty}_{c}(\mathbb{R} \times \Rtre)$
    and for all $t > 0$.
  \end{enumerate}

Suitable weak solutions are $L^{2}$-weakly continuous as functions of time~\cite[pp. 281--282]{Temam}, 
thus the initial condition (2) makes sense.
The solutions considered by Leray in~\cite{Ler}
are suitable for any divergence free $u_{0} \in L^{2}(\Rtre)$; see \cite[Theorem 2.3]{S2} or~\cite[Proposition 30.1 (A)]{Lem}.
Moreover, these solutions belong to $L^{\infty}_{t}L^{2}_{x} \cap L^{2}_{t}\dot{H}^{1}_{x}$
and satisfy the energy  
inequality 
\begin{equation}\label{ClassicEnergy}
\textstyle
\int_{\Rtre} |u (t)|^{2}  + \int_{0}^{t} \int_{\Rtre} | \nabla u |^{2} \leq \int_{\Rtre} |u(0)|^{2} \, ,
\quad
\mbox{for all}
\quad 
t > 0 \, .
\end{equation}
\subsubsection*{Suitable solutions with bounded energy}
We point out that any suitable weak solution which belongs to $L^{\infty}_{t}L^{2}_{x} \cap L^{2}_{t}\dot{H}^{1}_{x}$
satisfies the inequality~\eqref{ClassicEnergy}, as consequence of the local inequality (\ref{GenEnIneq}), via a simple limiting argument.
From (\ref{ClassicEnergy}) and the weak 
convergence to the data, also the strong $L^{2}$ convergence~$u(t) \to u_{0}$
can be easily deduced.
Moreover, thanks to the suitability, the energy inequality can be 
\lq restarted'  
\begin{equation}\label{RestClassicEnergy}
\textstyle
\int_{\Rtre} |u (t)|^{2}  + \int_{t_{0}}^{t} \int_{\Rtre} | \nabla u |^{2} \leq \int_{\Rtre} |u(t_{0})|^{2},
\quad
\mbox{for all}
\quad 
t > t_0 \, ,
\end{equation}   
at almost any time $t_{0} > 0$. 
From this fact and the weak continuity, also the strong $L^{2}$
continuity follows, at almost any time $t_{0} >0$.
In the proof of~Lemma \ref{Lemma:PertGenEnIneq} we will show how to deduce a family of \lq restarted' 
local energy inequalities from (\ref{GenEnIneq}), the same argument allows to deduce
(\ref{RestClassicEnergy}) by (\ref{ClassicEnergy}).

By Sobolev's embedding and interpolation, any function in~$L^{\infty}_{t}L^{2}_{x} \cap L^{2}_{t}\dot{H}^{1}_{x}$ 
also belongs to 
\begin{equation}\label{UIP}
L^{\bar{r}}_{t}L^{\bar{q}}_{x} \ \ \mbox{if} \ \ 
2/\bar{r} + 3/ \bar{q} = 3/2, \ \ \bar{q} \leq 6,
\qquad
L^{\bar{r}}_{T}L^{\bar{q}}_{x} \ \ \mbox{if} \ \ 
2/\bar{r} + 3/ \bar{q} \geq 3/2, \ \ 2\leq \bar{q} \leq 6 \, ,
\end{equation}
for all $T >0$, where~$L^{r}_{T}$ means that the time integration is restricted to the interval~$(0,T)$.
In particular, this allows to make sense to the following representation formula for the pressure
\begin{equation}\label{RecoverP}
P = - \Delta^{-1} \nabla \otimes \nabla \cdot (u \otimes u) = \mathcal{R} \otimes \mathcal{R} \cdot (u \otimes u) \, ,
\end{equation}
at almost any time $t>0$. 
Here $\mathcal{R} := (\mathcal{R}_{1}, \mathcal{R}_{2}, \mathcal{R}_{3})$ and $\mathcal{R}_{j}$ is the $j$-th coordinate 
oriented Riesz transform. Indeed, using the $L^{p>1}$ boundedness of $\mathcal{R}_{j}$,
the pressure automatically 
belongs to 
\begin{equation}\label{UIPPressure}
L^{\bar{r}}_{t}L^{\bar{q}}_{x} \ \ \mbox{if} \ \ 
2/\bar{r} + 3/ \bar{q} = 3, \ \ 1 < \bar{q} \leq 3,
\qquad
L^{\bar{r}}_{T}L^{\bar{q}}_{x} \ \ \mbox{if} \ \ 
2/\bar{r} + 3/ \bar{q} \geq 3, \ \ 1 < \bar{q} \leq 3 \, .
\end{equation}

\subsubsection*{Regularity properties} The following regularity criterion~\cite[Proposition 2]{CKN}, 
applies to suitable weak solutions. This
criterion will be fundamental in the rest of the paper.
We denote $B(x,r) \subset \Rtre$ the open ball of radius $r$ centred at $x$. Let  
\begin{equation}\nonumber
  Q_{r}(t,x) := (t- r^{2}, t) \times B(x,r) 
\end{equation}
the (space-time) \emph{parabolic cylinder} 
of radius $r$ with top point $(t,x)$ and
 \begin{equation}\nonumber \label{Def:SPCparabolic}
  Q^{*}_{r}(t,x) := Q_{r}(t + r^{2}/8,x) =
 ( t- 7r^{2}/8 < s < t+r^{2}/8 ) \times B(x,r) \, .
\end{equation}

\begin{lemma}[Caffarelli--Kohn--Nirenberg]\label{Lemma:CKNLEmma}
  There is 
  an absolute constant $\varepsilon^{*}$ such that the following holds. A point $(t,x)$
  is regular (see Definition~\ref{Def:Regulaity}) for any 
  suitable weak solution $u$ to the Navier--Stokes equation such that
  \begin{equation}\label{CKNConditionSB}
  \limsup_{r \rightarrow 0} \frac{1}{r}
  \int\int_{Q^{*}_{r}(t,x)} |\nabla u|^{2} < \varepsilon^{*}.
  \end{equation}
\end{lemma}

\begin{remark}\label{Rem:Uniqueness}
It is well known (see for instance~\cite[Proposition 14.2]{Lem}) that reference solutions~$w$ to the Navier--Stokeq 
equation (see Definition~\ref{Def:RefSol}) satisfy
the energy identity 
$$\| w(t) \|^{2}_{L^{2}} + \int_{0}^{t} \| \nabla w \|^{2}_{L^{2}} = \| w_{0} \|^{2}_{L^{2}} \, .$$ 
Moreover, the Prodi--Serrin uniqueness result~\cite{Prodi}, \cite{Ser2}, 
tells us that these solutions
are unique in the 
class of weak solutions~$w'$
which satisfies the relative energy inequality.  
Thus~$w$ must coincide
with the solutions given by the Leray approximation procedure, that are, in particular, suitable.    
Moreover $w \in C_{b}([0,\infty); L^{2}(\Rtre))$, see \cite[Proposition 14.3]{Lem}, namely it also satisfies
$\| w(s) - w(t) \|_{L^{2}_{x}} \to 0$ as $s \to t$, for all $t \geq 0$.
Finally, all the points~$(t,x) \in (0,\infty) \times \Rtre$ 
are regular for $w$. In fact, the regularity condition~\eqref{CKNConditionSB} is satisfied at any~$(t,x) \in (0,\infty) \times \Rtre$; 
see for instance Lemma~\ref{Lemma:Postposto}. For a more direct argument we refer to~\cite{Fabes}, \cite{Giga}.

\end{remark}

\section{Proof of the main Theorem \ref{MainThmPerturbation}}\label{sec:proof}

We prove it as a consequence of the more general Theorem~\ref{MainThmPerturbationMoreGen}. The advantage 
is that this will allow us to consider
reference solutions with infinite energy, like in the following definition.

\begin{definition}\label{Def:GenRefSol}
We say that a suitable
weak solution
$w$ to the Navier--Stokes equation with divergence free initial 
data $w_{0} \in L^{2}_{loc}(\Rtre)$ is a
\emph{generalized reference solution of size} $\mathcal{K}$ if
(\ref{Def:KRefProp}) holds, for some admissible couple, and the
regularity condition~(\ref{CKNConditionSB}) is satisfied at any $(t,x) \in (0,\infty) \times \Rtre$. 
We also require~$w \in C_{b}([0,\infty); L^{2}(K))$, for any compact set $K \subset \Rtre$, and that the 
pressure can be represented as
$P = \mathcal{T} \cdot (w \otimes w)$, where $\mathcal{T}$ is a~$3\times3$ symmetric matrix of operators that are
bounded on $L^{p}(|x|^{\alpha p} dx)$, for all $1 < p < \infty$ and 
$-3/p < \alpha < 3-3/p$. 
\end{definition}

The difference is that 
in the generalized notion we are not assuming~$w_{0} \in L^{2}(\Rtre)$
and~$w \in L^{\infty}_{t}L^{2}_{x} \cap L^{2}_{t}\dot{H}^{1}_{x}$. 
This makes necessary to require a priori the suitability, the regularity condition~\eqref{CKNConditionSB}, the
continuity of $t \to w(t) \in L^{2}(K)$ 
and the representation formula for the pressure. 

Notice that all these properties are shared by reference solutions; see Remark~\ref{Rem:Uniqueness}. In this case, the
representation formula for the pressure is satisfied letting~$\mathcal{T} = \mathcal{R} \otimes \mathcal{R}$, 
where~$\mathcal{R} := (\mathcal{R}_{1}, \mathcal{R}_{2}, \mathcal{R}_{3})$ 
and $\mathcal{R}_{j}$ is the $j$-th coordinate 
oriented Riesz transform; see~\eqref{RecoverP}.
Notice that $\mathcal{R} \otimes \mathcal{R}$ 
satisfies the weighted estimates we have required for $\mathcal{T}$. These family of estimate indeed 
holds for even more general singular integrals of convolution type~\cite{Stein}. 

\begin{definition}\label{ModSolDef}
Let $v_{0} \in L^{2}(\Rtre)$. We say that 
$v \in L^{\infty}_{t}L^{2}_{x} \cap L^{2}_{t}\dot{H}^{1}_{x}$ 
is a suitable weak solution
to the perturbed Navier--Stokes equation, around the solution $w$, with initial data $v_{0}$,
when there exists $P_{v} \in L^{3/2}_{loc}((0,\infty) \times \Rtre)$
such that 
\begin{equation}\label{ModProbInPrelimin}
\left \{
\begin{array}{rcl}
\partial_{t}v + (v \cdot \nabla) v + 
(v \cdot \nabla) w + (w \cdot \nabla) v
  - \Delta v & = & -\nabla P_{v}  \\
\nabla \cdot v & = & 0  \, ,
\end{array}\right. 
\end{equation}
in the sense of distributions, and 
$v(t) \to v_{0}$ weekly in $L^{2}$ as $t \to 0^{+}$.
Moreover, given any non negative $\phi \in C^{\infty}_{c}(\mathbb{R} \times \Rtre)$, the perturbed (local) energy inequality 
\begin{eqnarray}\label{eq:StrongPertGenInProofOLD}
  \textstyle
  \int_{\mathbb{R}^{3}}  
  |v(t,x)|^{2} \phi(t,x) dx
  & + & 
  \textstyle
  2 \int_{t_0}^{t}\int_{\mathbb{R}^{3}} 
   |\nabla  v|^{2} \phi
  \le 
  \int_{\mathbb{R}^{3}}  |v(t_0,x)|^{2}  \phi(t_0,x) dx 
  \\ \nonumber
 &  + &
  \textstyle
  \int_{t_0}^{t} \int_{\mathbb{R}^{3}} 
  |v|^{2} (\phi_{t} + \Delta \phi)
   + 
  \int_{t_0}^{t} \int_{\mathbb{R}^{3}} (|v|^{2} 
  + 2P_{v})v \cdot \nabla \phi
  \\ \nonumber
 & + &
  \textstyle
  \int_{t_0}^{t}\int_{\mathbb{R}^{3}} 
  |v|^{2} w \cdot \nabla \phi
  +
   2 \int_{t_0}^{t}\int_{\mathbb{R}^{3}} 
  (v \cdot w) v \cdot \nabla \phi  
  +  (v \cdot \nabla ) v \cdot w \phi
\end{eqnarray}
is satisfied for all $t > t_{0}$, where $t_{0}$ may be zero or almost any real number in~$(0,\infty)$, and there 
exists a $3\times 3$ symmetric matrix~$\mathcal{T}$ of operators, that are bounded on
$L^{p}(|x|^{\alpha p} dx)$, for all $1 < p < \infty$ and 
$-3/p < \alpha < 3-3/p$, and such that  
\begin{equation}\label{ReprFormSing}
P_{v}:= \mathcal{T} \cdot (v \otimes v) + 2 \, \mathcal{T}  \cdot (v \otimes w) \, .
\end{equation}

\end{definition}

Thus, our main theorem is now the following:

\begin{theorem}\label{MainThmPerturbationMoreGen}
Given $(r,q)$ admissible, there exists a constant $\delta_{0} > 0$ such that the following holds. 
Let~$w$ be a generalized reference solution of size $\mathcal{K}$ to the 
Navier--Stokes equation 
with divergence free data $w_{0} \in L^{2}_{loc}$.
The set $\Pi_{\delta_{0},\bar{x}}$ is regular
for every suitable weak solution~$u$ to the Navier--Stokes equation 
with divergence free data $u_{0} \in L^{2}_{loc}$ such that
\begin{equation}\label{MTSAMoreGen}
    \| |x-\bar{x}|^{-1/2} (u_{0} - w_{0})\|_{L^{2}}\le \delta_{0} e^{-\mathcal{K}/\delta_{0}},
\end{equation}
and such that $u-w$ is a 
suitable weak solution to the perturbed Navier--Stokes equation 
\eqref{ModProbInPrelimin}, around the solution $w$, with data $u_{0} - w_{0} \in L^{2}$.
\end{theorem}

We have already observed that that 
reference solutions of size~$\mathcal{K}$ are, in particular, generalized 
reference solutions of size $\mathcal{K}$; see also Lemma~\ref{Lemma:Postposto}. 
 Moreover, it is straightforward to check that if 
$u \in L^{\infty}_{t}L^{2}_{x} \cap L^{2}_{t}\dot{H}^{1}_{x}$ is a suitable weak solution to the Navier--Stokes equation 
with pressure $P_{u}$ and 
data $u_0 \in L^{2}(\Rtre)$
and $w$ is a reference solution with pressure $P_{w}$ and data $w_0 \in L^{2}(\Rtre)$, 
then
the difference~$u-w$ is a suitable weak solution to the 
perturbed Navier--Stokes equation~\ref{ModProbInPrelimin}, around the solution $w$, with pressure~$P_{u-w}:=P_{u} - P_{w}$
and data~$u_{0} - w_{0}$. We refer to Lemma~\ref{Lemma:PertGenEnIneq} for details.
Here we only remark that, since we have the representation 
formulas~$P_{u} = \mathcal{R} \otimes \mathcal{R} \cdot (u \otimes u)$ and~$P_{w} = \mathcal{R} \otimes \mathcal{R} \cdot (w \otimes w)$, 
see~\eqref{RecoverP}, 
\eqref{ReprFormSing} holds taking~$\mathcal{T} = \mathcal{R} \otimes \mathcal{R}$.
In conclusion, Theorem \ref{MainThmPerturbationMoreGen} 
implies Theorem \ref{MainThmPerturbation}.

Moreover, keeping this in mind, the existence of suitable solutions~$u$ that satisfy the
assumption of Theorem~\ref{MainThmPerturbationMoreGen} is ensured by the Leray theory, 
once we consider data~$u_{0} \in L^{2}$ and $w$ is a reference solution. 
This is not obvious when when~$w$ is a generalized reference solution, since we may need to handle unbounded 
energies, for instance~$u_{0}$ only locally square integrable. However, there are some relevant situations in which these
(infinite energy) suitable solutions can be actually constructed by a simple adaptation of the Leray theory; see Proposition~\ref{InfiniteExistenceNew}. 
This is the case when we consider the admissible couple~$(r, q)=(2, \infty)$, for which Theorem~\ref{MainThmPerturbationMoreGen} 
is substantially more efficient than (the simpler version) Theorem~\ref{MainThmPerturbation}.

\subsubsection*{Idea of the proof}

The main idea behind 
Theorem~\ref{CKNSmallData} is to use the 
local energy inequality (\ref{GenEnIneq}) and the weighted $L^{2}$ smallness assumption on $u_{0}$
to prove that the regularity condition (\ref{CKNConditionSB}) is satisfied at any point 
inside the regular set.

A natural way to attack the perturbative case is
trying to do the same, using 
the perturbed energy inequality~(\ref{eq:StrongPertGenInProofOLD}), with initial data $v_{0} := u_{0}-w_{0}$.   
The difficulty is that the new terms in the perturbed energy inequality (\ref{eq:StrongPertGenInProofOLD}) 
contain the reference solution~$w$, so that they can not be handled in a perturbative way, since~$w$ may be large.  
To avoid this, we 
distinguish two time regimes $t \leq t^{*}$, $t > t^{*}$
and choose $t^{*}$ in such a way that these hard terms can be controlled for $t > t^{*}$, using a cancelation 
in the energy inequality that is analogous to the one used in the proof of Theorem~\ref{CKNSmallData}. Then we use the
(exponential) smallness assumption (\ref{ImprFinalm}) on $u_{0} - v_{0}$ to control the weighted $L^{2}$ norm of   
the solution $u-w$ up to the time $t^{*}$, so that we can \lq restart' the small data problem at the time $t^{*}$. 
An appropriate choice of $t^{*}$ permits to conclude the proof.

\begin{proof}

We can assume  $\bar{x} = 0$ since the general case follows by translation.
We divide the proof into three steps.

\subsection{First step: the perturbed equation}

\

Let $v_{0} := u_{0} - w_{0}$ and $v := u-w$.
Fix $\xi \in \mathbb{R}^{3}$. 
We have assumed~$v$ to be a
weak solution to the perturbed Navier--Stokes equation~\eqref{ModProbInPrelimin}, 
that, after the change of variables    
\begin{equation}\label{TransfI}
  (t, y) = (t, x - \xi t),
  \qquad
  v_{\xi}(t,y) := v(t,x),
  \qquad
  w_{\xi}(t,y) := w(t,x),
\end{equation}
becomes
\begin{equation}\nonumber
  \left \{
  \begin{array}{rcl}
  \partial_{t}v_{\xi} + ((v_{\xi} - \xi )\cdot \nabla) v_{\xi} + 
  (v_{\xi} \cdot \nabla) w_{\xi} + (w_{\xi} \cdot \nabla) v_{\xi}
    -\Delta v_{\xi} & = & - \nabla P_{v_{\xi}}   \\
  \nabla \cdot v_{\xi} & = & 0  \\
  v_{\xi}(0,\cdot) & = & v_{0} \\
  \mathcal{T} \cdot (v_{\xi} \otimes v_{\xi}) 
  + 2 \, \mathcal{T}  \cdot (v_{\xi} \otimes w_{\xi}) = P_{v_{\xi}}(t,y) := P_{v}(t,x)    .
  \end{array}\right. 
\end{equation}
where $v_{\xi}(0,\cdot)  =  v_{0}$ means that we have $L^{2}$ weak convergence as $t \to 0^{+}$.  
Let 
\begin{equation}\label{Def:SigmaNu}
  \sigma_{\mu}(y) := (\mu + |y|^{2})^{-\frac12},
\quad
\mu >0
\end{equation}
and
define 
$$
t^{*}(\xi, \mu) := \sup \{ t \in [0,\infty] \ : 
\  \textstyle \int_{0}^{t} \int_{\Rtre} \sigma_{\mu}(y) |\nabla v_{\xi} (\tau , y)|^{2}  dy d \tau  \leq \int_{0}^{t} \| w_{\xi} \|^{r}_{L^{q}_{y}}(\tau)  d\tau \}.
$$ 
The integral
$\int_{0}^{t} \int_{\Rtre} \sigma_{\mu}(y) |\nabla v_{\xi} (\tau , y)|^{2}  dy d \tau $ is finite, for any $t >0$, 
since it is smaller than  
$\mu^{-\frac{1}{2}}\int_{0}^{t} \int_{\Rtre} |\nabla v (\tau , y)|^{2}  dy d \tau $, that is bounded by assumption. However, our ultimate goal is
to obtain a bound that is uniform in~$\mu >0$.
Since the functions 
$t \to \int_{0}^{t} \int_{\Rtre} \sigma_{\mu}(y) |\nabla v_{\xi} (\tau , y)|^{2}  dy d \tau $ and 
$t \to \int_{0}^{t} \| w_{\xi} \|^{r}_{L^{q}_{y}}(\tau)  d\tau$ are continuous  
and non decreasing in $t>0$, it is easy to show that
\begin{equation}\label{Def:RBar}
\begin{array}{lccc}
\mbox{(i)} & \int_{0}^{t^{*}} \int_{\Rtre} \sigma_{\mu}(y) |\nabla v_{\xi} (\tau , y)|^{2}  dy d \tau  = \int_{0}^{t^{*}} \| w_{\xi} \|^{r}_{L^{q}_{y}}(\tau)  d\tau ; & &  
\\
&&&
\\
\mbox{(ii)} & \int_{t^{*}}^{t} \int_{\Rtre} \sigma_{\mu}(y) |\nabla v_{\xi} (\tau , y)|^{2}  dy d \tau  > \int_{t^{*}}^{t} \| w_{\xi} \|^{r}_{L^{q}_{y}}(\tau)  d\tau  & \text{for all} & t > t^{*},
\end{array}
\end{equation}
provided $t^{*} < \infty$.
If $t^{*} = \infty$
either there exists a divergent sequence $t_{n} \to \infty$ such that
$$
\int_{0}^{t_{n}} \int_{\Rtre} \sigma_{\mu}(y) |\nabla v_{\xi} (\tau , y)|^{2}  dy d \tau = \int_{0}^{t_n} \| w_{\xi} \|^{r}_{L^{q}_{y}}(\tau) d\tau, 
$$
or
there exists a single (possibly equal to zero) $t_{n}$ 
which satisfies this and such that   
$$ 
\int_{t_{n}}^{t} \int_{\Rtre} \sigma_{\mu}(y)|\nabla v_{\xi} (\tau , y)|^{2} dy d \tau  < \int_{t_{n}}^{t} \| w_{\xi} \|^{r}_{L^{q}_{y}}(\tau) d\tau 
\quad
\mbox {for all}
\quad
 t > t_{n}.
$$ 
In both cases
we have 
$$
\int_{0}^{\infty} \int_{\Rtre} \sigma_{\mu}(y) |\nabla v_{\xi} (\tau , y)|^{2} dy d\tau  \leq 
\int_{0}^{\infty} \| w_{\xi} \|^{r}_{L^{q}_{y}}(\tau) d\tau =: \mathcal{K},
$$ 
that, recalling the property~(i) in~(\ref{Def:RBar}), implies 
\begin{equation}\label{Prop:BB(t)BeforeRBar}
\int_{0}^{t^{*}} \int_{\Rtre} \sigma_{\mu}(y) |\nabla v_{\xi} (\tau , y)|^{2}  dy d \tau  \leq \mathcal{K} \, ,
\end{equation}
for any possible~$t^{*} \in [0,\infty]$.

For any $t > t^{*}$ we
consider
the (space-time) segment 
$$
L(t,\xi) := \{ (s, \xi s) \ : \ s \in (0,t) \}. 
$$
We will investigate for which $(t,\xi)$ this set is regular. 
Notice that the change of variables (\ref{TransfI})
maps $L(t, \xi)$ 
into $(0,t) \times \left\{ 0 \right\} $, a vertical segment above the origin of the space-time.

\subsection{Second step: estimates for $t \leq t^{*}$}\label{SSFNM}

We have assumed $v$ to satisfy the perturbed energy inequality~\eqref{eq:StrongPertGenInProofOLD}, that after the change of 
variables~(\ref{TransfI}) becomes  
\begin{eqnarray}\label{ModEnWitXi}
  \textstyle
   \int_{\mathbb{R}^{3}} &  & \!\!\!\!\!\!\!\!\!\!\!\!\!    
  |v_{\xi}|^{2} \phi(t,y) dy
   + 
  \textstyle
  2 \int_{0}^{t}\int_{\mathbb{R}^{3}} 
   |\nabla  v_{\xi}|^{2} \phi
  \le 
  \int_{\mathbb{R}^{3}}  |v_{0}|^{2} \phi(0,y) dy
  \\   \nonumber
  & + &
  \textstyle
    \int_{0}^{t} \int_{\mathbb{R}^{3}} 
  |v_{\xi}|^{2} (\phi_{t} - \xi \cdot \nabla \phi + \Delta \phi)
  +
  \textstyle
   (|v_{\xi}|^{2} 
  + 2P_{v_{\xi}})v_{\xi} \cdot \nabla \phi
  \\ \nonumber
  & + &
  \textstyle
  \int_{0}^{t}\int_{\mathbb{R}^{3}} 
  |v_{\xi}|^{2} w_{\xi} \cdot \nabla \phi
  +  \textstyle
  2 
  \int_{0}^{t}\int_{\mathbb{R}^{3}}  
  (v_{\xi} \cdot w_{\xi}) v_{\xi} \cdot \nabla \phi 
  +
    (v_{\xi} \cdot \nabla ) v_{\xi} \cdot w_{\xi} \phi \, ,
\end{eqnarray}
valid for all $t > 0$ and
$\phi\in C^{\infty}_{c}(\mathbb{R} \times \mathbb{R}^{3})$ non negative.
Thus 
\begin{eqnarray}\label{GenEn} 
    \textstyle
    \int_{\mathbb{R}^{3}} 
    \!\!
    & |v_{\xi}|^{2}   &    
    \!\!\!
     \phi  (t,y) dy
    + 
    \textstyle
    2 \int_{0}^{t}\int_{\mathbb{R}^{3}} 
     |\nabla  v_{\xi}|^{2} \phi 
    \le 
    \int_{\mathbb{R}^{3}}  |v_{0}|^{2} \phi(0,y) dy
    \\ \nonumber
    & + &
    \textstyle
    \!\!\!\!\!\!
    \int_{0}^{t} \int_{\mathbb{R}^{3}} 
    |v_{\xi}|^{2} (\phi_{t} - \xi \cdot \nabla \phi + \Delta \phi)
   + (|v_{\xi}|^{2} 
    + 2P_{v_{\xi}})v_{\xi} \cdot \nabla \phi
    \\  \nonumber
    & + &
    \textstyle
    \!\!\!\!\!\!
    \int_{0}^{t}\int_{\mathbb{R}^{3}} 3 |v_{\xi}|^{2} 
    |w_{\xi}| |\nabla \phi|
    + 18  |v_{\xi}| |\nabla v_{\xi}| |w_{\xi}| |\phi| \, .
  \end{eqnarray}
By a standard approximation argument
(see the proof of Lemma~8.3 in \cite{CKN}) 
this still holds for any test function of the form
\begin{equation}\nonumber
  \phi(t,y) := \psi(t) \phi_{1}(y)
\end{equation}
with
$
\phi_{1} \in C^{\infty}_{c}(\mathbb{R}^{3})$
non negative
and 
\begin{equation}\nonumber
  \psi :  [0,\infty) \rightarrow [0,\infty)
  \quad \mbox{absolutely continuous with} 
  \quad  \dot{\psi} \in L^{1}(0,\infty).
\end{equation}
We shall choose here
\begin{equation}\nonumber
  \psi(t):=1,
  \qquad
  \phi_{1} =\sigma_{\mu}(y) \chi(\delta |y|), 
\end{equation}
where $\sigma_{\mu}(y)$ has been defined in (\ref{Def:SigmaNu}), $\delta>0$  
and $\chi : [0,\infty) \to [0,\infty)$
is a smooth non increasing function such that
\begin{equation}\nonumber
  \chi =1 \ \text{on}\  [0,1],
  \qquad
  \chi =0  \ \text{on}\ [2, \infty].
\end{equation} 
Recalling~$P \in L^{3/2}_{loc}((0,\infty) \times \Rtre)$,~$w \in L^{r}_{t}L^{q}_{x}$, $v \in L^{\infty}_{t}L^{2}_{x} \cap L^{2}_{t}\dot{H}^{1}_{x}$, 
so that $v$ also belongs the mixed spaces in~\eqref{UIP}, since the same clearly holds
for $P_{v_{\xi}}, w_{\xi}, v_{\xi}$, we can easily pass to the limit  
$\delta \rightarrow 0$ so that
\begin{equation}\label{PassToTheLim}
\begin{split}
  \textstyle
  \left[ \int_{\mathbb{R}^{3}}  
    |v_{\xi}|^{2} \sigma_{\mu} \right]_{0}^{t} 
  &+
  \textstyle
  2\int_{0}^{t}\int_{\mathbb{R}^{3}}  
    |\nabla v_{\xi}|^{2} \sigma_{\mu}  \le
  \\
  & \le 
  \textstyle
  \int_{0}^{t}\int_{\mathbb{R}^{3}}  
  |v_{\xi}|^{2}(- \xi \cdot \nabla \sigma_{\mu} 
  + \Delta \sigma_{\mu}) 
  +
  \textstyle
   (|v_{\xi}|^{2} 
  + 2P_{v_{\xi}})v_{\xi}\cdot \nabla \sigma_{\mu}
  \\
  &+
  \textstyle
  18 \int_{0}^{t}\int_{\mathbb{R}^{3}} 
    |v_{\xi}| |\nabla v_{\xi}| |w_{\xi}| \sigma_{\mu}
  + 
  3  
  |v_{\xi}|^{2} | w_{\xi} | | \nabla \sigma_{\mu} |.
\end{split}
\end{equation}
Then, since
\begin{equation}\label{eq:sigmapr}
  \textstyle
 |\nabla \sigma_{\mu}| < (\mu + |y|^{2})^{-1} 
 = \sigma_{\mu}^{2},
 \qquad \Delta \sigma_{\mu} < 0,
\end{equation}
we arrive to the inequality
\begin{eqnarray}\label{PutIn}
  \textstyle
  \left[ \int_{\mathbb{R}^{3}} \sigma_{\mu} 
    |v_{\xi}|^{2} \right]_{0}^{t} 
  & + &
  \textstyle
  2\int_{0}^{t}\int_{\mathbb{R}^{3}} \sigma_{\mu} 
    |\nabla v_{\xi}|^{2}
  \le 
   |\xi|  
  \int_{0}^{t}\int_{\mathbb{R}^{3}} 
    \sigma_{\mu}^{2} |v_{\xi}|^{2}
  \\ \nonumber
  & +  &
  \textstyle
  \int_{0}^{t}\int_{\mathbb{R}^{3}} \sigma_{\mu}^{2}
    (|v_{\xi}|^{3} + 2|P_{v_{\xi}}||v_{\xi}|
    +3 |v_{\xi}|^{2}|w_{\xi}|) 
    + 18 \sigma_{\mu} |v_{\xi}| | \nabla v_{\xi} | 
    | w_{\xi} |.
\end{eqnarray}
The next goal is to deduce an integral inequality
for the functions 
\begin{equation}\nonumber
a_{\mu}(t) := \int_{\mathbb{R}^{3}} \sigma_{\mu}(y) |v_{\xi}(t,y)|^{2} dy, 
\qquad
B_{\mu}(t) := \int_{0}^{t}\int_{\mathbb{R}^{3}} 
\sigma_{\mu}(y) | \nabla v_{\xi}(\tau,y)|^{2} dy d\tau  \, .
\end{equation}
In order to bound the terms on the right hand side of (\ref{PutIn})
we use the Stein weighted estimates for singular integrals (\ref{SteinIneq})
and the Caffarelli--Kohn--Nirenberg interpolation inequalities (\ref{CKNInequalityCKNVersion}).
For brevity we will refer to these inequalities as SS and CKN, respectively.

We first bound the terms involving the pressure  
$
 P_{v_{\xi}} =   \mathcal{T} \cdot (v_{\xi} \otimes v_{\xi}) 
  + 2 \, \mathcal{T} \cdot (v_{\xi} \otimes w_{\xi}).
$
We have
\begin{align}\label{PressDec}
  \textstyle
  2\int_{\RT} \sigma_{\mu}^{2} |P_{v_{\xi}}| |v_{\xi}|
  & \textstyle
  \leq 
  2\int_{\RT} \sigma_{\mu}^{2} |v_{\xi}| 
  | \mathcal{T} \cdot (v_{\xi} \otimes v_{\xi})| 
  \\ \nonumber
  & \textstyle 
   + 4 \int_{\RT} \sigma_{\mu}^{2} |v_{\xi}| 
  | \mathcal{T}  \cdot (v_{\xi} \otimes w_{\xi})|
   =: I + II.
\end{align}
Here and in the following $Z \geq 1$ denotes several
constants, only depending on $(r,q)$, possibly increasing from line to line.
By the SS inequality~\eqref{SteinIneq} 
we have\footnote{In Definition~\ref{Def:GenRefSol} we have required $\mathcal{T}$ to satisfy the family 
of inequalities~\eqref{SteinIneq} only in the case~$\mu =0$, however the general case can be deduced as shown in~\cite[Lemma 7.2]{CKN}. 
}
\begin{equation}\nonumber
  I \le
  2 \|\sigma_{\mu} 
  \mathcal{T}   \cdot (v_{\xi} \otimes v_{\xi})\|_{L^{2}}
  \| \sigma_{\mu}  v_{\xi}  \|_{L^{2}}
  \le
  Z
  \|\sigma_{\mu} |v_{\xi}|^{2} \|_{L^{2}}
  \| \sigma_{\mu}  v_{\xi}  \|_{L^{2}} 
  \le
  Z
  \|\sigma_{\mu}^{1/2} v_{\xi} \|_{L^{4}}^{2}
  \| \sigma_{\mu}  v_{\xi}  \|_{L^{2}} 
\end{equation}
then using the CKN inequality \eqref{CKNInequalityCKNVersion} we obtain
\begin{equation}\nonumber
  I\le
  Z
  \|\sigma_{\mu}^{1/2} \nabla v_{\xi}  \|^{3/2}_{L^{2}}
  \| \sigma_{\mu}^{1/2}  v_{\xi}  \|_{L^{2}}^{1/2} 
  \cdot
  \|\sigma_{\mu}^{1/2} \nabla v_{\xi}  \|^{1/2}_{L^{2}}
  \| \sigma_{\mu}^{1/2}  v_{\xi}  \|_{L^{2}}^{1/2} 
  =
  Z
  \dot{B}_{\mu} a_{\mu}^{1/2} 
  \le 
  \frac{1}{6}\dot{B}_{\mu}  
  + Z \dot{B}_{\mu}a_{\mu}.
\end{equation}
In a similar way 
\begin{eqnarray} \nonumber
  II
  &  \le  &
   4 \| \sigma_{\mu} 
   \mathcal{T}  \cdot (v_{\xi} \otimes w_{\xi}) \|_{L^{2q/(q+1)}}
  \| \sigma_{\mu}  v_{\xi} \|_{L^{2q/(q-1)}} 
  \\ \nonumber
  & \le &
  Z
  \| \sigma_{\mu} |v_{\xi}| |w_{\xi}| \|_{L^{2q/(q+1)}}
  \| \sigma_{\mu}  v_{\xi}  \|_{L^{2q/(q-1)}}
   \le 
  Z
   \|  w_{\xi} \|_{L^{q}}
  \| \sigma_{\mu}  v_{\xi}  \|_{L^{2q/(q-1)}}^{2}
\end{eqnarray}
and, again by CKN with
$2 \theta = 1 + 3/q$, which implies $r = 1/(1-\theta)$,
where
$\theta$ is the interpolation parameter in 
(\ref{CKNInequalityCKNVersion}), we get
\begin{equation}\nonumber
  II
  \le
  Z
  \|  w_{\xi} \|_{L^{q}}
  \| \sigma_{\mu}^{1/2}  v_{\xi}  \|^{2(1-\theta)}_{L^{2}} 
  \| \sigma_{\mu}^{1/2}  \nabla v_{\xi}  \|^{2\theta}_{L^{2}},
  =
  Z
  \|  w_{\xi} \|_{L^{q}}
  a_{\mu}^{1-\theta}\dot{B}_{\mu}^{\theta}
  \le
  \frac{1}{6}\dot{B}_{\mu} 
  + Z \| w_{\xi} \|^{r}_{L^{q}} a_{\mu}. 
\end{equation}

We now consider the other terms in the right hand side of \eqref{PutIn}.
As above, we use CKN to bound
\begin{equation}\nonumber
  \textstyle
  |\xi|
  \int_{\RT} \sigma_{\mu}^{2} |v_{\xi}|^{2} 
  \le
  Z
  |\xi|
  \| \sigma_{\mu}^{1/2} \nabla v_{\xi}  \|_{L^{2}}
  \| \sigma_{\mu}^{1/2} v_{\xi} \|_{L^{2}} =
  Z|\xi|(\dot{B}_{\mu}a_{\mu})^{1/2}
  \le
  \frac{1}{6}\dot{B}_{\mu} + Z |\xi|^{2} a_{\mu};
\end{equation}
and
\begin{equation}\nonumber
  \textstyle
  \int_{\RT} \sigma_{\mu}^{2} |v_{\xi}|^{3} 
  = \| \sigma_{\mu}^{2/3} v_{\xi} \|_{L^{3}}^{3} 
  \le
  Z
  \| \sigma_{\mu}^{1/2} \nabla v_{\xi}  \|^{2}_{L^{2}}
  \| \sigma_{\mu}^{1/2} v_{\xi} \|_{L^{2}}
  =
  Z
  \dot{B}_{\mu} a_{\mu}^{1/2} 
  \leq 
  \frac{1}{6}\dot{B}_{\mu}  + Z \dot{B}_{\mu}a_{\mu}.
\end{equation}
In order to bound the terms with $w_{\xi}$ we use CKN with $2 \theta = 1 + 3/q$
\begin{equation}\nonumber
\begin{split}
  \textstyle
  3 \int_{\RT}\sigma_{\mu}^{2} |v_{\xi}|^{2} |w_{\xi}|
  \le&
  3 \|  w_{\xi} \|_{L^{q}}
  \| \sigma_{\mu}  v_{\xi}  \|_{L^{2q/(q-1)}}^{2}
  \le
  Z
  \|  w_{\xi} \|_{L^{q}}
  \| \sigma_{\mu}^{1/2}  v_{\xi}  \|^{2(1-\theta)}_{L^{2}} 
  \| \sigma_{\mu}^{1/2}  \nabla v_{\xi}  \|^{2\theta}_{L^{2}}
  \\
  =&
  Z\|  w_{\xi} \|_{L^{q}}
  a_{\mu}^{1-\theta}
  \dot{B}_{\mu}^{\theta}
  \le
   \frac{1}{6}\dot{B}_{\mu}  
  + Z \|  w_{\xi} \|^{r}_{L^{q}} a_{\mu}
\end{split}
\end{equation}
and CKN with $\theta = 1- 2/r$, which implies $2/(1-\theta) =r$,
\begin{equation}\nonumber
\begin{split}
 \textstyle
  18 \int_{\RT}\sigma_{\mu} |v_{\xi}| |\nabla v_{\xi}| 
  |w_{\xi}| 
  \le&
  18  \| \sigma_{\mu}^{1/2}  \nabla v_{\xi}  \|_{L^{2}}
  \| w_{\xi} \|_{L^{q}} 
  \| \sigma_{\mu}^{1/2}  v_{\xi}  \|_{L^{2q/(q-2)}}
  \\
  \le&
  Z
  \| \sigma_{\mu}^{1/2}  \nabla v_{\xi}  \|_{L^{2}}
  \| w_{\xi} \|_{L^{q}}
  \| \sigma_{\mu}^{1/2}  \nabla v_{\xi}  \|^{\theta}_{L^{2}}
  \| \sigma_{\mu}^{1/2}  v_{\xi}  \|^{1-\theta}_{L^{2}}
  \\
  =&
  Z\| w_{\xi} \|_{L^{q}}
  \dot{B}_{\mu}^{(1+\theta)/2}a_{\mu}^{(1-\theta)/2}
  \le
   \frac{1}{6}\dot{B}_{\mu}  
  + Z \|  w_{\xi} \|^{r}_{L^{q}} a_{\mu}.
\end{split}
\end{equation}
Now, recalling \eqref{PutIn},
summing all these inequalities and absorbing
the resulting term~$ \int_{0}^{t}\dot{B}_{\mu}(\tau)d\tau =
  B_{\mu}(t)$ from the right hand side into the left hand side, we obtain
\begin{equation*}
  a_{\mu}(t) + B_{\mu}(t) \leq 
  a_{\mu}(0) +
  Z\int_{0}^{t}  \left( 
  |\xi|^{2} + 
  \dot{B}_{\mu}(\tau) +  
  \|  w_{\xi} (\tau,\cdot) \|^{r}_{L^{q}(\Rtre)}
  \right) a_{\mu}(\tau) \ d\tau.
\end{equation*}
Let
\begin{equation*}
    a(t) := \int_{\mathbb{R}^{3}} |y|^{-1} 
    |v_{\xi}(t,y)|^{2}dy \, .
\end{equation*}
Since $a_{\mu}(0) \leq a(0)$,
we have obtained the estimate
\begin{equation*}
  a_{\mu}(t) + B_{\mu}(t) \leq 
  a(0)  +
  Z\int_{0}^{t}  \left( 
  |\xi|^{2} + 
  \dot{B}_{\mu}(\tau) + 
  \|  w_{\xi}(\tau,\cdot)  \|^{r}_{L^{q}(\Rtre)}
  \right) a_{\mu}(\tau) \ d\tau,
\end{equation*}
for some constant $Z$.
By Gr\"onwall's lemma we get,
for $0 \le t \le t^{*}$:
\begin{equation}\label{Labella}
  a_{\mu}(t)\le a(0) e^{ZA},
  \qquad
  A := B_{\mu}(t^{*})+\|w_{\xi}\|_{L^{r}_{t}L^{q}_{y}}^{r}+t^{*}|\xi|^{2}.
\end{equation}
Recalling~\eqref{Prop:BB(t)BeforeRBar}, namely that the quantity $B_{\mu}(t^{*})$ is smaller than $\mathcal{K}$, 
and since\footnote{$w_{\xi}$, at fixed $t$, is simply a translation of $w$.}
$$
\|w_{\xi} \|_{L^{r}_{t}L^{q}_{y}}^{r} = \|w\|_{L^{r}_{t}L^{q}_{x}}^{r} =: \mathcal{K} \, ,
$$
we have
$$
A\le 2 \mathcal{K} +t^{*}|\xi|^{2}.
$$ 
%
Thus, if we restrict to the 
vectors $\xi$ such that
\begin{equation}\label{ResrtCondFirst}
|\xi|^{2}t^{*} \leq 1
\end{equation} 
the estimate (\ref{Labella}) gives 
\begin{equation*}
  a_{\mu}(t) \le \epsilon^{2}
    e^{Z(2 \mathcal{K} +1 )},
    \quad
    \mbox{for all}
    \quad 
    0 \leq t \leq t^{*}, 
\end{equation*}
where we have denoted
$$
\epsilon :=
\sqrt{a(0)} =
\| |x|^{-1/2} v_0 \|_{L^{2}}. 
$$
Taking a suitably larger constant $Z$, this implies
\begin{equation}\label{BoundBySmallnessBis}
  a_{\mu}(t)\le Z e^{Z \mathcal{K}}\epsilon^{2},
  \quad
    \mbox{for all}
    \quad 
    0 \leq t \leq t^{*}.
  \end{equation}

\subsection{Third step: estimates for $t >t^{*}$}\label{SameArgument}

Here the idea is to repeat the previous argument 
starting by the point 
$(t^{*},t^{*}\xi)$ of the segment~$L(t,\xi) \times \{ 0 \}$, rather than the origin. 
To do so we want to use the inequality~\eqref{ModEnWitXi}, but the time integration has to be 
over~$[t^{*}, t]$ rather than~$[0, t]$. Since we know that~$v$ satisfies the perturbed energy 
inequality~\eqref{eq:StrongPertGenInProofOLD}, 
changing variables as in~\eqref{TransfI}, we can actually do this 
on intervals~$[t^{*}_{n}, t^{*}]$, where $t^{*}_{n}$ is a sequence 
of times, smaller or equal than~$t^{*}$, and such that $t^{*}_{n} \to t^{*}$. 
Notice that this is possible since in the inequality 
\eqref{eq:StrongPertGenInProofOLD} the integration is over the time interval $[t_0, t]$ where 
$t_0$ is allowed to be zero or almost any real number in~$(0,\infty)$.

Choosing as test functions
$\phi(t,y):=\psi_{\mu}(t)\sigma_{\mu}(y) \chi(\delta |y|)$
where $\chi$ and $\sigma_{\mu}$ are as before\footnote{To be precise we have to consider
two vanishing sequences $\delta_{n}, \mu_{n}$ instead of $\delta, \mu$, in order to be sure that the exceptional times, starting from 
which~\eqref{eq:StrongPertGenInProofOLD} 
may not be satisfied for at least one of our test functions, has measure zero. However,
we keep writing
$\delta, \mu$ for simplicity.}, while
\begin{equation}\nonumber
  \psi_{\mu}(t) := e^{-k B_{t^{*}_{n}, \mu}(t)},
  \qquad 
  B_{t^{*}_{n}, \mu}(t) := 
  \int_{t^{*}_{n}}^{t}
  \int_{\mathbb{R}^{3}} 
  \sigma_{\mu}(y) |\nabla v_{\xi}(\tau,y)|^{2} dy d\tau \, ,
\end{equation}
with $k$ a positive constant to be specified, and proceeding as before, we arrive to 
\begin{equation}\nonumber
\begin{split}
  \textstyle
  [ \int_{\mathbb{R}^{3}} 
  &
  \psi_{\mu} \sigma_{\mu} 
  |v_{\xi}|^{2} ]_{t^{*}_{n}}^{ t}
   + 
  \textstyle
  2\int_{t^{*}_{n}}^{t}\int_{\mathbb{R}^{3}} 
  \psi_{\mu} \sigma_{\mu} |\nabla v_{\xi}|^{2} 
  \le
  \\
  \le &
  \textstyle
  \int_{t^{*}_{n}}^{t}\int_{\mathbb{R}^{3}} 
  \psi_{\mu} |v_{\xi}|^{2}
  (-k \dot{B}_{t^{*}_{n} , \mu} \sigma_{\mu} 
  -\xi \cdot \nabla \sigma_{\mu} + \Delta \sigma_{\mu}) 
  +
  \int_{t^{*}_{n}}^{ t}\int_{\mathbb{R}^{3}} 
  \psi_{\mu} (|v_{\xi}|^{2} 
  +2P_{v_{\xi}} ) v_{\xi} \cdot \nabla \sigma_{\mu} 
  \\
  +&
  \textstyle
  18 \int_{t^{*}_{n}}^{ t}\int_{\mathbb{R}^{3}} 
  \psi_{\mu} \sigma_{\mu}  |v_{\xi}| |\nabla v_{\xi}| 
  |w_{\xi}| 
  +
  3 \int_{t^{*}_{n}}^{t}\int_{\mathbb{R}^{3}} 
  \psi_{\mu} |v_{\xi}|^{2} | w_{\xi} | 
  | \nabla \sigma_{\mu} | \, .
\end{split}
\end{equation}
This implies, recalling \eqref{eq:sigmapr},
\begin{equation}\label{PutIn2}
\begin{split}
  \textstyle
  [ \int_{\mathbb{R}^{3}} 
  \psi_{\mu} \sigma_{\mu} 
  &
  \textstyle
  |v_{\xi}|^{2} 
  ]_{t^{*}_{n}}^{ t} 
  + 2\int_{t^{*}_{n}}^{ t}\int_{\mathbb{R}^{3}} 
  \psi_{\mu} \sigma_{\mu} |\nabla v_{\xi}|^{2} 
  \le
  \\
  \le &
  \textstyle
  \int_{t^{*}_{n}}^{ t}
  \int_{\mathbb{R}^{3}} \psi_{\mu} 
  |v_{\xi}|^{2}(|\xi| \sigma_{\mu}^{2} 
  - k \dot{B}_{t^{*}_{n} , \mu} \sigma_{\mu})
  \\
  + &
  \textstyle
  \int_{t^{*}_{n}}^{ t}
  \psi_{\mu}
  \int_{\mathbb{R}^{3}} 
  \sigma_{\mu}^{2}
  (|v_{\xi}|^{3} + 2|P_{v_{\xi}}||v_{\xi}|
  +3 |v_{\xi}|^{2}|w_{\xi}|) 
  + 18 \sigma_{\mu} |v_{\xi}| | \nabla v_{\xi} | | w_{\xi} |. 
\end{split}
\end{equation}
Again, our goal is to prove an integral inequality 
for the functions 
\begin{equation}\nonumber
  a_{\mu}( t) := 
  \int_{\mathbb{R}^{3}} \sigma_{\mu}(y) 
  |v_{\xi} (t,y)|^{2} dy,
  \qquad
  B_{t^{*}_{n}, \mu}(t) := 
  \int_{t^{*}_{n}}^{t} \int_{\mathbb{R}^{3}} 
  \sigma_{\mu}(y) |\nabla v_{\xi}(\tau,y)|^{2} dy d\tau \, .
\end{equation}
We need to bound the terms in the right hand side
of \eqref{PutIn2}. Recalling the decomposition~\eqref{PressDec},
with the same computations of the second step, we obtain
\begin{equation}\nonumber
  I \leq 
  \frac{1}{5}\dot{B}_{t^{*}_{n},\mu} 
  + Z \dot{B}_{t^{*}_{n},\mu} a_{\mu}.
\end{equation}
While, using the SS and CKN
inequality, the last one with $2 \theta = 1+ 3/q$, that implies $r= 1/(1-\theta)$,
\begin{eqnarray}\nonumber
  II
  &  \le  &
  Z\|   w_{\xi}  \|_{L^{q}}
  \| \sigma_{\mu}^{1/2}  v_{\xi}  \|^{2(1-\theta)}_{L^{2}} 
  \| \sigma_{\mu}^{1/2}  \nabla  v_{\xi}  \|^{2\theta}_{L^{2}}
  =
  Z
  \|   w_{\xi}  \|_{L^{q}}
  a_{\mu}^{1-\theta}\dot{B}_{t^{*}_{n},\mu}^{\theta}
  \\ \nonumber
  & = &
  \|   w_{\xi}  \|_{L^{q}}
  (a_{\mu}\dot{B}_{t^{*}_{n},\mu})^{1-\theta}\dot{B}_{t^{*}_{n},\mu}^{2\theta-1}
   \le 
  \frac{1}{5}\dot{B}_{t^{*}_{n},\mu} 
  +
  Z \dot{B}_{t^{*}_{n},\mu}a_{\mu}
  +
  \frac{1}{15} \|  w_{\xi}  \|^{r}_{L^{q}}.
\end{eqnarray}
Exactly as in the second step
\begin{equation}\nonumber
  \textstyle
  \int_{\RT} \sigma_{\mu}^{2} |v_{\xi}|^{3} 
  \le
  \frac{1}{5}\dot{B}_{t^{*}_{n},\mu}  
  +Z
  \dot{B}_{t^{*}_{n},\mu}a_{\mu} \, ,
\end{equation}
while the next terms has to been estimated differently. Using CKN
\begin{align}\nonumber
  \textstyle
  |\xi|\int_{\RT} \sigma_{\mu}^{2} |v_{\xi}|^{2} 
  = 
  |\xi| \| \sigma_{\mu}  v_{\xi}   \|_{L^{2}}^{2}
  & \le
  \textstyle
  Z
  |\xi|  \| \sigma_{\mu}^{1/2}   \nabla v_{\xi}   \|_{L^{2}}
  \| \sigma_{\mu}^{1/2}    v_{\xi}   \|_{L^{2}}
  \\ \nonumber
  &
  =
  Z
  |\xi|
  (\dot{B}_{t^{*}_{n},\mu} a_{\mu})^{1/2} 
  \le |\xi|^{2} + 
  Z
  \dot{B}_{t^{*}_{n},\mu} a_{\mu}
\end{align}
and, still using
CKN with $2\theta = 1+3/q$, that implies $r= 1/(1-\theta)$,
\begin{eqnarray}\nonumber
   & 3  & 
  \textstyle 
  \int_{\RT}\sigma_{\mu}^{2} |v_{\xi}|^{2} |w_{\xi}| 
   \le 
  3 \|   w_{\xi}  \|_{L^{q}}
  \| \sigma_{\mu}  v_{\xi}  \|_{L^{2q/(q-1)}}^{2}
  \\ \nonumber
  & \le &
  \textstyle
  Z
  \|  w_{\xi}  \|_{L^{q}}
  \| \sigma_{\mu}^{1/2}  v_{\xi}  \|^{2(1-\theta)}_{L^{2}} 
  \| \sigma_{\mu}^{1/2}  \nabla v_{\xi}  \|^{2\theta}_{L^{2}}
    =  
  Z
  \|  w_{\xi}  \|_{L^{q}}
  a_{\mu}^{1-\theta}\dot{B}_{t^{*}_{n}, \mu}^{\theta}
  \\ \nonumber
  & = &
  \textstyle
  Z
  \|  w_{\xi}  \|_{L^{q}}
  (a_{\mu} \dot{B}_{t^{*}_{n}, \mu})^{1-\theta} \dot{B}_{t^{*}_{n}, \mu}^{2\theta -1}
  \le
  \frac{1}{5}\dot{B}_{t^{*}_{n},\mu} 
  +
  Z \dot{B}_{t^{*}_{n},\mu}a_{\mu}
  +
  \frac{1}{15} \|  w_{\xi}  \|^{r}_{L^{q}_{y}} \, .
\end{eqnarray}
By CKN with $\theta = 1-2/r$, that implies $r=2/(1-\theta)$, 
\begin{equation}\nonumber
\begin{split}
  \textstyle
  18 \int_{\RT} \sigma_{\mu} |v_{\xi}| |\nabla v_{\xi}| |w_{\xi}| 
  \le &
  18  \| \sigma_{\mu}^{1/2}  \nabla v_{\xi}  \|_{L^{2}}
  \|  w_{\xi}  \|_{L^{q}} 
  \| \sigma_{\mu}^{1/2}   v_{\xi}   \|_{L^{2q(q-2)}}
  \\
  \le &
  Z
  \| \sigma_{\mu}^{1/2}  \nabla v_{\xi}  \|_{L^{2}}
  \| w_{\xi}\|_{L^{q}}
  \| \sigma_{\mu}^{1/2}   \nabla v_{\xi}  \|^{\theta}_{L^{2}}
  \| \sigma_{\mu}^{1/2}  v_{\xi}  \|^{1-\theta}_{L^{2}}
  \\
  =&
  Z
  \|  w_{\xi}  \|_{L^{q}}
  a_{\mu}^{(1-\theta)/2}\dot{B}_{t^{*}_{n}, \mu}^{(1+\theta)/2}
   = 
  Z
  \|  w_{\xi}  \|_{L^{q}}
  (a_{\mu} \dot{B}_{t^{*}_{n}, \mu})^{(1-\theta)/2}\dot{B}_{t^{*}_{n}, \mu}^{\theta}
  \\ 
   \le &
  \frac{1}{5}\dot{B}_{t^{*}_{n},\mu} 
  +
  Z \dot{B}_{t^{*}_{n},\mu}a_{\mu}
  +
  \frac{1}{15} \|  w_{\xi}  \|^{r}_{L^{q}_{y}} \, .
\end{split}
\end{equation}
We can now plug these inequalities in (\ref{PutIn2})
so that
\begin{equation}\nonumber
\begin{split}
  a_{\mu}(t)
  &
  \psi_{\mu}(t)
  \textstyle
  -a_{\mu}(t^{*}_{n})
  +2\int_{t^{*}_{n}}^{t}
  \dot{B}_{t^{*}_{n},\mu}(s)\psi_{\mu}(s)ds
  \le
  \\
  \le
  &
  \textstyle
  \int_{t^{*}_{n}}^{t}
  \psi_{\mu}(s)
  [ \dot{B}_{t^{*}_{n},\mu}(s)
  + 6 Z \dot{B}_{t^{*}_{n}, \mu} a_{\mu}(s)
  +|\xi|^{2}
  + \frac{1}{5} \|  w_{\xi} \|^{r}_{L^{q}_{y}}(s)
  -k\dot{B}_{t^{*}_{n}, \mu} a_{\mu}(s)
  ] \ ds.
\end{split}
\end{equation}
Now we subtract the first term of the right hand side
from the left hand side and
choose $k=6Z$ in order to cancel each other out the second and the last term on the right hand side. Thus, noting  
\begin{equation}\nonumber
  \int_{t^{*}_{n}}^{t}
  \dot{B}_{t^{*}_{n},\mu}\psi_{\mu}
  =
  -\frac{1}{6Z}
  \int_{t^{*}_{n}}^{t}
  \dot{\psi}_{\mu}=
  \frac{\psi_{\mu}(t^{*}_{n})-\psi_{\mu}(t)}{6Z}
  =
  \frac{1-\psi_{\mu}(t)}{6Z}
\end{equation}
we have proved 
\begin{eqnarray}\label{LastInPreq}
   a_{\mu}(t)
  \psi_{\mu}(t)
  & - &
  a_{\mu}(t^{*}_{n})
  +
  \frac{1-\psi_{\mu}(t)}{6Z}
  \leq
  \\  \nonumber
  & \le &
  |\xi|^{2} \int_{t^{*}_{n}}^{t}
  \psi_{\mu}(s)ds+
  \frac{1}{5} \int_{t^{*}_{n}}^{t}
  \|  w_{\xi} \|^{r}_{L^{q}_{y}} (s) \psi_{\mu}(s) ds
  =: I + II \, ,
  \end{eqnarray}
for $t > t^{*}_{n}$.
Since $\psi_{\mu} \leq 1$, 
the term $I$ is immediately bounded by
\begin{equation}\label{Eq:IBound}
I \leq 
|\xi|^{2} t \, .
\end{equation}
The bound of $II$ is more delicate and requires 
the property (ii) in~\eqref{Def:RBar}.
Indeed, letting
\begin{equation}\nonumber
B_{t^{*}\!\!\!, \, \mu}(t) := \int_{t^{*}}^{t}\int_{\mathbb{R}^{3}} 
\sigma_{\mu}(y) | \nabla v_{\xi}(\tau,y)|^{2} dy d\tau \, ,
\end{equation}
the property (ii) becomes
\begin{equation}\nonumber
B_{t^{*}\!\!\!, \, \mu}(t) > \int_{t^{*}}^{t} \| w_{\xi} \|_{L^{q}_{y}}^{r}(\tau) d \tau,
\quad \text{for} \quad t > t^{*}
\end{equation}
%
%
and, since $B_{t^{*}_{n}, \mu}(t) \geq B_{t^{*}\!\!\!, \, \mu}(t)$, we can bound 
\begin{align}\nonumber
5 \, II 
& =
\int_{t^{*}_{n}}^{t}
  \|  w_{\xi} (s,\cdot) \|^{r}_{L^{q}(\Rtre)}  \psi_{\mu} (s) ds
   =
\int_{t^{*}_{n}}^{t}
  \|  w_{\xi} (s,\cdot) \|^{r}_{L^{q}(\Rtre)}  e^{-6Z B_{t^{*}_{n}, \mu}(s)}  ds  
  \\  \nonumber
  &   <
  \int_{t^{*}_{n}}^{t^{*}}
  \|  w_{\xi} (s,\cdot) \|^{r}_{L^{q}(\Rtre)}    ds 
  +
  \int_{t^{*}}^{t}
  \|  w_{\xi}(s,\cdot) \|^{r}_{L^{q}(\Rtre)} e^{ - 6 Z \int_{t^{*}}^{s} 
  \|  w_{\xi}(\tau ,\cdot) \|^{r}_{L^{q}(\Rtre)}  d\tau }  ds 
  \\   \nonumber
  &  =  
  \int_{t^{*}_{n}}^{t^{*}}
  \|  w_{\xi}(s,\cdot) \|^{r}_{L^{q}(\Rtre)}    ds 
  - \frac{1}{6Z}
  \left[   e^{-6Z \int_{t^{*}}^{s} \|  w_{\xi}(\tau ,\cdot) \|^{r}_{L^{q}(\Rtre)}  d\tau }   \right]_{s=t^{*}}^{s=t}
    \\   \nonumber
    &  <
    \int_{t^{*}_{n}}^{t^{*}}
  \|  w_{\xi}(s,\cdot) \|^{r}_{L^{q}(\Rtre)}    ds 
  +
  \frac{1}{6Z}.
\end{align}
Thus, since $ \|  w_{\xi}(s,\cdot) \|^{r}_{L^{q}} = \| w(s,\cdot) \|^{r}_{L^{q}}$ is integrable, we can choose  
$t^{*}_{n}$ close enough to $t^{*}$ in such a way that 
\begin{equation}\label{Eq:IIBound}
II \leq \frac{1}{30Z}.
\end{equation}
Plugging the estimates (\ref{Eq:IBound}), (\ref{Eq:IIBound}) into the inequality (\ref{LastInPreq}), we obtain
\begin{equation}\label{eq:LastIneqWithNu}
  \textstyle
  \left( a_{\mu}(t) - \frac{1}{6Z} \right) \psi_{\mu}(t)
  - a_{\mu}(t^{*}_{n})
  +\frac{2}{15Z}
  - |\xi|^{2} t
  \le
  0
\end{equation}
In order to handle the term $a_{\mu}(t^{*}_{n})$
we use
$a_{\mu}(t^{*}_{n}) \le Ze^{Z \mathcal{K} } \epsilon^{2}$, which we have
proved in \eqref{BoundBySmallnessBis}, and 
we assume
\begin{equation}\nonumber
  \epsilon \le 1
  \quad\implies\quad
  a_{\mu}(t^{*}_{n}) 
  \le
  Ze^{Z \mathcal{K} } \epsilon.
\end{equation}
In fact we assume that
$\epsilon$ is so small that
\begin{equation}\label{eq:smallep}
  a_{\mu}(t^{*}_{n}) \leq Z e^{Z \mathcal{K} } \epsilon \le \frac{1}{30Z},
\end{equation}
in such a way that \eqref{eq:LastIneqWithNu} gives
\begin{equation}\nonumber
 \textstyle
  \left(  a_{\mu}(t)  - \frac{1}{6Z}  \right)  \psi_{\mu}(t)
  + \frac{1}{10Z} 
  - |\xi|^{2} t
  \le 0 \, ,
\end{equation}
or equivalently (we recall that $\psi_{\mu}(t) =  e^{-6Z B_{t^{*}_{n}, \mu}(t)}$) 
\begin{equation}\label{eq:FinalIneq}
 \textstyle
  a_{\mu}(t) 
  + (\frac{1}{10Z} 
  - |\xi|^{2} t)
  e^{ 6Z B_{t^{*}_{n}, \mu}(t) }
  \le
  \frac{1}{6Z}  \, ,
\end{equation}
%
%
%
%
for all $t > t^{*}$.
Let us finally assume 
\begin{equation}\label{eq:condxi2}
  \textstyle
  (\frac{1}{10Z}
    - |\xi|^{2} t)>0
  \quad {\it i.e.} \quad 
  |\xi|^{2}t<\frac{1}{10 Z}.
\end{equation}
Note that this assumption is stronger than 
\eqref{ResrtCondFirst}, namely
$|\xi|^{2}t^{*}\le 1$, since $Z\ge1$ and $t^{*} < t$.
The inequality \eqref{eq:FinalIneq} immediately 
gives 
$$
B_{t^{*}\!\!\!, \, \mu}(t) \leq B_{t^{*}_{n}, \mu}(t) < \infty,
\qquad \mbox{uniformly in $\mu >0$} \, ,
$$
for $t > t^{*}$ such that (\ref{eq:condxi2}) holds. Here we mean that $B_{t^{*}_{n}, \mu}(t)$ is bounded by
a constant independent on $\mu > 0$. This constant may become arbitrarily large as~$|\xi|^{2}t$ approaches $\frac{1}{10 Z}$, but this does 
not affect our argument.
On the other hand, recalling~\eqref{Prop:BB(t)BeforeRBar}, we also know that 
$$
B_{\mu}(t^{*}) = \int_{0}^{t^{*}} \int_{\mathbb{R}^{3}}  \sigma_{\mu}(y) |\nabla v_{\xi}(\tau,y)|^{2} dy d\tau  <  \infty  ,
\qquad \mbox{uniformly in $\mu >0$} \, ,
$$
Thus, since
$B_{\mu}(t) = B_{\mu}(t^{*}) + B_{t^{*}\!\!\!, \, \mu}(t)$,
we have proved
$$
B_{\mu}(t) < \infty, \qquad \mbox{uniformly in $\mu >0$}  \, ,
$$
for $t > t^{*}$ such that (\ref{eq:condxi2}) holds.
Since the weights~$\sigma_{\mu}(y)$ are increasing as~$\mu \to 0$, and they converges to~$|y|^{-1}$,  
we can pass to the limit and
come back to the old variables~($t,x$), so that 
\begin{equation}\label{FinalBound1}
 \lim_{\mu \to 0} B_{\mu}(t) = \int_{0}^{t} \int_{\mathbb{R}^{3}} |y|^{-1}|\nabla v_{\xi}(\tau, y)|^{2} \  dy d \tau     
   = \int_{0}^{t} \int_{\mathbb{R}^{3}}
  \frac{|\nabla v(\tau, x)|^{2}}{|x - \xi \tau|} \ dx d \tau    < \infty \, ,
\end{equation}
provided that \eqref{eq:condxi2} is satisfied.

This implies that the regularity condition (\ref{CKNConditionSB}) is satisfied
at any $(s,\xi s)\in L(t,\xi)$. 
Indeed,
let $0 < s < t $, and let $r > 0$ be so small that  
$0 < s - 7r^{2}/8  < s + r^{2}/8  < t$
and $|\xi| r \leq 1$. For all $(\tau , x) \in Q^{*}_{r}(s, \xi s)$
\begin{equation}\nonumber
  |x - \xi \tau | \le
  |x-\xi s|+|\xi||s-\tau|\le
  r + r^{2} |\xi| \leq 2r,
\end{equation}
from which we deduce
\begin{equation}\label{ContinuityOTIF}
  \frac{1}{r} \int \int_{Q^{*}_{r}(s, \xi s)} 
  |\nabla v(\tau , x)|^{2} \  dx d \tau  
  \le
  2 \int_{s - \frac{7}{8}r^{2}}^{s 
  + \frac{1}{8}r^{2}} \int_{\mathbb{R}^{3}}
  \frac{|\nabla v (\tau, x)|^{2}}{|x - \xi \tau|} \ dx d \tau \, .  
\end{equation}

Using this and (\ref{FinalBound1}) is clear that the quantity on the left hand side 
converges to zero as $r \rightarrow 0$. On the other hand, we already know that 
$$
\limsup_{r \rightarrow 0} \frac{1}{r} \int \int_{Q^{*}_{r}(s, \xi s)} 
|\nabla w (\tau , x)|^{2} \  dx d \tau  < \varepsilon^{*} \, ,
$$
since this is one of the
requirement to be a generalized reference solution.
Thus~$u = v+w$ satisfies the regularity condition (\ref{CKNConditionSB}) at any point
$(s,\xi s)\in L(t,\xi)$ with $(t,\xi)$ satisfying \eqref{eq:condxi2}.
This implies, by Lemma \ref{Lemma:CKNLEmma}, the regularity of $L(t,\xi)$.

\subsection{Conclusion of the proof}

Summing up we have shown that
there exists a 
constant $Z \geq1$, that only depends on~$(r,q)$, such that the following holds:
if $\epsilon$ is sufficiently small enough to satisfy \eqref{eq:smallep},
then the segment $L(t,\xi)$ is a regular set
for $u$, for any $\xi\in \mathbb{R}^{3}$ and $t > 0$ such that \eqref{eq:condxi2} holds.
If we set
\begin{equation*}
  \delta_{0}=\frac{1}{30Z^{2}},
\end{equation*}
then \eqref{eq:smallep} follows by
\begin{equation}\label{eq:smallepdel}
    \epsilon\le \delta_{0} e^{-\mathcal{K}/\delta_{0}}
\end{equation}
and \eqref{eq:condxi2} follows  by
$|\xi|^{2}t< \delta_{0}$, which is equivalent to
$t > \frac{|t \xi|^{2}}{\delta_{0}}$,
namely
\begin{equation}\nonumber
  \textstyle
  (t,t \xi)\in\Pi_{\delta_{0}},
  \qquad
  \Pi_{\delta_{0}}:=
  \left\{ (t,x)\in (0,\infty)\times \mathbb{R}^{3}
  \colon t>\frac{|x|^{2}}{\delta_{0}}  \right\}.
\end{equation}
Thus $L(t, t\xi)$ is regular 
provided that $\epsilon$ satisfies 
\eqref{eq:smallepdel} and $(t, t \xi) \in \Pi_{\delta_{0}}$. 
As a consequence, we conclude that the paraboloid $\Pi_{\delta_{0}}$,
that is the union of these segments for arbitrary $t > 0$ and $\xi \in \Rtre $,  
is a regular 
set for $u$
provided that \eqref{eq:smallepdel}, namely our smallness assumption~\eqref{MTSAMoreGen}, holds.
This concludes the proof.

\end{proof}

Here we prove that reference solutions are generalized reference solutions.

\begin{lemma}\label{Lemma:Postposto}
If~$w$ is a reference solution of size~$\mathcal{K}$ to the Navier--Stokes equation (see Definition~\ref{Def:RefSol}), then
it is also a generalized reference solution of size~$\mathcal{K}$ (see Definition~\ref{Def:GenRefSol}).  
\end{lemma}

\begin{proof}

Recalling Remark~\ref{Rem:Uniqueness} and~\eqref{RecoverP}, we only 
need to show that the regularity 
condition~\eqref{CKNConditionSB} is satisfied at any~$(t,x)\in (0,\infty)\times \Rtre$.
Recalling the argument at the end of the previous proof (after the inequality~(\ref{FinalBound1})) it 
suffices to prove
\begin{equation}\label{FinalBound2}
  \int_{0}^{t} \int_{\mathbb{R}^{3}}
  \frac{| \nabla w (\tau ,x)|^{2} }{|x - \xi \tau |} 
  \ dx d \tau   < \infty, 
\qquad   \mbox{for all $t > 0$ and $\xi \in \Rtre$}
\end{equation}

By translation invariance we can assume $x' = 0$, namely $w_{0} \in L^{2}(\Rtre) \cap L^{2}(|x|^{-1}dx)$. 
We change variables
\begin{equation}\nonumber
  (t, y) = (t, x - \xi t),
  \quad
    w_{\xi}(t,y) = w(t,x) \, ,
\end{equation}
so that the local energy inequality becomes
\begin{equation}\label{eq:energyw}
  \begin{split}
    \textstyle
    \int_{\mathbb{R}^{3}} 
    |w_{\xi}|^{2}  
    &
    \textstyle
     \phi (t,y) dy + 
    2 \int_{0}^{t}\int_{\mathbb{R}^{3}}  |\nabla  w_{\xi}|^{2} \phi 
    \le
    \int_{\mathbb{R}^{3}}  |w_{0}|^{2} \phi(0,y) dy \,
    +
    \\
    & 
    +
    \textstyle
      \int_{0}^{t} \int_{\mathbb{R}^{3}} 
    |w_{\xi}|^{2} (\phi_{t} - \xi \cdot \nabla \phi + \Delta \phi)
   +
    (|w_{\xi}|^{2} + 
    2P_{w_{\xi}})w_{\xi} \cdot \nabla \phi \, ,
  \end{split}
\end{equation}
where~$P_{w_{\xi}} = \mathcal{R} \otimes \mathcal{R} \cdot (w_{\xi} \otimes w_{\xi})$. We choose
$\phi(t,y):= \sigma_{\mu}(y) \chi(\delta |y|)$
where~$\sigma_{\mu}$ and~$\chi$
are as in the second step of the previous proof.
Exactly as before, taking the limit~$\delta \rightarrow 0$
and, using the inequalities (\ref{eq:sigmapr}),
(\ref{eq:energyw}) becomes
\begin{equation}\label{r.h.s.Easy}
  \textstyle
  \left[ \int_{\mathbb{R}^{3}} \sigma_{\mu} 
  |w_{\xi}|^{2} \right]_{0}^{t} + 
  2\int_{0}^{t}\int_{\mathbb{R}^{3}} 
  \sigma_{\mu} |\nabla w_{\xi}|^{2} 
  \le
  |\xi|\int_{0}^{t}\int_{\mathbb{R}^{3}}  
  \sigma_{\mu}^{2} |w_{\xi}|^{2}  
   +  \int_{0}^{t}\int_{\mathbb{R}^{3}} \sigma_{\mu}^{2}
   (|w_{\xi}|^{3} + 2|P_{w_{\xi}}||w_{\xi}| ).
\end{equation}
We desire an integral inequality
for the functions 
\begin{equation}\nonumber
a_{\mu}(t) = \int_{\mathbb{R}^{3}} \sigma_{\mu}(y) |w_{\xi}(t,y)|^{2} dy, 
\qquad
B_{\mu}(t) = \int_{0}^{t}\int_{\mathbb{R}^{3}} 
\sigma_{\mu}(y) | \nabla w_{\xi}(\tau,y)|^{2}dy d\tau .
\end{equation}

To bound the pressure term we use the 
SS and CKN inequalities, the last one with $2\theta = 1+3/q$, that implies $r=1/(1-\theta)$, so that
\begin{eqnarray}\nonumber
  \textstyle
  2 \int_{\mathbb{R}^{3}} \sigma_{\mu}^{2}
  |P_{w_{\xi}}||w_{\xi}| 
  & = & 
  \textstyle
  2 \int_{\mathbb{R}^{3}} \sigma_{\mu}^{2}
  |w_{\xi}| | \mathcal{R} \otimes \mathcal{R} \cdot  
  ( w_{\xi} \otimes w_{\xi} ) |
  \\ \nonumber
  & \leq  & 
  \| \sigma_{\mu}
  \mathcal{R} \otimes \mathcal{R} \cdot  
  ( w_{\xi} \otimes w_{\xi} )   \|_{L^{2q/(q+1)}}
  \| \sigma_{\mu}
  w_{\xi} \|_{L^{2q/(q-1)}}
  \\ \nonumber
  & \lesssim &
  \| \sigma_{\mu}
  | w_{\xi} |^{2}   \|_{L^{2q/(q+1)}}
  \| \sigma_{\mu}
  w_{\xi} \|_{L^{2q/(q-1)}}
  \\ \nonumber
  & \leq &
  \|  w_{\xi}    \|_{L^{q}}
  \| \sigma_{\mu}
  w_{\xi} \|_{L^{2q/(q-1)}}^{2}
  \\ \nonumber
  & \lesssim &
  \|  w_{\xi}    \|_{L^{q}}
  \| \sigma_{\mu}^{1/2}
  \nabla w_{\xi} \|_{L^{2}}^{2\theta}
  \| \sigma_{\mu}^{1/2}
  w_{\xi} \|_{L^{2}}^{2(1-\theta)}
  \\ \label{Appendx:FirstIn}
  & = &
  \|  w_{\xi}    \|_{L^{q}}
  \dot{B}^{\theta}_{\mu}a_{\mu}^{1-\theta}
  \le
  \textstyle
  \frac13  \dot{B}_{\mu}  
   +  C \| w_{\xi} \|_{L^{q}}^{r}  a_{\mu} .
\end{eqnarray} 
In a similar way, using CKN with $2 \theta = 1 + 3/q$
\begin{eqnarray}
  \textstyle
  \int_{\mathbb{R}^{3}} \sigma_{\mu}^{2}
  |w_{\xi}|^{3}
  & \leq &
  \| w_{\xi} \|_{L^{q}}
  \| \sigma_{\mu}^{2}  |w_{\xi}|^{2} \|_{L^{q/(q-1)}}
  \\ \nonumber
  &  =  &
  \| w_{\xi} \|_{L^{q}}
  \| \sigma_{\mu}  w_{\xi}  \|_{L^{2q/(q-1)}}^{2}
  \\ \nonumber
  & \lesssim &
  \|  w_{\xi}    \|_{L^{q}}
  \| \sigma_{\mu}^{1/2}
  \nabla w_{\xi} \|_{L^{2}}^{2\theta}
  \| \sigma_{\mu}^{1/2}
  w_{\xi} \|_{L^{2}}^{2(1-\theta)}
  \\ \label{Appendx:SecondIn}
  &  \le  &
  \textstyle
  \frac{1}{3}  \dot{B}_{\mu}  
  +  C \| w_{\xi} \|_{L^{q}}^{r}  a_{\mu} \, ,
\end{eqnarray}
and, again by CKN, 
\begin{equation}\nonumber
  \textstyle
  |\xi| \int_{\RT} \sigma_{\mu}^{2} |w_{\xi}|^{2} 
  \lesssim 
  |\xi|\cdot  \| \sigma_{\mu}^{1/2} 
  \nabla w_{\xi}  \|_{L^{2}}
  \| \sigma_{\mu}^{1/2} w_{\xi} \|_{L^{2}} 
  =
  |\xi|   (\dot{B}_{\mu} a_{\mu})^{1/2} 
  \leq \frac13   \dot{B}_{\mu} + C |\xi|^{2} a_{\mu} .
\end{equation}
Using these, the inequality~(\ref{r.h.s.Easy})
becomes
\begin{equation}\nonumber
 \textstyle
  a_{\mu}(t) + B_{\mu}(t) \leq 
  a_{\mu}(0) +
  C \int_{0}^{t}  \left( 
   |\xi|^{2} + 
  3  \|  w_{\xi}(\tau,\cdot)  \|^{r}_{L^{q}(\Rtre)}
  \right) a_{\mu}(\tau) \ d\tau ,
\end{equation}
for some $C>0$. 

Letting $a(t) := \int_{\mathbb{R}^{3}} |y|^{-1} |w_{\xi}(t,y)|^{2} dy$ and using 
$a_{\mu}(0) \leq a(0)$,
we arrive to
\begin{equation}\label{AForB}
 \textstyle
  a_{\mu}(t) + B_{\mu}(t) \leq 
  a(0)  +
  C\int_{0}^{t}  \left( 
  |\xi|^{2} + 
   \|  w_{\xi} (\tau,\cdot) \|^{r}_{L^{q}(\Rtre)}
  \right) a_{\mu}(\tau) \ d\tau \, .
\end{equation}
Since $a(0) = \| |x|^{- 1/2} w_{0} \|^{2}_{L^{2}_{x}}$
and~$\|w_{\xi} \|_{L^{r}_{t}L^{q}_{y}}^{r} = \|w\|_{L^{r}_{t}L^{q}_{x}}^{r} =: \mathcal{K}$,  
the Gr\"onwall inequality gives
$$
a_{\mu}(t) < a(0)e^{C(t|\xi|^{2} + \mathcal{K})} \quad 
\mbox{for all}
\quad t > 0 \, .
$$ 
 Plugging this into the right hand side of~(\ref{AForB}), since~$\|w_{\xi}\|_{L^{r}_{t}L^{q}_{x}}^{r} =: \mathcal{K}$, we
find out that~$B_{\mu}(t)$ is bounded, for any time~$t>0$, uniformly in~$\mu$. 
Thus, since the 
weights~$\sigma_{\mu}(y)$ are increasing as~$\mu \to 0$, we can pass to the limit
\begin{equation}\nonumber
\lim_{\mu \to 0}  B_{\mu}(t) = \int_{0}^{t} \int_{\mathbb{R}^{3}} |y|^{-1} 
  |\nabla w_{\xi}(\tau,y)|^{2} \ dyd\tau < \infty, \quad \mbox{for all} \quad t > 0.
\end{equation}
In particular, coming back to the $(t,x)$ variables, we see that \eqref{FinalBound2} is satisfied.
This concludes the proof.

\end{proof}

\section{Proof of Propositions~\ref{prop:Axysimm}\,-\,\ref{Prop:Beltr} }\label{Sec:ApplicProof}

\subsection*{Proof of Proposition \ref{prop:Axysimm}}

Let $w$ be the solution to the Navier--Stokes equation constructed in~\cite{Leon}, for the 
zero swirl initial data~$w_0$.
This solution 
satisfies the
energy inequality~(\ref{ClassicEnergy}), in particular it
belongs to~$L^{\infty}_{t}L^{2}_{x} \cap L^{2}_{t}\dot{H}^{1}_{x}$, and moreover 
\begin{equation}\label{LeonLemma}
\| w \|_{L^{\infty}_{t}\dot{H}^{1}_{x}} \lesssim  1 + \| w_{0} \|^{5/3}_{H^{2}_{x}}
\end{equation}
Indeed, as we will show at the end of the proof, this can be deduced following the proof of Lemma~5 in~\cite{Leon}.  
Thus $w \in L^{4}_{t}\dot{H}^{1}_{x}$, by interpolation, and 
$w \in L^{4}_{t}L^{6}_{x}$, by Sobolev embedding.
More precisely 
\begin{equation}
\| w \|_{L^{4}_{t}L^{6}_{x}} \lesssim  1+ \| w_{0} \|^{4/3}_{H^{2}_{x}} \, .
\end{equation}
Since~$(r,q)=(4,6)$ is an admissible couple, we have shown that~$w$ is 
a reference solution of size~$\mathcal{K} \leq \mathcal{C} ( 1 + \| w_{0} \|^{16/3}_{H^{2}}$), for some 
absolute constant~$\mathcal{C} > 1$; see Definition~\ref{Def:RefSol}. 
The statement then follows by 
Theorem~\ref{MainThmPerturbation},
taking $\delta_{2} = \mathcal{C}^{-1} \delta_{0}$.

It remains to prove (\ref{LeonLemma}).
First of all we notice that, since $\nabla \cdot w = 0 $, we have, by orthogonality, 
\begin{equation}\label{Curl=Nabla}
\| \nabla w \|_{L^{2}_{x}} = \| \nabla \times w \|_{L^{2}_{x}} \, .
\end{equation} 
Thus, we can equivalently show
\begin{equation}\label{LeonLemmaCurl}
\| \nabla \times w \|_{L^{\infty}_{t}L^{2}_{x}} \lesssim  1+ \| w_{0} \|^{5/3}_{H^{2}_{x}}
\end{equation}
Following the proof of Lemma 5  in~\cite{Leon},
we have  
\begin{align}\nonumber
\| \nabla \times w \|^{2}_{L^{2}_{x}}(t) & + k \int_{0}^{t} \| \nabla^{2} w \|^{2}_{L^{2}_{x}} (s) \ d s
\\
&
\leq \| \nabla \times w_{0} \|^{2}_{L^{2}_{x}}  +
\int_{0}^{t} \|  w \|_{L^{\infty}_{x}} \Big\| \frac{\nabla \times w}{\mathbf{r}}\Big\|_{L^{2}_{x}} \| \nabla \times w \|_{L^{2}_{x}}(s) \ d s \, ,
\end{align}
for a certain constant $k > 0$. We recall that $\mathbf{r}$ is the radial variables in a cylindrical polar coordinate system; see~\eqref{Def:AxySimmVF}. 
Thus, using the 
inequality~$\| w \|_{L^{\infty}(\Rtre)} \lesssim \| \nabla w \|_{L^{2}(\Rtre)}^{1/2} \| \nabla^{2} w \|_{L^{2}(\Rtre)}^{1/2}$, $\eqref{Curl=Nabla}$,
and the Young inequality,
\begin{align}\label{FinalLeonardi}
& 
\| \nabla \times w \|^{2}_{L^{2}_{x}}(t)
 +  
k \int_{0}^{t} \| \nabla^{2} w \|^{2}_{L^{2}_{x}}(s) \ d s
\\ \nonumber
& \leq  \| \nabla \times w_{0} \|^{2}_{L^{2}_{x}} +
\int_{0}^{t} 
\| \nabla^{2} w \|_{L^{2}_{x}}^{1/2} \,  
\Big\| \frac{\nabla \times w}{\mathbf{r}}\Big\|_{L^{2}_{x}}  
\, 
\| \nabla  w \|^{3/2}_{L^{2}_{x}}(s) d s.
\\ \nonumber
&  \leq \| \nabla \times w_{0} \|^{2}_{L^{2}_{x}} +
\frac{k}{2} \int_{0}^{t} \| \nabla^{2} w \|^{2}_{L^{2}_{x}} (s) \ d s 
+
C \, \Big\| \frac{\nabla \times w}{\mathbf{r}}\Big\|^{4/3}_{L^{\infty}_{t}L^{2}_{x}}
\int_{0}^{t}   \| \nabla  w \|^{2}_{L^{2}_{x}} (s) \ d s \, ,
\end{align}
for some $C$ that only depends on $k$.
Since $\int_{0}^{t}   \| \nabla w \|^{2}_{L^{2}_{x}}(s)ds \leq \| w_{0} \|^{2}_{L^{2}_{x}}$ and, by~\cite[Lemma 1.2]{Yud}, \cite[Lemma 3(ii)]{Leon},  
$$
\Big\| \frac{\nabla \times w}{\mathbf{r}}\Big\|_{L^{\infty}_{t}L^{2}_{x}}
\leq
\Big\| \frac{\nabla \times w_{0}}{\mathbf{r}}\Big\|_{L^{2}_{x}}
 \lesssim
 \| w_{0} \|_{H^{2}_{x}}
 \, ,
 $$
\eqref{FinalLeonardi} 
implies~(\ref{LeonLemmaCurl}) and the proof is concluded.

 \hfill $\Box$

\subsection*{Proof of Proposition~\ref{prop:CGPert} }

Recaling Theorem 1 in \cite{ChemGall}, given $\sigma$ sufficiently small in the assumption~\eqref{GallChemSmallnessCond}, there exists a unique solution~$w$ to the Navier--Stokes equation 
with initial data $w_0$. 
As the authors point out, in order to prove this fact it is actually sufficient assume~$w_{0} \in \dot{B}_{\infty,2}^{-1}(\Rtre)$.
On the other hand, the extra assumption~$w_{0} \in \dot{H}^{1/2}(\Rtre)$ ensures that~$w$
belongs to~$\in C_{b}([0,\infty); \dot{H}^{1/2}_{x})$, via a standard propagation of regularity argument; see~\cite[Theorem 18.3]{Lem}.
This also implies~$w \in L^{2}_{t}\dot{H}^{3/2}_{x}$; see~\cite[Theorem 2]{GIP2}. Thus, by interpolation, we see that $w$
belongs to~$\in L^{r}_{t}\dot{H}^{s}_{x}$
provided~$r \geq 2$ and $s = 1/2 + 2/r$ and, by Sobolev embedding, we have~$w \in L^{r}_{t}L^{q}_{x}$ 
for every admissible couple $(r,q)$ provided $q \neq \infty$. 
(we recall that $(r,q)$ is admissible when $2/r+3/q =1$).
Since we have required $w_{0} \in L^{2}(\Rtre)$, we also have $w \in C_{b}([0,\infty); L^{2}(\Rtre))$, again by the propagation of regularity argument.   
In conclusion, $w$ is a reference solution of size~$\mathcal{K} = \| w \|^{r}_{L^{r}_{t}L^{q}_{x}}$, see Definition~\ref{Def:RefSol}, and the 
statement follows by 
Theorem~\ref{MainThmPerturbation}.

 \hfill $\Box$

\subsection*{Proof of Proposition~\ref{Prop:2DNew}}

The statement follows combining the forthcoming Propositions~\ref{prop:2dPreq} and~\ref{InfiniteExistenceNew}. 

\begin{proposition}\label{prop:2dPreq}
There exists a constant $\delta_{0} >0$ such that the following holds.
Let $W_0$ be a 2D divergence free vector field which belongs to $L^{2}(\mathbb{R}^{2}) \cap L^{1}(\mathbb{R}^{2})$ 
and let $W$ be
the (unique) solution to
the 2D Navier--Stokes equation with initial data $W_{0}$.
The set $\Pi_{\delta_{0},\bar{x}}$
is regular for
any suitable weak solution $u$ to the 3D Navier--Stokes equation, with divergence free initial data~$u_{0} \in L^{2}_{loc}(\Rtre)$,
such that: 
\begin{enumerate}
\item $u_{0} - \widetilde{W_{0}} \in L^{2}(\Rtre)$;
\item $u - \widetilde{W} \in L^{\infty}_{t}L^{2}(\Rtre) \cap L^{2}_{t}\dot{H}^{1}(\Rtre)$ is a suitable weak solution to the 
perturbed 3D Navier--Stokes equation~\eqref{ModProbInPrelimin}, around the solution~$\widetilde{W}$, 
with data $u_{0} - \widetilde{W_{0}}$;
\end{enumerate}
and
\begin{equation}\nonumber
    \| |x-\bar{x}|^{-1/2} (u_{0} - \widetilde{W_{0}})\|_{L^{2}(\Rtre)}\le \delta_{0} e^{- \delta_{0}^{-1} \| W \|^{2}_{L^{2}_{t}L^{\infty}(\mathbb{R}^{2})}}.
\end{equation}
\end{proposition}

It is not immediately clear that suitable weak solutions which satisfy the condition (2) of the statement actually exist.  
In Proposition~\ref{InfiniteExistenceNew} we will prove that this is indeed
the case, for any initial datum $u_{0}\in L^{2}_{loc}(\Rtre)$, as long as (1) is 
satisfied and~$\widetilde{W} \in L^{2}_{t}L^{\infty}_{x}$; see~\eqref{TIITC}.  
Thus, this Proposition has to be considered the completion of Proposition \ref{prop:2dPreq} and, for the same reason, of the 
forthcoming Proposition~\ref{Prop:BeltrPreq}.

\subsubsection*{Proof of Proposition \ref{prop:2dPreq}}
The 2D solution $W$ belongs to~$L^{\infty}_{t}L^{2}(\mathbb{R}^{2}) \cap L^{2}_{t}\dot{H}^{1}(\mathbb{R}^{2})$.
This is a well know fact, indeed it is unique in the class of 
solutions which satisfies the 2D energy inequality, and thus it is, in particular, suitable.
Moreover, as observed in the proof of Theorem 4 in~\cite{Ponce}, $W$ also belongs to~$\in L^{2}_{t}L^{\infty}(\mathbb{R}^{2})$.
By the definition of $\widetilde{W}$, we immediately have 
\begin{equation}\label{TIITC}
\widetilde{W} \in L^{2}_{t}L^{\infty}(\mathbb{R}^{3}) \cap L^{\infty}_{t}L^{2}(K) \cap L^{2}_{t}\dot{H}^{1}(K) \, , 
\end{equation}
for any compact set $K \subset \Rtre$.
Since~$(r,q)=(2,\infty)$ is admissible (namely $2/r+3/q =1$), in order to show that~$\widetilde{W}$ is a generalized reference 
solution (see Definition~$\ref{Def:GenRefSol}$), we need to 
check that it is suitable, that satisfies the regularity condition~$\eqref{CKNConditionSB}$ at any $(t,x) \in (0,\infty)\times \Rtre$, that
$t \to \widetilde{W}(t) \in C_{b}([0,\infty); L^{2}(K))$, $K \subset \Rtre$ compact, and that the representation formula for the pressure is satisfied. 
The suitability follows straightforwardly by that of~$W$. We omit the obvious details. We only notice that  
given any 
$\phi \in C^{\infty}_{c}( (0,\infty) \times \mathbb{R}^{3})$,
we can define a family of 2D functions $\phi_{x_{3}}(x_{1},x_{2}) := \phi(x_{1},x_{2},x_{3})$, so that $W$ solves the 2D 
Navier--stokes equation, in the 
weak sense, when we test against $\phi_{x_{3}}(x_{1},x_{2})$, and similarly, 
given $0 \leq \phi \in C^{\infty}_{c}( \mathbb{R} \times \mathbb{R}^{3})$,       
the local energy inequality
\begin{eqnarray}\label{Eq:Skip2d}
    & & 
    \textstyle
    \int_{\mathbb{R}^{2}} |W|^{2} \phi_{x_{3}} (t) + 2 \int_{0}^{t}\int_{\mathbb{R}^{2}} | \nabla W |^{2} \phi_{x_{3}}
    \\ \nonumber
        & \leq & 
        \textstyle
        \int_{\mathbb{R}^{2}} |W_{0}|^{2} \phi_{x_{3}} (0) + \int_{0}^{t} \int_{\mathbb{R}^{2}} |W|^{2} 
    (\partial_{t} \phi_{x_{3}} + \Delta \phi_{x_{3}}) 
    +\int_{0}^{t}\int_{\mathbb{R}^{2}} (|W|^{2} + 2P_{W})W \cdot \nabla \phi_{x_{3}} \, ,
    \end{eqnarray}
holds, where the space integration is over $(x_{1}, x_{2}) \in \mathbb{R}^{2}$.
Thus, integrating over $x_{3} \in \mathbb{R}$, we see that $\widetilde{W}$ satisfies the weak 3D equation, 
letting $P_{\widetilde{W}} = \widetilde{P_{W}}$, and
\begin{eqnarray}\label{Eq:Skip2dTilde}
    & & 
    \textstyle
    \int_{\mathbb{R}^{3}} |\widetilde{W}|^{2} \phi (t) + 2 \int_{0}^{t}\int_{\mathbb{R}^{3}} | \nabla \widetilde{W} |^{2} \phi
    \\ \nonumber
        & \leq & 
        \textstyle
        \int_{\mathbb{R}^{3}} |\widetilde{W_{0}}|^{2} \phi (0) + \int_{0}^{t} \int_{\mathbb{R}^{3}} |\widetilde{W}|^{2} 
    (\partial_{t} \phi + \Delta \phi) 
    +\int_{0}^{t}\int_{\mathbb{R}^{3}} (|\widetilde{W}|^{2} + 2 P_{\widetilde{W}})\widetilde{W} \cdot \nabla \phi \, ;
    \end{eqnarray}
  here we have used $\widetilde{W}_{3} = (\widetilde{W_{0}})_{3}  =0$,
     $\partial_{i}\widetilde{W}_{j} =0$ if $i=3$ or if $j=3$, and 
     $\partial_{3} P_{\widetilde{W}}=0$. 
    Similarly, since $W \in C_{b}([0,\infty); L^{2}(\Rtre))$ (see  \cite[Proposition 14.3]{Lem}), integrating with respect to $x_{3}$, 
    we see that $t \to \widetilde{W}(t) \in C_{b}([0,\infty); L^{2}(K))$, for any compact set~$K \subset \Rtre$.
    Moreover, since~$P_{W} = \sum_{i,j=1}^{2} \mathcal{R}_{i} \mathcal{R}_{j} W_{i} W_{j}$ and $W_{3} =0$,  
    the representation formula for the 3D pressure
    $P_{\widetilde{W}} = \sum_{i,j=1}^{3} \mathcal{R}_{i} \mathcal{R}_{j} \widetilde{W}_{i} \widetilde{W}_{j}$ 
    is valid. Indeed, letting $\mathcal{F}$ the Fourier transform, we have
    \begin{align}\nonumber
    \mathcal{F} 
    (\sum_{i,j=1}^{3}
     \mathcal{R}_{i} \mathcal{R}_{j}  \widetilde{W}_{i} \widetilde{W}_{j})
    & 
    =  - \sum_{i,j=1}^{3} 
    \frac{\xi_{i}\xi_{j}}{\xi_{1}^{2}+\xi_{2}^{2}+\xi_{3}^{2}} \mathcal{F} ( \widetilde{W}_{i} \widetilde{W}_{j})
    =
    - \sum_{i,j=1}^{3}
    \frac{\xi_{i}\xi_{j}}{\xi_{1}^{2}+\xi_{2}^{2}+\xi_{3}^{2}} \mathcal{F} ( \widetilde{W_{i} W_{j}})
    =
    \\  \nonumber
    &
    = 
    - \delta_{\xi_{3}=0} \sum_{i,j=1}^{2} 
    \frac{\xi_{i}\xi_{j}}{\xi_{1}^{2}+\xi_{2}^{2}} \mathcal{F} ( W_{i} W_{j}) 
    = \delta_{\xi_{3}=0} \mathcal{F} (\sum_{i,j=1}^{2} \mathcal{R}_{i} \mathcal{R}_{j} W_{i} W_{j} )
    =
    \\ \nonumber
    &
    = \delta_{\xi_{3}=0} \mathcal{F} (P_{W}) =  \mathcal{F} (\widetilde{P_{W}}) = \mathcal{F} (P_{\widetilde{W}}) \, .
    \end{align}
    We remark that there is no trouble to make sense of $\mathcal{R}_{i} \mathcal{R}_{j} \widetilde{W}_{i} \widetilde{W}_{j}$ as long as
    $\widetilde{W}_{i} \widetilde{W}_{j}(t)$ is bounded, see for instance~\cite[Remark 8.1.18]{Graf}, and so for almost 
    any~$t \in (0,\infty)$, since~$\widetilde{W}$ belongs to $\in L^{2}_{t}L^{\infty}(\mathbb{R}^{3})$.
    
Finally, by definition of $\widetilde{W}$, one see that  
   \begin{equation}\nonumber
     \int_{B(x,r)} |\nabla \widetilde{W}|^{2} \leq 2r \int_{\mathbb{R}^{2}} |\nabla W|^{2} \, , 
    \end{equation}
so that
  \begin{eqnarray}\nonumber
    \frac{1}{r}
     \int \int_{Q^{*}_{r}(t,x)} |\nabla \widetilde{W}|^{2} 
      \leq  2
    \int_{t- 7r^{2}/8}^{t+r^{2}/8} \int_{\mathbb{R}^{2}} |\nabla W|^{2}(s, x_{1}, x_{2}) dx_{1}dx_{2} ds \to 0
     \quad 
    \mbox{as} 
    \quad 
    r \to 0 \, ,
\end{eqnarray}
because of the (space-time) integrability of $|\nabla W|^{2}$.

In conclusion, we have shown that 
$\widetilde{W}$ is a generalized reference solution of size~$\mathcal{K} = \| W \|^{2}_{L^{2}_{t}L^{\infty}(\mathbb{R}^{2})}$,
so that
the statement follows by Theorem~\ref{MainThmPerturbationMoreGen}.

 \hfill $\Box$


\subsection{Proof of Proposition~\ref{Prop:Beltr} }

The statement follows combining the forthcoming Propositions~\ref{Prop:BeltrPreq} and~\ref{InfiniteExistenceNew}. 
Notice that, given a Beltrami field $w_{0} \in L^{\infty}(\Rtre)$, the family of rescaled Beltrami fields~$t \to e^{-t \lambda^{2}} w_{0}$ 
solves the Navier--Stokes equation and
belongs to~$L^{2}_{t}L^{\infty}_{x}$; see Subsection~\ref{Sec:Beltrami} and~\eqref{VOB}.

\begin{proposition}\label{Prop:BeltrPreq}
There exists a constant $\delta_{3} >0$ such that the following holds.
Let~$w_0 \in  L^{\infty}(\Rtre)$ such that~$\nabla \times w_{0} = \lambda w_{0}$ for some~$\lambda \neq 0$.
The set $\Pi_{\delta_{0},\bar{x}}$
is regular for
any suitable weak solution $u$ to the Navier--Stokes equation, with divergence free initial data~$u_{0} \in L^{2}_{loc}$,
such that: 
\begin{enumerate}
\item $u_{0} - w_{0} \in L^{2}$;
\item $u - e^{-t \lambda^{2}} w_{0} \in L^{\infty}_{t}L^{2}_{x} \cap L^{2}_{t}\dot{H}^{1}_{x}$ is a suitable weak solution to the 
perturbed Navier--Stokes equation~\ref{ModProbInPrelimin}, around the solution~$e^{-t \lambda^{2}} w_{0}$, with initial data $u_{0} - w_{0}$;
 \end{enumerate}
 and
 \begin{equation}\nonumber
    \| |x-\bar{x}|^{-1/2} (u_{0} - w_{0})\|_{L^{2}}\le \delta_{3} e^{- \delta_{3}^{-1} \lambda^{-2} \| w_{0} \|^{2}_{L^{\infty}}}.
\end{equation}
\end{proposition}

\subsubsection*{Proof of Proposition~$\ref{Prop:BeltrPreq}$}

Recalling Subsection~\ref{Sec:Beltrami}, since $w_{0}$ is analytic (it is a solution of $\Delta w_{0} = -\lambda^{2} w_{0}$), it is clear that  
$w(t,x) := e^{-\lambda^{2}t}w_{0}$ is a suitable weak solution to the Navier--Stokes equation that also satisfies the regularity
condition~(\ref{CKNConditionSB}) at any~$(t, x) \in (0,\infty) \times \Rtre$. Moreover, it s clear 
that~$t \in [0,\infty) \to w(t) \in L^{2}(K)$ is continuous for any compact $K \subset \Rtre$ and, since the pressure can be chosen 
to be~$P=-\frac{1}{2}|w(t,x)|^{2}$, the 
representation formula required by Definition \ref{Def:GenRefSol} trivially holds taking $\mathcal{T} = -\frac{1}{2} Id$.
Finally
\begin{equation}\label{VOB}
\| w \|^{2}_{L^{2}_{t} L^{\infty}_{x}} = \int_{0}^{\infty} e^{-2 \lambda^{2} s} \|w_{0}\|^{2}_{L^{\infty}_{x}} ds =
\frac{1}{2} \lambda^{-2} \|w_{0}\|^{2}_{L^{\infty}_{x}} \, ,
\end{equation}
so that $w$ is a generalized reference solution of size $\mathcal{K} =  \frac{1}{2} \lambda^{-2} \|w_{0}\|^{2}_{L^{\infty}_{x}}$,
and the statement follows by Theorem~\ref{MainThmPerturbationMoreGen}, taking~$\delta_{3} = 2 \delta_{0}$. 

 \hfill $\Box$

\subsection{Weak solutions with unbounded energy}

In the following statement we assume the initial data~$u_{0}$ to be only locally square integrable, but,
since the difference $u_{0} - w_{0}$ belongs to $L^{2}$, with $w_{0}$ the initial datum of a generalized 
reference solution~$w \in L^{2}_{t}L^{\infty}_{x}$, 
we are still able to prove the existence of suitable weak solutions.

\begin{proposition}\label{InfiniteExistenceNew}
Let $w \in L^{2}_{t}L^{\infty}_{x}$ be a generalized reference solution to the Navier--Stokes 
equation (see Definition~\ref{Def:GenRefSol}), with divergence free data $w_{0} \in L^{2}_{loc}$.
For any divergence free initial datum~$u_{0}$, such that~$u_{0} - w_{0} \in L^{2}$,
there exists a suitable weak 
solution~$u$ to the Navier--Stokes equation, such 
that~$v := u-w \in L^{\infty}_{t}L^{2}_{x} \cap L^{2}_{t}\dot{H}^{1}_{x}$ 
is a suitable weak solution
to the perturbed Navier--Stokes equation~\eqref{ModProbInPrelimin}, around $w$, with data~$v_{0} := u_{0} - w_{0}$. 
\end{proposition}

\subsubsection*{Proof of Proposition \ref{InfiniteExistenceNew}}
 
Once we have found~$v$, it is
straightforward to check that the full statement follows taking
$u := w + v$ and $P_{u} := P_{w} + P_{v}$; see Lemma~\ref{LasLem} for more details.  
In order to prove the existence of~$v$, we first consider, as usual, the family of mollified problems ($0< \varepsilon <1$):
\begin{equation}\label{CauchyNSForvExistenceMollified}
\left \{
\begin{array}{lcl}
\partial_{t}v^{(\varepsilon)} + ((v^{(\varepsilon)} * \rho^{\varepsilon}) \cdot \nabla) v^{(\varepsilon)} + 
(v^{(\varepsilon)} \cdot \nabla) w 
& &
\\
 + (w \cdot \nabla) v^{(\varepsilon)}
  -\Delta v^{(\varepsilon)} = - \nabla P_{v^{(\varepsilon)}} &  &    \\
\nabla \cdot v^{(\varepsilon)} = 0  &  &   \\
v^{(\varepsilon)}(0)   = v_{0} \, , &  &  
\end{array}\right. 
\end{equation}
where $0 \leq \rho \in C^{\infty}_{c}(\Rtre)$, $\int_{\Rtre} \rho =1$, 
$\rho^{\varepsilon}(x):= \varepsilon^{-3} \rho(\varepsilon^{-1} x)$.
Since all the vector fields involved are divergence free,  the first equation can be rewritten as
\begin{equation}\label{ALLemma}
\partial_{t} v^{(\varepsilon)} =
\nabla \cdot [ - (v^{(\varepsilon)} * \rho^{\varepsilon}) \otimes v^{(\varepsilon)} - 
v^{(\varepsilon)} \otimes w 
-  w \otimes v^{(\varepsilon)}
 -  P_{v^{(\varepsilon)}} Id + \nabla v^{(\varepsilon)} ] \, .
\end{equation}

A standard fixed point argument allows to find, for any $0 < \varepsilon <1$, a unique  
solution~$v^{(\varepsilon)}$ in $C_{b}([0, T); L^{2}_{x}) \cap L^{2}_{T}\dot{H}^{1}_{x}$
 such that, for all $t < T$: 
\begin{equation}\label{HYInftyProof}
  \textstyle
  \int_{\mathbb{R}^{3}}  
  |v^{(\varepsilon)} (t)|^{2} 
   +  
  \textstyle
  2 \int_{0}^{t}\int_{\mathbb{R}^{3}} 
   |\nabla  v^{(\varepsilon)}|^{2} 
  \le 
   \int_{\mathbb{R}^{3}}  |v_{0}|^{2} 
   + 
   2 \int_{0}^{t}\int_{\mathbb{R}^{3}}  (v^{(\varepsilon)} \cdot \nabla ) v^{(\varepsilon)} \cdot w \, ;
\end{equation}
we write $L^{r}_{T}$ when the time integration is restricted to $t \in (0,T)$.
The local existence time $T$ depends, in principle, on $\|v_{0}\|_{L^{2}}$, but it can be taken arbitrarily large   
via a standard continuation argument.    
Indeed, using H\"older's and Young's inequalities, we can estimate the last term of (\ref{HYInftyProof}) as
\begin{eqnarray}\nonumber
& & 
\textstyle
2 \int_{0}^{t}\int_{\mathbb{R}^{3}}  (v^{(\varepsilon)} \cdot \nabla )   v^{(\varepsilon)}   \cdot  w
\leq  
 \| \nabla v^{(\varepsilon)} \|^{2}_{L^{2}_{t}L^{2}_{x}}
+ 
\| |v^{(\varepsilon)}| |w | \|^{2}_{L^{2}_{t}L^{2}_{x}}
\\ \nonumber
& \leq &
\textstyle
 \int_{0}^{t}\int_{\mathbb{R}^{3}} 
   |\nabla  v^{(\varepsilon)}|^{2}
+ 
\int_{0}^{t} \| w \|^{2}_{L^{\infty}_{x}}(s) \left(  \int_{\mathbb{R}^{3}}  |v^{(\varepsilon)} (s,x)|^{2} dx \right) ds \, ,
\end{eqnarray}
plugging this into the (\ref{HYInftyProof}) and absorbing the term $  \int_{0}^{t}\int_{\mathbb{R}^{3}} 
   |\nabla  v^{(\varepsilon)}|^{2}$ into the left hand side, we obtain 
\begin{equation}\nonumber
  \textstyle
  \int_{\mathbb{R}^{3}}  
  |v^{(\varepsilon)} (t)|^{2} 
   +  
  \textstyle
   \int_{0}^{t}\int_{\mathbb{R}^{3}} 
   |\nabla  v^{(\varepsilon)}|^{2} 
  \le 
   \int_{\mathbb{R}^{3}}  |v_{0}|^{2}   
  +
 \int_{0}^{t} \| w \|^{2}_{L^{\infty}_{x}}(s) \left(  \int_{\mathbb{R}^{3}}  |v^{(\varepsilon)} (s,x)|^{2} dx \right) ds \, .
\end{equation}
Thus, since $\| w \|^{2}_{L^{\infty}_{x}}(t)$ is integrable by assumption, the Gr\"onwall's inequality gives
\begin{equation}\label{UnifEstimates}
\textstyle
\int_{\mathbb{R}^{3}}  
  |v^{(\varepsilon)} (t,x)|^{2} dx \leq Ae^{\mathcal{K}},
  \quad
  \int_{0}^{t}\int_{\mathbb{R}^{3}} 
   |\nabla  v^{(\varepsilon)} (s,x)|^{2} \ dx ds \leq A(1+\mathcal{K}e^{\mathcal{K}})  \, ,
\end{equation}
for all $t < T$, where $A := \| v_{0} \|^{2}_{L^{2}}$ and $\mathcal{K} := \| w \|^{2}_{L^{2}_{t}L^{\infty}_{x}}$. 
Then, the continuation argument allows to extend the local theory to any $T>0$, in such a way that 
\begin{equation}\label{ALLemmaPreq}
\sup_{t \geq 0}
\int_{\mathbb{R}^{3}}  
  |v^{(\varepsilon)}(t,x)|^{2} dx \leq Ae^{\mathcal{K}},
  \quad
  \int_{0}^{\infty}\int_{\mathbb{R}^{3}} 
   |\nabla  v^{(\varepsilon)} (s,x)|^{2} \ dx ds \leq A(1+\mathcal{K}e^{\mathcal{K}}) \, . 
\end{equation}
Since these bounds hold uniformly in $0 < \varepsilon < 1$, we can extract, for any $T >0$, a subsequence,
that with a little abuse of notations we still denote~$v^{(\varepsilon)}$, that converges weakly to 
some $v$ in 
$L^{2}((0,T) \times \Rtre)$, 
and such 
that $\nabla v^{(\varepsilon)}$ converges to $\nabla v$ in 
$L^{2}((0,\infty)\times \Rtre)$, as $\varepsilon \to 0$.
In the following we will repeat the same abuse a few times. 
By the weak lower semicontinuity of the norm and~\eqref{ALLemmaPreq}, we also have that $v$ belongs to $L^{2}_{t}\dot{H}^{1}_{x}$.
Now, it is well known that in the context of the Navier--Stokes equation the weak convergence of $v^{(\varepsilon)}$ to $v$ can be promoted to
strong convergence in $L^{2}((0,T) \times \Rtre)$, for any $T >0$. This is also the case here, by a straightforward adaptation
of the standard argument, that we will recall at the end of the proof.   
This also implies that, 
on a certain subsequence, we have 
\begin{equation}\label{EgorLemma}
\| v^{(\varepsilon)}(t) - v(t) \|_{L^{2}(\Rtre)} \to 0
\quad
\mbox{for almost every $t \in (0,\infty)$} \, .
\end{equation}
Using this and the first inequality in~\eqref{ALLemmaPreq}, we immediately see that $v \in L^{\infty}_{t}L^{2}_{x}$. 
Since we already observed that $v \in L^{2}_{t}\dot{H}^{1}_{x}$, using again~\eqref{ALLemmaPreq}, Sobolev embedding and 
interpolation, we see that,
for all $T > 0$,  
the vector fields $v$, $v^{(\varepsilon)}$ and  
$v^{(\varepsilon)} * \rho^{\varepsilon}$ are~$\varepsilon$-uniformly 
bounded in~$L^{\bar{r}}_{T} L^{\bar{q}}_{x}$, provided~$2/\bar{r} + 3/\bar{q} \geq 3/2$, $2 \leq \bar{q} \leq 6$; see~\eqref{UIP}.
Using this fact, triangle inequality and interpolation, the $L^{2}_{T}L^{2}_{x}$ (strong) 
convergence of $v^{(\varepsilon)}$ (and of $v^{(\varepsilon)} * \rho^{\varepsilon}$) can be promoted to (strong)
convergence in $L^{\bar{r}}_{T} L^{\bar{q}}_{x}$, provided~$2/\bar{r} + 3/\bar{q} > 3/2$, $2 < \bar{q} < 6$, for any $T >0$.  
Now we consider the pressure.
Taking the divergence of the first equation in (\ref{CauchyNSForvExistenceMollified}), we have 
\begin{equation}\nonumber
\Delta P_{v^{(\varepsilon)}} 
= - \nabla \otimes \nabla \cdot [ (v^{(\varepsilon)} * \rho^{\varepsilon}) \otimes v^{(\varepsilon)} + 2 v^{(\varepsilon)} \otimes w ] \, ,
\end{equation}
 or equivalently (we recall that~$\mathcal{R}$ is the Riesz transform)
\begin{equation}\label{SameIsTrue11}
P_{v^{(\varepsilon)}} = 
\mathcal{R} \otimes \mathcal{R} \cdot [ (v^{(\varepsilon)} * \rho^{\varepsilon}) \otimes v^{(\varepsilon)} + 2 v^{(\varepsilon)} \otimes w ] \,.
\end{equation} 
Since we have also assumed~$w \in L^{2}_{t}L^{\infty}_{x}$, 
using the~$L^{q > 1}$ boundedness of the Riesz transform, we have 
that $P_{v^{(\varepsilon)}}$ is $\varepsilon$-uniformly bounded in~$L^{\bar{r}}_{T} L^{\bar{q}}_{x}$, 
provided $2/\bar{r} + 3/\bar{q} \geq 3$, $2 \leq \bar{q} \leq 3$; comparing to~\eqref{UIPPressure}, the 
restriction~$2 \leq \bar{q}$ arises since 
we also need to estimate the contribution of $w$ to the pressure. 
Thus we can extract a subsequence, that we still denote with $P_{v^{(\varepsilon)}}$, which converges 
weekly to a function $P_{v}$ in~$L^{\bar{r}}_{T} L^{\bar{q}}_{x}$, 
$2/\bar{r} + 3/\bar{q} \geq 3$, $2 \leq \bar{q} \leq 3$. Then the limit~$P_{v}$ has to coincide with
\begin{equation}\label{SameIsTrue12}
P_{v} := \mathcal{R} \otimes \mathcal{R} \cdot (v \otimes v +2 v \otimes w) \, .
\end{equation}

We are now ready to show that 
the couple $(v, P_{v})$ is a suitable weak solution to the perturbed Navier--Stokes equation~\eqref{ModProbInPrelimin}, around~$w$; see
Definition~\eqref{ModSolDef}.
It is already clear that the equation~\eqref{ModProbInPrelimin} is satisfied in the sense of distribution, that 
the representation 
formula~\eqref{ReprFormSing} for the pressure holds taking~$\mathcal{T} = \mathcal{R} \otimes \mathcal{R}$, 
and we already know
$v \in L^{\infty}_{t}L^{2}_{x} \cap L^{2}_{t}\dot{H}^{1}_{x}$, $P_{v} \in L^{3/2}_{loc}((0,\infty) \times \Rtre)$.
Now we prove the perturbed local energy inequality~\eqref{eq:StrongPertGenInProofOLD}.  
We first consider the case~$t_{0} = 0$. For any non negative test function~$\phi$, we take the scalar product of the 
first equation in~(\ref{CauchyNSForvExistenceMollified}) 
with~$2\phi v^{(\varepsilon)}$, integrate over $(0,t) \times \Rtre$ and by parts, so that
 \begin{eqnarray}\label{EpsilonWeakEnergy}
  \textstyle
  & \int_{\mathbb{R}^{3}} &   
  |v^{(\varepsilon)} |^{2} \phi(t,x) dx
   +  
  \textstyle
  2 \int_{0}^{t}\int_{\mathbb{R}^{3}} 
  |\nabla  v^{(\varepsilon)}|^{2} \phi
  = 
  \int_{\mathbb{R}^{3}}  |v_{0}|^{2}  \phi(0,x) dx
  \\   \nonumber
  & + &
  \textstyle
   \int_{0}^{t} \int_{\mathbb{R}^{3}} 
  |v^{(\varepsilon)}|^{2} (\phi_{t} + \Delta \phi)
  +  |v^{(\varepsilon)}|^{2} (v^{(\varepsilon)}*\rho^{\varepsilon}) \cdot \nabla \phi 
   + 2P_{v^{(\varepsilon)}} v^{(\varepsilon)} \cdot \nabla \phi
  \\   \nonumber
  & + &
  \textstyle
  \int_{0}^{t}\int_{\mathbb{R}^{3}} 
  |v^{(\varepsilon)}|^{2} w \cdot \nabla \phi
  +
  2  (v^{(\varepsilon)} \cdot w ) v^{(\varepsilon)} \cdot \nabla \phi  
  + 2  (v^{(\varepsilon)} \cdot \nabla ) v^{(\varepsilon)} \cdot w \phi .
  \end{eqnarray} 
Using this and the integrability and convergence properties of the functions involved, it is now straightforward to pass to the limit  
$\varepsilon \to 0$ so that  
  \begin{eqnarray}\label{AEEnIneqq}
  \textstyle
  & \int_{\mathbb{R}^{3}} &  
  |v |^{2} \phi(t,x)  dx
   +  
  \textstyle
  2 \liminf_{\varepsilon \to 0} \int_{0}^{t}\int_{\mathbb{R}^{3}} 
   |\nabla  v^{(\varepsilon)}|^{2}  \phi
  \leq 
  \\   \nonumber
  & \leq &
  \textstyle
   \int_{\mathbb{R}^{3}}  |v_{0}|^{2} \phi(0,x) dx
   +
   \int_{0}^{t} \int_{\mathbb{R}^{3}} 
  |v|^{2} (\phi_{t} + \Delta \phi)
  +  |v|^{2} v \cdot \nabla \phi 
   + 2P_{v} v \cdot \nabla \phi
  \\   \nonumber
   & + &
  \textstyle
  \int_{0}^{t}\int_{\mathbb{R}^{3}} 
  |v|^{2} (w \cdot \nabla \phi) 
  +
  2  (v \cdot w) v \cdot \nabla \phi  
  + 2  (v \cdot \nabla ) v \cdot w \phi \, ,
  \end{eqnarray}
for almost every $t > 0$.
We omit the details of this straightforward fact.
We only remark that, in order to handle the term
$(v \cdot \nabla ) v \cdot w \phi$, we need to notice that by~\eqref{EgorLemma} end Egoroff's theorem, we have, for any $0 < \delta < 1$, $t >0$, 
that    
$\sup_{s \in \Omega_{t,\delta}} \| v^{(\varepsilon)}(s) - v(s) \|_{L^{2}(\Rtre)} < \delta $, for all $\varepsilon$ sufficiently small,
where $\Omega_{t,\delta} \subset (0,t)$ with $|(0,t) \setminus \Omega_{t}| < \delta$.
Once we have proved~\eqref{AEEnIneqq}, for almost any $t >0$, since 
\begin{equation}\nonumber
 \int_{0}^{t}\int_{\mathbb{R}^{3}} 
  \phi |\nabla  v|^{2}
  \leq 
  \liminf_{\varepsilon \to 0} \int_{0}^{t}\int_{\mathbb{R}^{3}}
  \phi |\nabla  v^{(\varepsilon)}|^{2} \, , 
\end{equation}
the (\ref{eq:StrongPertGenInProofOLD}) has been proved for 
almost every $t > 0$ and $t_{0}=0$. In fact, we can extended it to any $t >0$ once we modify~$v$ on a 
set of zero Lebesgue measure, in order to make the function~$t \in (0,\infty) \to  v (t) \in L^{2}(\Rtre)$ 
weakly continuous. This is a standard argument in the theory of vector valued time dependent functions, 
once we notice that $\partial_{t} u \in L^{1}_{T}H^{-1}_{x}$, for all~$T >0$, that is immediately clear
if we rewrite the equation as
\begin{equation}\nonumber
\partial_{t} v =
\nabla \cdot [ - (v  \otimes v - 
v \otimes w 
-  w \otimes v
 -  P_{v} Id + \nabla v ] \, ,
\end{equation}
as we may taking advantage of the fact that $v$ and $w$ are divergence free. 
Then, the case~$t_{0} > 0$
of the inequality can be deduced, for almost any $t_{0}>0$, by the case $t_{0}=0$, as explained in 
the forthcoming Lemma \ref{Lemma:PertGenEnIneq}.

Let now $\Phi$ be a
$C^{2}_{c}(\Rtre)$ vector field. Recalling the momentum equation~\eqref{CauchyNSForvExistenceMollified}, we have
\begin{align}\nonumber
     \textstyle
     \int_{0}^{t'}  \partial_{t} 
& 
     \textstyle \big(  \int_{\Rtre} v^{(\varepsilon)}(t) \cdot \Phi \big)  dt =
\\ \nonumber
& 
  \textstyle  = \int_{0}^{t'} 
    \big( \int_{\Rtre}  ((v^{(\varepsilon)} * \rho^{\varepsilon}) \otimes v^{(\varepsilon)}   +
   v^{(\varepsilon)} \otimes w  +
   w \otimes v^{(\varepsilon)}  
    + P_{v^{(\varepsilon)}}  )(s) \cdot  \nabla \Phi 
     +  v^{(\varepsilon)} (s) \cdot \Delta \Phi \big)  ds \, ,
\end{align}   
for all $t' >0$.
Thus, recalling the ($\varepsilon$-uniform) integrability properties of $v^{(\varepsilon)}$, $P_{v^{(\varepsilon)}}$ 
and~$w \in L^{2}_{t}L^{\infty}_{x}$,
we easily arrive to
 \begin{equation}\label{LBNL}
 \lim_{t' \to 0^{+}} \sup_{0< \varepsilon < 1} \left| \int_{\Rtre} ( v^{(\varepsilon)}(t') - v_{0} ) \cdot \Phi \right|  = 0 
 \end{equation}
The $L^{2}$-weak convergence $v(t) \to v_{0}$ as $t \to 0^{+}$ is now a consequence of the following observation. Given 
any $0 < \delta < 1$, for any sufficiently small time $t >0$, we can find a time $t'$ at which~\eqref{EgorLemma} holds 
and such that $|t-t'|$ is also sufficiently small that 
we have these three facts: $|\int_{\Rtre} (v(t) - v(t')) \cdot \Phi | < \delta$, as a consequence of the 
$L^{2}$-weak continuity of the function $t \in (0,\infty) \to v(t)$, 
$|\int_{\Rtre} (v^{(\varepsilon)}(t') - v_{0}) \cdot \Phi | < \delta$, as a consequence of~\eqref{LBNL}, and, taking $\varepsilon$ sufficiently small,
$|\int_{\Rtre} (v(t') - v^{(\varepsilon)}(t')) \cdot \Phi | < \delta$, as a consequence of~\eqref{EgorLemma} .  
In conclusion, we have shown that $v$ is a suitable solution to the perturbed equation~\eqref{ModProbInPrelimin}.

In the proof we have used the
strong $L^{2}((0,T)\times \Rtre)$ convergence of~$v^{(\varepsilon)}$
to~$v$, 
as~$\varepsilon \to 0$. We conclude proving this fact. 
Using~(\ref{ALLemma}) and the ($\varepsilon$-uniform) integrability properties of $v^{(\varepsilon)}$, $P_{v^{(\varepsilon)}}$ and 
$w \in L^{2}_{t}L^{\infty}_{x}$, 
we can immediately check that~$\sup_{\varepsilon>0} \| \partial_{t}v^{(\varepsilon)} \|_{L^{1}_{T}H^{-1}_{x}}$ is finite. 
Thus, recalling~(\ref{ALLemmaPreq}), we are allowed to 
use the Aubin--Lions lemma (see for instance~\cite[Theorem 12.1]{Lem2})
to extract a subsequence 
\begin{equation}\label{LLAB}
v^{(\varepsilon)} \to v,
\quad
\mbox{strongly in} 
\quad
L^{2}_{loc}((0,\infty) \times \Rtre) \, .
\end{equation}
Define now $\gamma_{\leq R}(x) := \gamma(x/R)$, where $R > 1$ and $\gamma \geq 0$ is a smooth cut-off function of the unit ball in $\Rtre$. 
Let then $\eta^{2}_{>R} := 1-\gamma^{2}_{\leq R}$ and split
 $$
v^{(\varepsilon)} - v = \gamma^{2}_{\leq R}( v^{(\varepsilon)} -v ) + \eta^{2}_{> R} ( v^{(\varepsilon)} - v) \, .
$$
Using the strong $L^{2}_{loc}$ convergence (\ref{LLAB}) and 
\begin{equation}\label{EnergyOutsidePreq}
\limsup_{R \to \infty} \ \sup_{0 < \varepsilon < 1}  \int_{0}^{T} \int_{\Rtre} \eta^{2}_{> R}(x) |(v^{(\varepsilon)} - v)(t,x)|^{2} dxdt  = 0 \, ,
\end{equation}
we can immediately deduce
strong convergence in~$L^{2}((0,T)\times \Rtre)$.
In order to prove~(\ref{EnergyOutsidePreq}),
since~$v \in L^{2}((0,T)\times \Rtre)$, 
it suffices to show
\begin{equation}\label{EnergyOutside}
\limsup_{R \to \infty} \ \sup_{0 < \varepsilon < 1}  \int_{0}^{T} \int_{\Rtre} \eta^{2}_{> R}(x) |v^{(\varepsilon)} (t,x)|^{2} dxdt  = 0 \, .
\end{equation}
This will be deduced by an appropriate energy-type inequality.
We take 
the scalar product of the first equation in~(\ref{CauchyNSForvExistenceMollified}) with the vector field~$2 \eta_{> R}^{2} v^{(\varepsilon)}$. 
After elementary manipulations we obtain 
\begin{eqnarray}\label{EnOutHard}
& &
\partial_{t} | v^{(\varepsilon)}_{> R} |^{2}  
  +    
2 |\nabla  v^{(\varepsilon)}_{> R}|^{2} =
\nabla \cdot Z_{R,\varepsilon}   +
\\ \nonumber 
& + & 
 \left[ 
|v^{(\varepsilon)}|^{2} (v^{(\varepsilon)} * \rho^{\varepsilon}) 
 +  2P_{v^{(\varepsilon)}} v^{(\varepsilon)} 
 + 
|v^{(\varepsilon)}|^{2} w 
 \right] \cdot \nabla \eta_{> R}^{2} 
 + 2 |v^{(\varepsilon)}|^{2} |\nabla \eta_{>R}|^{2}  +
\\ \nonumber 
 & + &
 2 \eta_{>R} (v^{(\varepsilon)} \cdot w ) v^{(\varepsilon)} \cdot \nabla \eta_{>R}
  + 
2  (v_{>R}^{(\varepsilon)} \cdot \nabla ) v_{>R}^{(\varepsilon)} \cdot w ,
\end{eqnarray}
 where 
 $$v^{(\varepsilon)}_{> R} :=     \eta_{> R}   \,  v^{(\varepsilon)}  \, ,  $$
 and
 \begin{equation}\nonumber
 Z _{R,\varepsilon} 
   :=      
 \eta_{>R}^{2} \left[ 
 \nabla (|v^{(\varepsilon)}|^{2}) 
 - |v^{(\varepsilon)}|^{2} (v^{\varepsilon} * \rho^{\varepsilon}) 
 - 2 P_{v^{(\varepsilon)}} v^{(\varepsilon)} 
 - 2 (v^{(\varepsilon)} \cdot w ) v^{(\varepsilon)} 
 - |v^{(\varepsilon)}|^{2} w
 \right].
 \end{equation}
We refer to~\cite[Proposition 14.1]{Lem} for the detailed computation in the case~$w=0$.  
Since~$Z_{R,\varepsilon}$ 
is integrable on $(0,T) \times \Rtre$, uniformly in~$\varepsilon$,
we arrive to $(0 < t < T)$:
\begin{eqnarray}\label{ITLHSFinal}
 && 
 \textstyle
 \int_{\Rtre} | v^{(\varepsilon)}_{> R} |^{2} (t) 
+2  \int_{0}^{t} \int_{\Rtre}  |\nabla  v^{(\varepsilon)}_{> R}|^{2}
 \leq  
 \int_{\Rtre} | v_{> R} (0) |^{2}  +
 \\ \nonumber
  &+ &
  \textstyle
 \int_{0}^{t}\int_{\Rtre}
 \left[ 
 |v^{(\varepsilon)}|^{2} (v^{(\varepsilon)} * \rho^{\varepsilon}) 
+2P_{v^{(\varepsilon)}} v^{(\varepsilon)} 
+
|v^{(\varepsilon)}|^{2} w
\right]  \cdot \nabla \eta_{> R}^{2} + 2 |v^{(\varepsilon)}|^{2} |\nabla \eta_{>R}|^{2}  
+ 
\\ \nonumber
& + &
\textstyle
2  \int_{0}^{t}\int_{\Rtre}
 \eta_{>R} (v^{(\varepsilon)} \cdot w ) v^{(\varepsilon)} \cdot \nabla \eta_{>R}
+
(v_{>R}^{(\varepsilon)} \cdot \nabla ) v_{>R}^{(\varepsilon)} \cdot w \, ,
\end{eqnarray}
Finally, since
$ | \nabla \eta_{> R} |, | \nabla \eta_{> R}^{2} |  \lesssim R^{-1}$ 
and 
\begin{equation}\nonumber
\textstyle
 2 \int_{0}^{t}\int_{\mathbb{R}^{3}}  (v_{>R}^{(\varepsilon)} \cdot \nabla ) v_{>R}^{(\varepsilon)} \cdot w
 \leq
 \int_{0}^{t}\int_{\mathbb{R}^{3}} 
   |\nabla  v_{>R}^{(\varepsilon)}|^{2}
+ 
\int_{0}^{t} \| w \|^{2}_{L^{\infty}_{x}}(s) \big(  \int_{\mathbb{R}^{3}}  |v_{>R}^{(\varepsilon)} (s,x)|^{2} dx \big) ds,
\end{equation}
after absorbing $\int_{0}^{t}\int_{\mathbb{R}^{3}}  |\nabla  v_{>R}^{(\varepsilon)}|^{2} $ into the 
left hand side of (\ref{ITLHSFinal}), we obtain
\begin{equation}\nonumber    
\textstyle
\int_{\Rtre} | v^{(\varepsilon)}_{> R} |^{2} (t,x) dx 
 \lesssim_{T}
\textstyle
R^{-1}   
+
 \int_{\Rtre} | v_{> R} (0,x) |^{2}dx   
 +
 \int_{0}^{t} \| w \|^{2}_{L^{\infty}_{x}}(s) \big(  \int_{\mathbb{R}^{3}}  |v_{>R}^{(\varepsilon)} (s,x)|^{2} dx \big) ds \, .
\end{equation}
Since $\| w \|^{2}_{L^{\infty}_{x}}(t)$ is integrable,
the Gr\"onwall inequality gives
\begin{equation}\label{GreatGron}
\int_{\Rtre} | v^{(\varepsilon)}_{> R} |^{2} (t,x)dx \lesssim_{T}
A e^{\mathcal{K}}, 
\quad
A := R^{-1} + \| \eta_{> R} v_{0}  \|^{2}_{L^{2}}, 
\quad
\mathcal{K} := \| w \|^{2}_{L^{2}_{t} L^{\infty}_{x}}
\end{equation}
for any $0<t<T$.
Since the right hand side of~(\ref{GreatGron}) does not depend on~$\varepsilon$ anymore and goes to zero as~$R \to \infty$, 
we obtain~(\ref{EnergyOutside}), so that the
proof is concluded.

\hfill $\Box$

\section{Small data}\label{Sect:SmallData}

When $\mathcal{K}$ is sufficiently small, we can improve the size of the regular set of
small perturbations of reference solutions.

\begin{theorem}\label{Theorem:DL}
There exists a constant $\delta_{4} > 0$ such that the following holds. 
Let 
$w$ be a reference solution of size $\mathcal{K} \leq \delta_{4}$ to the Navier--Stokes equation
with divergence free 
initial data $w_{0}$ (see Definition~\ref{Def:RefSol}).
For any $M > 1$, 
the set 
\begin{equation}\label{SmallAssSNU}
    \Pi_{M \delta_{4}, \bar{x}}
    :=
    \left\{ (t,x) \ : \ t > 
    \frac{|x-\bar{x}|^{2}}{M \delta_{4}} \right\}
  \end{equation}
is regular,
for every suitable weak solution $u \in L^{\infty}_{t}L^{2}_{x} \cap L^{2}_{t}\dot{H}^{1}_{x}$
to the Navier--Stokes equation with divergence free initial data~$u_{0} \in L^{2}$ such that
\begin{equation}\label{Hyp:Theorem:DL}
      \| |x-\bar{x}|^{-1/2} (u_{0} - w_{0})\|_{L^{2}}  \le \delta_{4} e^{-M^{2}/\delta_{4}}.
\end{equation}

\end{theorem}

The size of the regular set \eqref{SmallAssSNU} increases indefinitely as long as we consider smaller perturbations of $w_0$. 
More precisely, if we take a sequence $u_0^{n}$ such that
$$\| |x-\bar{x}|^{-1/2} (u_{0}^{n} - w_{0})\|_{L^{2}} \to 0 \quad \mbox{as $n \to \infty$} \, ,$$
then we can find a divergent sequence of real numbers $M_n$ such that \eqref{Hyp:Theorem:DL} holds, so that 
that $ \Pi_{M_n \delta_{4}, \bar{x}}$ is a regular set for the corresponding weak solutions $u^{n}$. Notice that we clearly have
$ \Pi_{M_n \delta_{4}, \bar{x}} \to \{ t > 0\} $ as $n \to \infty$ (since $M_{n} \to \infty$).
The case $\mathcal{K}=0$ of Theorem \ref{Theorem:DL}, namely~$w=w_{0}=0$, 
has been proved in~\cite[Corollary 1.6]{DL2}. For a direct argument we refer to~\cite{R}.

\subsection{Proof of Theorem \ref{Theorem:DL}}\label{sec:proof2}


\subsubsection*{Idea of the proof}

Comparing with Theorem \ref{MainThmPerturbation}, when we look at the perturbed energy inequality,
we have no more trouble with the terms in which the reference solution $w$ is involved, since $w$ has been assumed to be
small ($\mathcal{K} \ll1$). This
gives us enough freedom to improve the size of the regular sets.
We again 
distinguish two time regimes $t \leq t^{*}$, $t > t^{*}$, but now we
choose $t^{*}$ in such a way that the term contributing to the size of the regular set becomes very small for $t > t^{*}$. 
Then we use the
(exponential) smallness assumption (\ref{SmallAssSNU}) on $u_{0} - v_{0}$ to control the weighted $L^{2}$ norm of~$u-w$, 
up to the time $t^{*}$. This permits to conclude the proof.

\begin{proof}

We restrict to $\bar{x} =0$; the general case follows by translation.
Let $v_{0} := u_{0} - w_{0}$, $v:=u-w$ and
denote $\epsilon := \| |x|^{-1/2} v_{0} \|_{L^{2}(\Rtre)}$.
For all $\xi \in \mathbb{R}^{3}$ and $T > 1$ we investigate     
when 
$$
L(T,\xi) := \{ (s, \xi s) \ : \ s \in (0,T) \}
$$
is a regular set. 
We again change variables
\begin{equation}\label{NewCOV}
  (t, y) = (t, x - \xi t),
  \qquad
  v_{\xi}(t,y) = v(t,x),
  \qquad
  w_{\xi}(t,y) = w(t,x) \, ,
\end{equation}
and set
\begin{equation}\label{Def:SigmaNuBis}
  \sigma_{\mu}(y) := (\mu + |y|^{2})^{-\frac12},
\quad
\mu >0 \, .
\end{equation}
We define, for any
$M > 1$: 
\begin{equation}\nonumber
  \Gamma (M, T, \xi, \mu) := \left\{ s \in (0,T] \ : \
  \int_{s}^{s + T/M} \int_{\RT} 
  \sigma_{\mu}(y) |\nabla v_{\xi} (\tau , y)|^{2} \ dy d\tau  > M
  \right\}
\end{equation}
and 
\begin{equation}\label{DefOfBarSSF}
t^{*}(M, T, \xi, \mu) := 
\begin{cases}
  \inf \, \Gamma (M, T, \xi, \mu)    
  & \mbox{if} \quad \Gamma (M, T, \xi, \mu) \neq \emptyset    \\
  T   &   \mbox{otherwise} \, .
\end{cases}
\end{equation}
The following estimate is an immediate consequence of the definition of $t^{*}$,
\begin{equation}\label{BasicPropertiesSF}
  B_{\mu}(t^{*}) := \int_{0}^{t^{*}} \int_{\RT} \sigma_{\mu}(y)
  |\nabla v_{\xi}(\tau ,y)|^{2}  \ d \tau dy  \leq M(M+1)
  \le 
  2M^{2} \, .
  \end{equation}
If there exists $\mu^{*}$ such that $t^{*}(\cdot, \mu) = T$ for all $0< \mu < \mu^{*} $, 
taking the limit~$\mu \to 0$ into~\eqref{BasicPropertiesSF}, that we are allowed since the weights $\sigma_{\mu}(y)$ are increasing as $\mu \to 0$, 
we get
\begin{equation}\nonumber
  B(T) := \int_{0}^{T} \int_{\RT} |y|^{-1}
  |\nabla v_{\xi}(\tau ,y)|^{2}   dy d \tau   
  =
  \int_{0}^{T} \int_{\RT} \frac{|\nabla v(\tau ,x)|^{2}}{|x-\xi \tau|} dx d \tau   
  \le 
  2M^{2} \, .
  \end{equation}
Using this bound and the argument at the end of Subsection \ref{SameArgument}, after the inequality~(\ref{FinalBound1}),
we show that $L(T,\xi)$ is regular. 

Thus the hard case turns out to be
$0 \leq t^{*}(\cdot, \mu) < T$ for a vanishing sequence of $\mu$. 
Noting that the quantity 
$t^{*}(\cdot, \mu)$ can not increase as $\mu \to 0$, this simply means that there exists
 some $\mu > 0$ such that $0 \leq t^{*}(\cdot, \mu) < T$. When this happens, we will write for simplicity
 $0 \leq t^{*} < T$. 
 We set
\begin{equation}\nonumber
a_{\mu}(t) := \int_{\Rtre} \sigma_{\mu}(y) |v_{\xi}(t,y)|^{2} dy = \int_{\Rtre} \frac{|v(t,x)|^{2}}{|x-\xi t|} dx.  
\end{equation}
As in Subsection~\ref{SSFNM}, see (\ref{Labella}), we have 
\begin{equation}\nonumber
a_{\mu}(t) \leq 
\epsilon^{2} e^{ZA}, 
\quad
A := B_{\mu}(t^{*}) + \mathcal{K} + t^{*} |\xi|^{2},
\quad
\mbox{for all}
\quad
0 \leq t \leq t^{*},
\end{equation}
for a certain constant $Z >1$. 
We assume 
\begin{equation}\label{ResrtCondFirstSF}
| \xi |^{2} t^{*} \leq M^{2},
\qquad
\epsilon \leq 1,
\end{equation} 
so that, using also the (\ref{BasicPropertiesSF}), we obtain
\begin{equation}\label{Calca}
a_{\mu}(t) \leq e^{Z(M^{2} + \mathcal{K})} \epsilon, 
\quad
\mbox{for all}
\quad
0 \leq t \leq t^{*}, 
\end{equation}
where $Z$ is a suitably larger constant.

Let $t^{*}_{n} \to t^{*}$, with $t^{*}_{n} \leq t^{*}$, so that
the perturbed energy inequality 
\begin{equation}\label{GenInSF}
\begin{split}
  \textstyle
  \int_{\mathbb{R}^{3}}  
  |v_{\xi}(t,y)|^{2}  \phi(t,y) dy
  &+ 
  \textstyle
  2 \int_{t^{*}_{n}}^{t}\int_{\mathbb{R}^{3}} 
   |\nabla  v_{\xi}|^{2} \phi 
  \le 
  \int_{\mathbb{R}^{3}}  |v_{\xi}(t^{*}_{n},y)|^{2} \phi(t^{*}_{n},y) dy
  \\
  & 
  \textstyle
   \!\!\!\!\!\!\!\!\!\!\!\!\!\!\!\!\!\!\!\!\!\!\!\!\!\!\!\!\!\!\!\!\!\!\!\!\!\!\!\!\!\!\!\!\!
   + 
  \int_{t^{*}_{n}}^{t} \int_{\mathbb{R}^{3}} 
  |v_{\xi}|^{2} (\phi_{t} - \xi \cdot \nabla \phi + \Delta \phi)
  +
  \int_{t^{*}_{n}}^{t} \int_{\mathbb{R}^{3}} (|v_{\xi}|^{2} 
  + 2P_{v_{\xi}})v_{\xi} \cdot \nabla \phi
  \\
  &
  \textstyle
  \!\!\!\!\!\!\!\!\!\!\!\!\!\!\!\!\!\!\!\!\!\!\!\!\!\!\!\!\!\!\!\!\!\!\!\!\!\!\!\!\!\!\!\!\!
  +
  \int_{t^{*}_{n}}^{t}\int_{\mathbb{R}^{3}} 
  |v_{\xi}|^{2} (w_{\xi} \cdot \nabla \phi) 
  +
  \textstyle
   2
  \int_{t^{*}_{n}}^{t}\int_{\mathbb{R}^{3}}  
   (v_{\xi} \cdot w_{\xi}) (v_{\xi} \cdot \nabla \phi)  
  +  (v_{\xi} \cdot \nabla ) v_{\xi} \cdot w_{\xi} \phi,
\end{split}
\end{equation}
holds for the sequence of test functions we are going to define.
Recall that this is the perturbed energy inequality (\ref{eq:StrongPertGenInProofOLD}), after the change of variables (\ref{NewCOV}).
We choose
$$
\phi(t,y):=\psi_{\mu}(t)\sigma_{\mu}(y) \chi(\delta |y|) \, ,
$$
where $\delta > 0$, 
$\chi : [0,\infty) \to [0,\infty)$
is a smooth non increasing function such that
\begin{equation*}
  \chi =1 \ \text{on}\  [0,1],
  \qquad
  \chi =0  \ \text{on}\ [2, \infty],
\end{equation*} 
 $ \sigma_{\mu}(y)$ has been defined in (\ref{Def:SigmaNuBis})
and
\begin{equation*}
  \psi_{\mu}(t) := e^{-k B_{t^{*}_{n}, \mu}(t)},
  \qquad 
  B_{t^{*}_{n}, \mu}(t) := 
  \int_{t^{*}_{n}}^{t}
  \int_{\mathbb{R}^{3}} 
  \sigma_{\mu}(y) |\nabla v_{\xi}(\tau,y)|^{2} dy d\tau  \, ,
\end{equation*}
with $k$ a positive constant to be specified.
Again, we should consider vanishing sequences $\delta_{n}, \mu_{n}$, instead of $\delta, \mu$, in order to 
be sure that the exceptional times, starting from 
which~\eqref{eq:StrongPertGenInProofOLD}, and so \eqref{GenInSF}, 
may not be satisfied, for at least one of our test functions, has measure zero. We keep writing
$\delta, \mu$, for simplicity.
As in Subsection \ref{SameArgument}, by a repeated use of the 
Stein bound for singular integrals and of the Caffarelli--Kohn--Nirenberg inequality, and   
appropriate cancellations
between the left hand side and the right hand side of \eqref{GenInSF}, we obtain, taking~$k = 6Z$, 
\begin{equation}\label{LastInSF}
  \textstyle
  a_{\mu}(t)
  \psi_{\mu}(t)
  -a_{\mu}(t^{*}_{n})
  +
  \frac{1-\psi_{\mu}(t)}{6Z}
  \le
  |\xi|^{2} \int_{t^{*}_{n}}^{t}
  \psi_{\mu}(s)ds+
  \frac{1}{5} \int_{t^{*}_{n}}^{t}
  \|  w(s,\cdot) \|^{r}_{L^{q}(\Rtre)} ds \, ;
\end{equation}
compare with the inequality (\ref{LastInPreq}).

The second term of the right hand side is immediately bounded by $\frac{1}{5} \mathcal{K}$, that we have assumed to be small.
In order to bound the first term of the right hand side in an efficient way, we need to 
take advantage of the key definition \eqref{DefOfBarSSF}
which implies
$$
B_{t^{*}\!\!, \, \mu}(t) :=
\int_{t^{*}}^{t}
  \int_{\mathbb{R}^{3}} 
  \sigma_{\mu}(y) |\nabla v_{\xi}(\tau,y)|^{2} dy d\tau 
  \geq  M
$$ 
for all $t \geq t^{*} + T/M$. 
Thus (we recall that $\psi_{\mu}(t) =  e^{-6Z B_{t^{*}_{n}, \mu}(t)}$ and $M,Z > 1$)

\begin{eqnarray}\nonumber
  \textstyle
  & & 
  \int_{t^{*}_{n}}^{t^{*} + T} \psi_{\mu} (s)  ds =
  \int_{t^{*}_{n}}^{t^{*} + T} e^{- 6Z B_{t^{*}_{n},\mu}(s)}  ds
 \\ \nonumber
 & \leq & 
\int_{t^{*}_{n}}^{t^{*} + T} e^{- 6Z B_{t^{*}\!\!, \, \mu}(s)} ds
=  
\int_{t^{*}_{n}}^{t^{*} + T/M} e^{- 6Z B_{t^{*}\!\!, \, \mu}(s)}  ds 
+ \int_{t^{*} + T/M}^{t^{*} + T} e^{- 6Z B_{t^{*}\!\!, \, \mu}(s)}  ds
  \\ \nonumber
  \textstyle
  & \leq &
  t^{*} - t^{*}_{n} +
  \frac{T}{M} + e^{-6ZM} \left( T - \frac{T}{M} \right)
  <
  t^{*} - t^{*}_{n} 
  +  \frac{2T}{M}
  - e^{-6ZM}\frac{T}{M} \, .
\end{eqnarray}
Thus, since $t^{*}_{n} \to t^{*}$, we can choose $n$ large enough so that
\begin{equation}\nonumber
\int_{t^{*}_{n}}^{t^{*} + T} \psi_{\mu} (s)  ds 
\leq \frac{2T}{M}.
\end{equation}
Consequently\begin{equation}\nonumber
  \textstyle
  a_{\mu}(t)
  +(\frac{1}{6Z}
  - \frac{1}{5} \mathcal{K}
  -e^{Z(M^{2} + \mathcal{K})} \epsilon
  -2|\xi|^{2}\frac{T}{M})
  e^{6Z B_{t^{*}_{n},\mu}(t)}
  \le
  \frac{1}{6Z},
\end{equation}
for $t^{*} \leq t \leq t^{*} +T$. 
Indeed, this follows by the inequality (\ref{LastInSF}), once we 
recall~$a_{\mu}(t^{*}_{n}) \leq e^{Z(M^{2}+\mathcal{K})} \epsilon$,
see (\ref{Calca}), and we use the estimates we have just proved for its right hand side.

We take $\mathcal{K}$ and $\epsilon$ such that
\begin{equation}\label{SmallAssIPSF}
 \mathcal{K} \leq \frac{1}{6Z},
\qquad
e^{Z(M^{2} + \mathcal{K})} \epsilon \leq \frac{1}{30Z},
\end{equation}
thus
\begin{equation}\label{eq:Int3SF}
  \textstyle
  a_{\mu}(t)
  +(\frac{1}{10Z} 
  -2|\xi|^{2}\frac{T}{M})
  e^{6Z B_{t^{*}_{n},\mu}(t)}
  \le
  \frac{1}{6Z}.
\end{equation}
We furthermore assume that
\begin{equation}\label{eq:condxi2SF}
  \textstyle
  (\frac{1}{10Z} 
    -2|\xi|^{2}\frac{T}{M})>0
  \quad \mbox{namely},
   \quad 
  |\xi|^{2}T< \frac{M}{20Z},
\end{equation}
which is stronger than~\eqref{ResrtCondFirstSF}, {\it i.e.}
$|\xi|^{2} t^{*} \le M^{2}$, since $t^{*} \le T$ and $M,Z > 1$.  
In this way, the inequality (\ref{eq:Int3SF}) immediately implies 
$$
B_{t^{*}\!\!, \mu}(t) \leq B_{t^{*}_{n}, \mu}(t)  < \infty,
\quad
\mbox{uniformly in $\mu >0$} \, , 
$$
for all $t^{*} \leq t \leq t^{*} +T$, as long as~(\ref{eq:condxi2SF}) holds.
On the other hand, we already know that  
$B_{\mu}(t^{*}) \leq 2M^{2}$, see (\ref{BasicPropertiesSF}), 
so that, since $B_{\mu}(t) = B_{\mu}(t^{*}) + B_{t^{*}\!\!,\mu} (t)$, we have proved
\begin{equation}\label{PTTL}
B_{\mu}(t) := \int_{0}^{t} \int_{\RT} \sigma_{\mu}(y)
  |\nabla v_{\xi}(\tau ,y)|^{2}  \ dy d \tau  < \infty,  \quad \text{uniformly in $\mu >0$} \, ,
\end{equation} 
for all $0 \leq t \leq t^{*} +T$,
as long as~(\ref{eq:condxi2SF}) holds. Again, we mean that~$B_{\mu}(t)$ is bounded by 
a constant that is independent on $\mu$. This constant may however it may increase indefinitely as 
as~$|\xi|^{2}T$ approaches $\frac{M}{20 Z}$, but this does 
not affect our argument.  
Since the weights~$\sigma_{\mu}(y)$ are increasing as~$\mu \to 0$, and they converges to~$|y|^{-1}$, we can pass to 
the limit in~\eqref{PTTL},
so that
\begin{equation}\nonumber
B(t) := \int_{0}^{t} \int_{\RT} |y|^{-1}
           |\nabla v_{\xi}(\tau ,y)|^{2}  \ dy d \tau 
            < \infty  \, , 
\end{equation} 
under the same conditions.
In particular, going back to the old variables,
\begin{equation}\nonumber
B(T) = \int_{0}^{T} \int_{\Rtre} \frac{| \nabla v(\tau,x)|^{2}}{|x - s \xi |} \ dx d \tau  < \infty \, ,
\end{equation}
and we have already observed (see Subsection~\ref{SameArgument}) that this implies that
$L(T,\xi)$
is a regular set. 

Summing up, if we assume (\ref{SmallAssIPSF}) and (\ref{eq:condxi2SF}), then
$L(T,\xi)$ is regular. Notice that the condition (\ref{SmallAssIPSF}) is ensured 
by
$$
\mathcal{K} \leq \delta_{4},
\quad
\epsilon \leq \delta_{4} e^{- M^{2} / \delta_{4}} \, ,
$$
once we choose $\delta_{4} = 1/(90Z^{2})$. 
Under this choice, the condition (\ref{eq:condxi2SF}) is implied by
\begin{equation}\nonumber
  \textstyle
  (T,T \xi)\in\Pi_{M\delta_{4}},
  \qquad
  \Pi_{M\delta_{4}}:=
  \{(t,x)\in (0,\infty)\times \mathbb{R}^{3}
  \colon t > \frac{|x|^{2}}{M\delta_{4}}\} \, ,
\end{equation}
and the proof is completed because 
$\Pi_{M\delta_{4}}$
is the union of such segments for arbitrary $T > 1$.

\end{proof}

\subsection{Proof of Theorem \ref{thm:Gap}}

We first recall (a partial version of) the Kato~$L^{3}$ theorem, which is needed in order to prove Theorem \ref{thm:Gap}.  

\begin{theorem}[\cite{Kato}]\label{KatoThm}
There exists a constant $\varepsilon_{1} > 0$ such that the following holds.
  If 
  $w_{0} \in L^{2} \cap L^{3}$ is a divergence free vector field
  such that
  \begin{equation}\label{SmallAssInGlobalThm}
    \| w_{0} \|_{L^{3}}
    < \varepsilon_{1} \, ,
  \end{equation}
  then there exists a 
  global unique smooth solution $w \in C_{b}([0,\infty); L^{2}_{x}) \cap L^{2}_{t}\dot{H}^{1}_{x}$ 
  to the Navier--Stokes equation with data $w_{0}$, 
  and 
  \begin{equation}\label{NotInfty} 
    \| w \|_{L^{5}_{t}L^{5}_{x}}
    \leq C
     \| w_{0} \|_{L^{3}}.
  \end{equation}
\end{theorem}

\begin{remark}\label{EndRemark}
If we also have
$w_{0} \in L^{2}(|x-x'|^{-1}dx)$, for some $x' \in \Rtre$, then $w$ is a reference solution of 
size~$\mathcal{K} = \| w \|^{5}_{L^{5}_{t}L^{5}_{x}}$; see Definition \ref{Def:RefSol}.
\end{remark}

Theorem~\ref{thm:Gap} covers the gap between the regularity Theorem~\ref{CKNSmallData} and the full regularity of 
solutions with small $L^{3}$ initial data. A similar phenomenon has been observed in~\cite{DL2} where
some additional angular integrability to the data has been required in order to gain regularity. The connection 
between smoothness and higher angular integrability of solutions to the Navier--Stokes equation 
has also been observed in \cite{Ren}. This is not surprising since the basic inequalities that 
are used to handle the local energy estimates, like weighted bound for the Riesz transform and the 
Caffarelli--Kohn--Nirenberg inequality~\eqref{CKNInequalityCKNVersion}, improves under additional angular integrability assumptions; 
see \cite{CLuc}, \cite{DL}, \cite{DenapoliDrelichmanDuran10-a}.

We prove Theorem~\ref{thm:Gap} when $\bar{x} = 0$, the general case follows by translation.
In order to apply Theorem \ref{Theorem:DL}, an
appropriate decomposition of the initial data is required.

For any $s >0$, we split
\begin{equation}\nonumber
u_{0}= \mathbb{P} u_{0, \leq s} + \mathbb{P} u_{0, > s} =: w_{0} + v_{0},
\end{equation}
where 
\begin{equation}\nonumber
u_{0, \leq s} (x) := 
\left\{
\begin{array}{lll}
u(x)  & \mbox{if} & |x| |u(x)| \leq s ,
\\
0  &  \mbox{otherwise} .  &
\end{array}
\right.
\end{equation}
Let us recall that $\mathbb{P}$ is the projection onto the divergence free vector fields subspace, and
it can be represented as $\mathbb{P} = Id + (\mathcal{R} \otimes \mathcal{R}) \cdot$. 
This decomposition satisfies the key estimates 
\begin{equation}\label{eq:decompSF}
  \begin{array}{lcl}
  \| w_{0}\|_{L^{3}(\Rtre)} 
  &\leq & C s^{1 - p/3} 
  \| |x|^{\alpha} u_{0}  \|^{ p/3}_{L^{p}(\Rtre)}
  \\
  && 
  \\
  \| |x|^{-1/2} v_{0} \|_{L^{2}(\Rtre)}
  &\leq & C s^{1-p/2} 
    \| |x|^{\alpha} u_{0}  \|^{p/2}_{L^{p}(\Rtre)}  \, ,
  \end{array}
\end{equation}
where $2<p<3$ and $\alpha = 1 - 3/p$.
Indeed, since the Riesz transform $\mathcal{R}$, and so $\mathbb{P}$, are bounded on~$L^{3}$ 
and on $L^{2}(|x|^{-1}dx)$,   
see for instance~\cite{Stein}, the estimates~\eqref{eq:decompSF} are consequence of   
the elementary estimates
\begin{equation}\label{eq:decompSFSimple}
  \begin{array}{lcl}
  \|   u_{0, \leq s}  \|_{L^{3}(\Rtre)} 
  &\leq & C s^{1 - p/3} 
  \| |x|^{\alpha} u_{0}  \|^{ p/3}_{L^{p}(\Rtre)}
  \\
  && 
  \\
  \| |x|^{-1/2}  u_{0, > s} \|_{L^{2}(\Rtre)}
  &\leq & C s^{1-p/2} 
    \| |x|^{\alpha} u_{0}  \|^{p/2}_{L^{p}(\Rtre)}  \, .
  \end{array}
\end{equation}
We choose
\begin{equation*}
  s = \frac{p-2}{3-p}
\end{equation*}
so that
\begin{equation}\label{eq:decompSFTheta}
  \begin{array}{lcl}
  \| w_{0}\|_{L^{3}(\Rtre)} 
  &\leq & C \theta_{1}(p) 
  \| |x|^{\alpha} u_{0}  \|_{L^{p}(\Rtre)}^{ p/3}
  \\
  && 
  \\
  \| |x|^{-1/2} v_{0} \|_{L^{2}(\Rtre)}
  &\leq & C \theta_{2}(p) 
    \| |x|^{\alpha} u_{0}  \|_{L^{p}(\Rtre)}^{p/2}  \, ,
  \end{array}
\end{equation}
for some $C > 1$, where
$ \theta_{1}(p), \theta_2(p) $ have been defined in (\ref{def:theta12}).

Then we choose $\delta_{1} = \mathcal{C}^{-1} \min (1, \delta_{4}, \varepsilon_1)$, where $\delta_{4}$ is the small constant 
in Theorem \ref{Theorem:DL}, $\varepsilon_1$ is the small constant in the Kato's Theorem \ref{KatoThm}
and $\mathcal{C}^{1/10}$ is a constant larger than the ones in (\ref{eq:decompSFTheta}), (\ref{NotInfty}). 
By (\ref{eq:decompSFTheta}) and the first assumption in (\ref{ILSAGAP}),
\begin{equation}\label{TT1SF}
\| w_{0} \|_{L^{3}}  \leq  \mathcal{C}^{1/10} \theta_1(p) \| |x|^{\alpha} u_{0}  \|_{L^{p}(\Rtre)}^{p/3} \leq \mathcal{C}^{1/10} \delta_{1} \, .  
\end{equation}
By (\ref{NotInfty}) and (\ref{TT1SF}),
\begin{equation}\label{TT2SF}
\| w \|^{5}_{L^{5}_{t}L^{5}_{x}}  \leq  \mathcal{C}^{1/2} \| w_{0} \|^{5}_{L^{3}}   \leq  \mathcal{C} \delta_{1}^{5}  \leq  \mathcal{C} \delta_{1}  \leq  \delta_{1} \, .
\end{equation}
By (\ref{eq:decompSFTheta}) and the second assumption in (\ref{ILSAGAP}), 
\begin{equation}\label{TT3SF}
\| |x|^{-1/2} v_{0} \|_{L^{2}}   \leq   \mathcal{C}^{1/10} \theta_2(p) \| |x|^{\alpha} u_{0}  \|_{L^{p}(\Rtre)}^{p/2} 
\leq \mathcal{C}^{1/10} \delta_{1} e^{- M^{2}/\delta_{1}} \leq \delta_{4} e^{- M^{2}/\delta_{4}} \, .
\end{equation}
Thanks to (\ref{TT2SF}), (\ref{TT3SF}) and recalling Remark \ref{EndRemark}\footnote{
Notice that $w_{0} \in L^{2}(|x-x'|^{-1}dx)$. Indeed we have assumed $u_{0} \in L^{2}(|x-x'|^{-1}dx)$ and we 
can estimate~$\||x-x'|^{-1/2} w_{0} \|_{L^{2}} = \| |x-x'|^{-1/2} \mathbb{P}u_{0, \leq s} \|_{L^{2}} \leq \| |x-x'|^{-1/2} u_{0} \|_{L^{2}} $.   }
we can apply Theorem \ref{Theorem:DL}
to conclude that $\Pi_{M\delta_{4}}$ is a regular set for $u$. 
Since $\Pi_{M\delta_{1}} \subset \Pi_{M\delta_{4}}$ the proof is complete.

\section{Appendix}\label{sec:ES}

\begin{lemma}\label{Lemma:PertGenEnIneq}
If 
$u \in L^{\infty}_{t}L^{2}_{x} \cap L^{2}_{t}\dot{H}^{1}_{x}$ is a suitable weak solution to the Navier--Stokes equation 
with pressure $P_{u}$ and 
divergence free data $u_0 \in L^{2}$
and $w$ is a reference solution (see Definition~\ref{Def:RefSol}) with pressure $P_{w}$ and divergence free data~$w_0 \in L^{2}$, 
then
the difference~$u-w$ is a suitable weak solution to the 
perturbed Navier--Stokes equation~\eqref{ModProbInPrelimin}, around the solution~$w$, with pressure~$P_{u-w}:=P_{u} - P_{w}$
and data~$u_{0} - w_{0}$.
\end{lemma}

\begin{proof}
The only non trivial thing to prove is that the perturbed local energy inequality~\eqref{ModSolDef} is satisfied.
This is formally justified once we take the scalar product of the momentum equation for $v$, namely the first 
equation in~\eqref{ModProbInPrelimin}, against the vector field~$2 \phi v$, we integrate over
$(t_{0},t) \times \Rtre$ and then by parts. 
A rigorous proof requires the suitability of $u$ and the fact that $w$ is a reference solution.

We first consider the case~$t_{0} =0$.  
Since $(u,P_{u})$ and $(w,P_{w})$ satisfy the local energy inequality~\eqref{GenEnIneq}, taking the difference, is straightforward to 
check that~\eqref{ModSolDef} 
is a consequence of 
\begin{align} \nonumber 
  0 & =  
  2
\textstyle
  \int_{\mathbb{R}^{3}}  
  u \cdot w \phi(t,x) dx  
 - 2
  \int_{\mathbb{R}^{3}}  u_{0} \cdot w_{0} \phi(0,x) dx 
- 2
  \int_{0}^{t} \int_{\mathbb{R}^{3}} 
  u \cdot w \phi_{t} 
  \\ \nonumber 
  & - 
 \textstyle
 2
   \int_{0}^{t}\int_{\mathbb{R}^{3}} 
  (u \! \cdot \! w) w \! \cdot \! \nabla \! \phi
   + 2 \int_{0}^{t}\int_{\mathbb{R}^{3}}  (u \! \cdot \! \nabla  ) u  \! \cdot \! w \phi
 -  2 \int_{0}^{t}\int_{\mathbb{R}^{3}}  (w \! \cdot \! \nabla ) u \!  \cdot \! w \phi
  \\ \nonumber 
 &  + 
 \textstyle
  4 \int_{0}^{t}\int_{\mathbb{R}^{3}}
   \phi \nabla u \cdot \nabla w  
  - 2 \int_{0}^{t} \int_{\mathbb{R}^{3}} 
  u \cdot w \Delta \phi
\\ \nonumber 
  & -  
 \textstyle
  2  \int_{0}^{t} \int_{\mathbb{R}^{3}}   
  P_{w} u  \cdot \nabla \phi 
   -
  2 \int_{0}^{t} \int_{\mathbb{R}^{3}}   P_{u} w \cdot \nabla \phi   \, .
\end{align}
Again, this identity is formally justified once we take the scalar product of the momentum equation
for $u$ against the vector field $2 \phi w$, the scalar product of the momentum equation
for $w$ against $2 \phi u$, we integrate over $(0,t) \times \Rtre$ and then by parts.
Since $u$ is a suitable solution, this procedure is rigorous once we consider, rather than the field
$2 \phi u$, the mollified field~$2 \phi u^{\varepsilon}$, letting eventually $\varepsilon \to 0$. 
The mollification has to be in the space-time variables, 
namely $f^{\varepsilon}(t,x):=  \int_{\mathbb{R}^{4}} f(s,y)  \rho^{\varepsilon}(t-s, x-y) dyds$,
where $\rho^{\varepsilon}:= \varepsilon^{-4} \rho(\varepsilon^{-4}\cdot)$, $0 \leq \rho \in C^{\infty}_{c}(\mathbb{R}^{4})$ 
and $\int_{\mathbb{R}^{4}} \rho =1$. 
The limit $\varepsilon \to 0$ can be easily justified recalling that $u \in L^{\infty}_{t}L^{2}_{x} \cap L^{2}_{t}\dot{H}^{1}_{x}$, 
$w \in C_{b}([0,\infty); L^{2}_{x}) \cap L^{2}_{t}\dot{H}^{1}_{x} \cap L^{r}_{t}L^{q}_{x}$, with $(r,q)$ admissible, and
$P_{u}, P_{w} \in L^{3/2}_{loc}((0,\infty) \times \Rtre)$. 
Notice that the function $t \mapsto \int_{\Rtre} u(t,x) \cdot w (t,x) \phi(t,x) dx$, that appears in the first line of the identity,
is continuous, at any $t \geq 0$, since $t \in [0,\infty) \to u(t)$ is $L^{2}$-weakly continuous and~$w$ belongs to $C_{b}([0,\infty); L^{2}_{x})$; see 
Section~\ref{sec:prelim} and
Remark~\ref{Rem:Uniqueness}. Notice also that the integrability of the term
$(u \! \cdot \! \nabla  ) u  \! \cdot \! w \phi$, which is the hardest to be handled, 
is guaranteed by $u \in L^{\bar{r}}_{t}L^{\bar{q}}_{x} \cap L^{2}_{t} \dot{H}^{1}_{x}$, $2/\bar{r} + 3/ \bar{q} = 3/2$, $2\leq \bar{q} \leq 6$
(see \eqref{UIP}) and $w \in L^{r}_{t}L^{q}_{x}$ with $(r,q)$ admissible, that is $2/\bar{r} + 3/ \bar{q} = 1$, $3 < q \leq \infty$.

Once we have proved \eqref{eq:StrongPertGenInProofOLD} for $t_{0}=0$, the inequality for almost any $t_{0 }>0$ can be 
deduced 
in the following 
way.
For any $\varepsilon > 0$, we consider the auxiliary test functions $\phi^{\varepsilon} (t,x) :=  \eta^{\varepsilon}(t) \phi(t,x)$
where $\eta^{\varepsilon}(t)$ is an $\varepsilon$-mollification of the step function with jump in $t_0$,
namely $\eta^{\varepsilon}(t) := \chi_{[t_{0}, \infty)} * \rho^{\varepsilon}$, 
and $\rho^{\varepsilon} := \varepsilon^{-1}\rho(\varepsilon^{-1}t)$, with~$0 \leq \rho \in C^{\infty}_{c} (\mathbb{R})$ and~$\int_{\mathbb{R}} \rho =1$.
We have 
$\partial_{t} \phi^{\varepsilon}(t,x)
=  \rho^{\varepsilon}(t-t_0) \phi(t,x) +  \eta^{\varepsilon}(t)\partial_{t}\phi(t,x)$,
and, as~$\varepsilon \to 0$: 
$$ 
\eta^{\varepsilon} \to \chi_{[t_0,\infty]}, \ \
\phi^{\varepsilon} \to \chi_{[t_0,\infty]} \phi, \ \
\nabla \phi^{\varepsilon} \to \chi_{[t_0,\infty]} \nabla \phi, \ \
\Delta \phi^{\varepsilon} \to \chi_{[t_0,\infty]} \Delta \phi \, .
$$

We now apply the inequality (\ref{eq:StrongPertGenInProofOLD}), 
proved in the case $t_{0}=0$, with the test function $\phi^{\varepsilon}$. Taking 
the limit $\varepsilon \rightarrow 0$, 
we get 
\begin{eqnarray}
     \textstyle
     \int_{\Rtre} |v|^{2} \phi (t) 
&+& 
  \textstyle
  2 \int_{t_0}^{t}\int_{\Rtre} | \nabla v |^{2} \phi
    \leq  
   \liminf_{\varepsilon \rightarrow 0}  \int_{0}^{t}    \rho^{\varepsilon}(s-t_0)   \left(\int_{\Rtre} |v|^{2}\phi\right)(s)  ds
    \\ \nonumber
& + & 
\textstyle
\int_{t_0}^{t} \int_{\Rtre} |v|^{2} 
    (\phi_{t} + \Delta \phi) 
    +\int_{t_0}^{t}\int_{\Rtre} (|v|^{2} + 2P_{v})v \cdot \nabla \phi
 \\ \nonumber
 & + &
 \textstyle
 \int_{t_0}^{t}\int_{\mathbb{R}^{3}} 
  |v|^{2} (w \cdot \nabla \phi) 
  +
   2 \int_{t_0}^{t}\int_{\mathbb{R}^{3}} 
  (v \cdot w) (v \cdot \nabla \phi)  
  + \phi (v \cdot \nabla ) v \cdot w.
\end{eqnarray}
Here we have used the integrability properties of $v, P_{v}, w$ that we have recalled above.
Thus, the required inequality is satisfied for any $t_{0}$ such that
\begin{equation}\nonumber
\lim_{\varepsilon \rightarrow 0}  \int_{0}^{t}    \rho^{\varepsilon}(s-t_0)   \left(\int_{\Rtre} |v|^{2}\phi\right)(s)  ds
= 
\int_{\Rtre} |v|^{2} \phi(t_0),
\end{equation}
that is actually true for all the Lebesgue points 
of the function $t \rightarrow \int_{\Rtre} |v|^{2} \phi (t)$, and so for almost every $t_{0} >0$, since $v \in L^{\infty}_{t}L^{2}_{x}$. This concludes 
the proof.

\end{proof}

\begin{lemma}\label{LasLem}
Let $w \in L^{2}_{t}L^{\infty}_{x}$ be a generalized reference solution to the Navier--Stokes 
equation (see Definition~\ref{Def:GenRefSol}) with divergence free data $w_{0} \in L^{2}_{loc}$.
If
$v$ is a suitable weak solution to the perturbed equation~\eqref{ModProbInPrelimin}, around $w$, 
with divergence free data $v_{0} \in L^{2}(\Rtre)$, then $u := v + w$ is a suitable weak solution to the Navier--Stokes equation
with data~$u_{0} := v_{0} + w_{0}$ and pressure $P_{u} := P_{v} + P_{w}$.
\end{lemma}

Recalling that $(2,\infty)$ is an admissible couple, the proof is analogous to that of Lemma~\ref{Lemma:PertGenEnIneq}.
Here $v$, which plays the role played by $u$ in the previous lemma, enjoys the same integrability property, 
while~$w$ belongs to $L^{\infty}_{t}L^{2}(K) \cap L^{2}_{t}\dot{H}^{1}(K) \cap L^{2}_{t}L^{\infty}_{x}$,
$K \subset \Rtre$ compact, and 
$t \in [0,\infty) \to w(t) \in L^{2}(K)$, $K \subset \Rtre$ compact, is 
continuous. However, since in the proof all the vector fields are multiplied 
against a compactly supported test function, this does not affect the argument.

\begin{proposition}[Caffarelli--Kohn--Nirenberg 
  \cite{CaffarelliKohnNirenberg84-a}]
  Assume that 
  \begin{enumerate}
    \item $r > 0$, $0 < \theta \leq 1$, 
    $\gamma < 3/r$, $\alpha < 3/2$, $\beta < 3/2$;
    \item $-\gamma + 3/r = \theta (-\alpha + 1/2) + (1-\theta)(-\beta + 3/2)$;
    \item $ \theta \alpha + (1-\theta)\beta \leq \gamma   $;
    \item when $- \gamma + 3/r = -\alpha + 1/2$,
    assume also $\gamma \leq \theta (\alpha +1) + (1-\theta)\beta$.
  \end{enumerate}
  Then
  \begin{equation}\label{CKNInequalityCKNVersion}
    \| \sigma_{\mu}^{\gamma} f \|_{L^{r}(\mathbb{R}^{3})} \leq C
    \| \sigma_{\mu}^{\alpha} \nabla f \|^{\theta}_{L^{2}(\mathbb{R}^{3})} 
    \| \sigma_{\mu}^{\beta} f \|^{1-\theta}_{L^{2}(\mathbb{R}^{3})}, 
  \end{equation}
  where $\sigma_{\mu} := (\mu + |x|^{2})^{-1/2}$, 
  $\mu \geq 0$. The constant $C$ is independent of $\mu$.
\end{proposition}

The (\ref{CKNInequalityCKNVersion}) has been proved in \cite{CaffarelliKohnNirenberg84-a}
in the case $\mu = 0$. 
The general case can be obtained by the following standard argument.
First notice that, by rescaling, it suffices to prove (\ref{CKNInequalityCKNVersion}) in the case $\mu =1$. Then we notice that
\begin{equation}\nonumber
\sigma_{1} \simeq 1 \quad \text{if} \quad |x| \lesssim 1,
\qquad
\sigma_{1} \simeq |x|^{-1} \quad \text{if} \quad |x| \gtrsim 1 .
\end{equation}
Split $f = f_{1} + f_{2} := \chi_1 f + \chi_2  f $, where
$\chi_1$ is a smooth cut-off function of the unit ball, namely 
$0 \leq \chi_1 \leq 1$, $\chi_1=1$ if $|x|\leq 1$, $\chi_1 = 0$ if $|x| \geq 2$
($\chi_2 := 1-\chi_1$). 
Using the inequality \eqref{CKNInequalityCKNVersion} with $\mu =0$ we have
\begin{equation}\label{CKNInequalityCKNVersionFar0}
 \| \sigma_{1}^{\gamma} f_{2} \|_{L^{r}(\mathbb{R}^{3})} 
 \lesssim
    \| \sigma_{1}^{\alpha} \nabla f_{2} \|^{\theta}_{L^{2}(\mathbb{R}^{3})} 
    \| \sigma_{1}^{\beta} f_{2} \|^{1-\theta}_{L^{2}(\mathbb{R}^{3})} \, ,
\end{equation}
where we have used that $\sigma_{1} \simeq |x|^{-1}$ in the support of $f_{2}$.

If the parameters $(\theta,r,\alpha,\beta,\gamma)$ satisfies the conditions (1--4), 
then 
 $(\theta,r,0,0,0)$
satisfies the same conditions once we replace (2) with 
$$
-\gamma + 3/r \geq \theta (-\alpha + 1/2) + (1-\theta)(-\beta + 3/2) \, .
$$
Thus, there exists $s \geq r$ so that $(\theta,s,0,0,0)$ satisfies (1--4), and we can 
use again the inequality~\eqref{CKNInequalityCKNVersion}
with $\mu =0$, so that
\begin{equation}
    \|  f_{1} \|_{L^{r}(\mathbb{R}^{3})} 
    \lesssim
    \|  f_{1} \|_{L^{s}(\mathbb{R}^{3})} 
    \lesssim
    \|  \nabla f_{1} \|^{\theta}_{L^{2}(\mathbb{R}^{3})} 
    \| f_{1} \|^{1-\theta}_{L^{2}(\mathbb{R}^{3})} \, ,
  \end{equation}
  recalling that $f_{1}$ is compactly supported.
This is actually the Gagliardo--Nirenberg inequality. 
Since $\sigma_{1} \simeq 1$ in the support of $f_{1}$ 
this implies
\begin{equation}\label{CKNInequalityCKNVersionClose0}
    \| \sigma_{1}^{\gamma} f_{1} \|_{L^{r}(\mathbb{R}^{3})} 
    \lesssim
    \| \sigma_{1}^{\alpha} \nabla f_{1} \|^{\theta}_{L^{2}(\mathbb{R}^{3})} 
    \| \sigma_{1}^{\beta} f_{1} \|^{1-\theta}_{L^{2}(\mathbb{R}^{3})} \, .
  \end{equation} 
We finally observe that, for $j=1,2$,
\begin{align}\label{EM}
\| \sigma_{1}^{\alpha} \nabla f_{j}\|_{L^{2}(\Rtre)} 
& 
\leq 
\| \sigma_{1}^{\alpha}  f \nabla \chi_j\|_{L^{2}(1 \leq |x| \leq 2)} +
\| \sigma_{1}^{\alpha} \chi_j \nabla f  \|_{L^{2}(\Rtre)}
\\ \nonumber
& \lesssim
\| \sigma_{1}^{\alpha}  f \|_{L^{6}(1 \leq |x| \leq 2)} +
\| \sigma_{1}^{\alpha} \nabla f  \|_{L^{2}(\Rtre)}
\lesssim
\| \sigma_{1}^{\alpha} \nabla f\|_{L^{2}(\Rtre)},
\end{align}
where in the last inequality we have used the embedding $\dot{H}^{1} \hookrightarrow L^{6}$.
Then \eqref{CKNInequalityCKNVersion} follows by 
(\ref{CKNInequalityCKNVersionFar0}), 
(\ref{CKNInequalityCKNVersionClose0}), (\ref{EM}). 

The family 
of inequalities~\eqref{SteinIneq} has been proved in \cite{Stein} in the case~$\mu =0$. The general case 
can be then deduced as shown in~\cite[Lemma 7.2]{CKN}.

\begin{proposition}[Stein]
Let $1 < p < \infty$ and 
$ - 3 + 3/p < \alpha < 3/p $. Let $T$ be a singular operator. Then 
\begin{equation}\label{SteinIneq}
  \|\sigma_{\mu}^{\alpha} T f \|_{L^{p}} 
  \leq C \|\sigma_{\mu}^{\alpha} f \|_{L^{p}}, 
  \end{equation}
where $\sigma_{\mu} := (\mu + |x|^{2})^{-1/2}$, 
  $\mu \geq 0$. The constant $C$ is independent of $\mu$.
\end{proposition}

\section{Acknowledgments}

The authors would like to thank Jean-Yves Chemin for suggesting to use the solutions of~\cite{ChemGall}
as examples of reference solutions. 
Renato Luc\`{a} is supported by the ERC Starting Grant 676675 FLIRT.

\end{document}